\documentclass[12pt]{amsart}
\usepackage[top=1.2in, left=1.2in, right=1.2in, bottom=1.2in]{geometry}
\usepackage{lmodern}
\usepackage[T1]{fontenc}
\usepackage{microtype}

\usepackage{ucs}
\usepackage[utf8x]{inputenc}

\usepackage{amssymb,amsfonts,amsmath,mathrsfs,color,amsthm,latexsym}
\usepackage[all]{xy}
\usepackage[colorlinks=true]{hyperref}
\usepackage{graphicx}
\usepackage{rotating}
\usepackage{array}
\usepackage{float}
\usepackage{xspace}
\usepackage[english]{babel}
\usepackage{units}
\usepackage{framed}
\usepackage{stmaryrd}

\newcommand{\bfi}{\bfseries\itshape}

\newcommand{\cok}{\kappa}

\newcommand{\cA}{\mathcal A}
\newcommand{\cB}{\mathcal B}
\newcommand{\cC}{\mathcal C}
\newcommand{\cD}{\mathcal D}
\newcommand{\cF}{\mathcal F}

\newcommand{\cH}{\mathcal H}

\newcommand{\cL}{\mathcal L}

\newcommand{\cO}{\mathcal O}

\newcommand{\cS}{\mathcal S}
\newcommand{\cU}{\mathcal U}
\newcommand{\cV}{\mathcal V}

\newcommand{\RR}{\mathbb R}

\newcommand{\bbR}{\mathbb R}
\newcommand{\bbT}{\mathbb T}

\newcommand{\rk}{{\rm rank\ }}

\newcommand{\QED}{\hfill $\square$\vspace{2mm}}

\newtheorem{theorem}{Theorem}[section]
\newtheorem{definition}[theorem]{Definition}
\newtheorem{lemma}[theorem]{Lemma}
\newtheorem{remark}[theorem]{Remark}
\newtheorem{proposition}[theorem]{Proposition}

\begin{document}
\setcounter{tocdepth}{1}

\title[Convexity of toric-focus integrable systems]{Convexity of singular
affine structures and toric-focus integrable Hamiltonian systems}

\author{Tudor Ratiu}
\address{School of Mathematics, Shanghai Jiao Tong
University, 800 Dongchuan Road, Minhang District, Shanghai, 200240 China
and Section de Mathématiques, Ecole Polytechnique Fédérale de Lausanne,
CH-1015 Lausanne, Switzerland. Partially supported by the  National Natural
Science Foundation of China grant number 11871334  and by NCCR SwissMAP
grant of the Swiss National Science Foundation. ratiu@sjtu.edu.cn,
tudor.ratiu@epfl.ch}
\author{Christophe Wacheux}
\address{IBS Center for Geometry and Physics, Math Building 108, POSTECH campus, Pohang, Gyeonsangbukdo, South Korea. cwacheux@ibs.re.kr}

\author{Nguyen Tien Zung}
\address{School of Mathematics, Shanghai Jiao Tong University (visiting professor),
800 Dongchuan Road, Minhang District, Shanghai, 200240 China
and Institut de Mathématiques de Toulouse, UMR5219,
Université Paul Sabatier, 118 route de Narbonne,
31062 Toulouse, France, tienzung@math.univ-toulouse.fr}

\begin{abstract}
This work is devoted to a systematic study of symplectic convexity for
integrable Hamiltonian systems with elliptic and focus-focus singularities. A
distinctive feature of these systems is that their base spaces are still smooth
manifolds (with boundary and corners), similarly to the toric case, but their
associated integral affine structures  are singular, with non-trivial monodromy,
due to focus singularities. We obtain a series of convexity results, both
positive and negative, for such singular integral affine base spaces. In
particular, near a focus singular point, they are locally convex and the
local-global convexity principle still applies. They are also globally convex
under some natural additional conditions. However, when the monodromy is
sufficiently big then the local-global convexity principle breaks down, and the
base spaces can be globally non-convex even for compact manifolds. As one of
surprising examples, we construct a 2-dimensional ``integral affine black
hole'', which is locally convex but for which a straight ray from the center can
never escape. \end{abstract}

\subjclass{37J35,  53A15, 52A01}

\keywords{integrable system, toric-focus, convexity,
integral affine structure, singularities}

\maketitle

\tableofcontents

\section{Introduction}
label{sec:introduction}
This paper is devoted to the study of convexity properties of integrable Hamiltonian systems in the
presence of focus-focus singularities. These singularities are one of the three kinds of
elementary non-degenerate singularities for integrable systems (the other two being elliptic and
hyperbolic). They are of special interest in mathematics and physics for at least the following two reasons:

- They can be found everywhere, even in the most simple physical systems, e.g.,
the spherical pendulum, the symmetric spinning top, the internal movement of
molecules, the nonlinear
Schr\"odinger equation, the Jaynes-Cummings-Gaudin model, and so on.
(See, e.g., \cite{BaDo2015,BoFo-IntegrableBook2004, CuBa2015,
DGSZ-DynamicalMonodromy2009, EJS-BendingQuantum2004,
Zhilinskii-Deffect2006} and references therein.)
They also appear in the theory of special Lagrangian fibrations on Calabi-Yau
manifolds related to mirror symmetry. (See, e.g.,
\cite{CaMa_Lagrangian2009,Gross_TropicalBook2011,
GrSi-AffineComplex2011,KoSo-Affine2006} and references therein.)

- They give rise to the phenomenon of non-trivial monodromy, both classical and quantum,
for both the Lagrangian fibrations on symplectic manifolds and the singular
integral affine  structure on the base spaces. The notion of monodromy was
introduced by Duistermaat~\cite{Duistermaat-AA1980}. The phenomenon of
non-trivial monodromy for focus-focus singularities was first observed in
concrete examples by Cushman and
Kn\"{o}rrer~\cite{CuKno-EnergymomentumLagrTop1985}. The general formula was
obtained independently by Zung~\cite{Zung-Focus1997} and
Matveev~\cite{Matveev1996}. This monodromy is an obstruction to the existence of
global action-angle variables, and is also an obstruction to the existence of
global quantum numbers in a very large class of integrable quantum systems,
including such simple systems as the $H_2O$ molecule.

For our study of convexity, this monodromy also creates a challenge and leads to
some very surprising results.

We restrict our attention to toric-focus integrable Hamiltonian  systems, i.e.,
systems which admit only non-degenerate elliptic and focus-focus singularities
and no hyperbolic singularity. The reason is that the base space (i.e., the
space of connected components of the level sets of the momentum map) branches
out at hyperbolic singular points, so it does not make much sense to talk about
convexity at those points. On the other hand, if the system is toric-focus, then
the base space is still a smooth manifold, together with an associated
(singular) integral
affine structure, and one can talk about the convexity of that structure.

When there are no focus-focus singularities, then the system is toric (i.e, it
admits a Hamiltonian torus action of half the dimension), and its convexity
properties are well-known, due to the famous Atiyah--Guillemin--Sternberg
theorem \cite{At1982,GuSt-Convexity1982} and the Delzant classification theorem
\cite{Delzant1988}. For a subclass of toric-focus systems in dimension 4 called
semitoric systems, V\~u Ng\d{o}c \cite{San-Polytope2007} developed a theory of
multi-valued momentum polytopes which are similar to Delzant polytopes.
Nevertheless, the problem of intrinsic convexity of semitoric systems has never
been studied anywhere in the literature, as far as we know.

Let us now present the main concepts and results of this paper.

In order to study convexity of singular affine structures, we shall first define
the class of singular affine structures that we will study, namely, the \emph{affine structure with focus singularities}. Then, we define clearly what
convexity means for them. So we introduce the notion of straight lines,
especially those passing through singular points, and study their existence and
their multiple branchings. Then we say that a set in a singular affine space is
convex if any two points can be joined by at least one (and maybe several)
straight line segment in that set. We also need 
some extensions of the notion of convexity for singular affine structures: local and global convexity, strong convexity, $\Sigma$-convexity, and
a notion of convex hull (which exists but may be non-unique). We also 
make use of the notion of multiple-valued
integral affine charts around focus singular points in our study.

The key technical assumption in almost all theorems is either compactness or a
properness condition for non-compact spaces. As expected from the theory of
non-linear Lie group actions and momentum maps, without some properness hypothesis convexity may fail.  The so-called local-global convexity principle remains one
of the main technical tools which links local to global convexity; it is
presented in detail, in the framework needed in this paper, in Section
\ref{section:LocalConvexity}, even in the presence of focus singularities.

The paper presents two kinds of convexity results, positive and negative.

\subsection*{Positive convexity results} \hfill

$\bullet$ (Theorem \ref{thm:FF1Local}, Theorem \ref{thm:FF2Local},
Proposition \ref{prop:FocusBox1})
\textit{Local convexity near focus points}:
Any box around a simple focus singularity is convex and, moreover, the
local-global convexity principle holds in such a box.

$\bullet$ (Theorem \ref{thm_gobal_convex_2D_compact},
Theorem \ref{thm:GlobalConvex2Dfirst}, Theorem \ref{thm:ConvexS2})
\textit{Global convexity in dimension 2}:
Let $\mathcal{B}$ be the 2-dimensional base space of a toric-focus
integrable Hamiltonian system on a connected
4-dimensional symplectic manifold (with or without boundary).
Assume that the base $\cB$ is compact or satisfies some properness
conditions. If $\cB$ is non-empty and locally convex, then
$\mathcal{B}$ is convex (in its own affine structure) and is
topologically either a disk, an annulus, or a M\"obius band.
There are globally convex examples if  $\cB$ is a sphere.

$\bullet$ (Theorem \ref{thm:ConvexSmallMonodromy},
Theorem \ref{thm:GlobalConvexnDnoncompact}).
\textit{Global convexity if there exists a system-preserving
$\mathbb{T}^{n-1}$-action or, equivalently,
$(n-1)$ globally well-defined affine functions on the base space}:
Let $\cB$ be the base space of a toric-focus integrable Hamiltonian
system with $n$ degrees of freedom on a connected symplectic
manifold $M$. Assume that the system admits a
global Hamiltonian $\mathbb{T}^{n-1}$-action and that $\cB$
is either compact or satisfies some  properness
conditions. Then $\cB$ is convex.

\subsection*{Negative convexity results} \hfill

All negative convexity results we have in this paper are due
to monodromy, which is the main potential obstruction to
convexity. We find integrable Hamiltonian systems whose base
spaces of the associated Lagrangian foliations possess
sufficiently complicated monodromy which then enforces
non-convexity of these base spaces.

$\bullet$ (Theorem \ref{thm:NonConvexS2})
\textit{Existence of ``\textit{affine black holes}''
in dimension 2 (or higher), leading to locally convex but globally
non-convex integral affine structures
on the 2-sphere $S^2$} (with 24 focus singular points,
counted with multiplicities).

$\bullet$ (Theorem \ref{thm:NonConvexDim3})
\textit{Existence of a non-convex box of dimension $3$ (or
higher) around two focus curves passing nearby each other}.
The local monodromy group in this case is a
representation of the free group with 2 generators.

$\bullet$ (Theorem \ref{thm:NonConvexFocus2})
\textit{Existence of a non-convex $\text{focus}^m$ box
for any $m \geq 2$}. The local monodromy group in this case is
the free Abelian group $\mathbb{Z}^m$.

\subsection*{Organization of the paper} \hfill

The paper is organized as follows.

In Section \ref{section:ReviewConvex1} and \ref{section:ReviewConvex2}
we give a brief overview of convexity results in symplectic geometry and
in integrable Hamiltonian systems. Strictly speaking, these sections are
not necessary for the understanding of the rest of the paper. However,
we believe that this overview is interesting in its own right, especially
to the non-experts, and give a clear picture of where our main results
are situated in the field of symplectic convexity.

In Section \ref{section:ToricFocus}
we recall the notion of integrable Hamiltonian systems
and present their interpretation as
singular Lagrangian fibrations. We recall the main results about their
singularities, give the associated normal forms, and present the topology
and affine structure on the base space of the fibration. This naturally
leads to the definition of toric-focus systems, the object of study in
this paper, whose known properties are reviewed.

Section \ref{sec_ffsingularities} is devoted to the study of affine
manifolds with focus singularities. We first show how monodromy
induces affine coordinates near focus points in arbitrary dimensions
and study the behavior of the affine structure in the presence of
focus points in increasing order of complexity.

In Section \ref{sec_straight_lines_convexity} we define the notion of
straight lines, with special emphasis on
the ones passing through singular points. This allows us to introduce
the notions of convexity and local convexity with respect to the
underlying singular affine structure. We show the existence and
branching of extensions of straight lines hitting singular points. In
view of global convexity, we also introduce the notions of strongly
convex subsets of a singular affine manifold.

In Section \ref{section:LocalConvexity} we prove local convexity
theorems for the singular affine structure around a singular point
with just one focus component. We also explain that the presence of a
singular point with many focus components may lead to non-convexity,
even for compact manifolds, by showing why it is not convex on an example.

Section \ref{sec_global_convexity} is devoted to both positive and
negative results on global convexity. We show that if the base space
of the singular Lagrangian fibration of a two degrees of freedom
integrable system with no hyperbolic singularities has non-empty boundary
and is either compact or proper, then it
is convex. If this base space is a two dimensional sphere,
in contrast to the case when it has boundary, it is not necessarily
convex. We present both a convex and non-convex example in this situation.
If the dimension of this base space is at least three, then there are
non-convex examples, even when the manifold is compact, has a boundary,
and is topologically a cube. All the non-convex examples are related
to the fact that the monodromy group is big. When there exist $n-1$ global
affine functions on the base space, which means that the monodromy group
is not too complicated, we prove that the base
space is convex in any dimension $n$, under the assumption
that the base space is compact or its affine
structure is proper. If the affine structure is not proper, then it is
already known \cite{PeRaVN2017} that convexity fails, even if there is
a Hamiltonian $\mathbb{T}^{n-1}$-action.

\section{A brief overview of convexity in symplectic geometry}
\label{section:ReviewConvex1}

The present work is part of a big program aimed at understanding
the relationship between symplectic and Poisson
geometry to convexity of the momentum maps of Hamiltonian actions
and several of its generalizations.
Readers who are familiar with this subject can skip this overview section.

\subsection{Kostant's Linear Convexity Theorem} \hfill

 The link between convexity
and coadjoint orbits appears for the first
time in a 1923 paper of Schur \cite{Schur1923} who showed that the set of
diagonals of an isospectral set of Hermitian $n \times n$ matrices,
viewed as a subset of $\mathbb{R}^n$, lies
in the convex hull whose vertices are the vectors formed by the $n!$
permutations of its eigenvalues. Horn \cite{Horn1954a} proved that this
inclusion is an equality. The fundamental
breakthrough pointing the way towards the basic relationship between
convexity and certain aspects of Lie theory and symplectic geometry
is due to Kostant \cite{Kostant1973} who realized that the
\textit{Schur-Horn Convexity Theorem} is
just a special case of a much more general statement: the projection
of a coadjoint orbit of a connected compact Lie group relative to a
bi-invariant inner product onto the dual of a Cartan subalgebra is the
convex hull of the Weyl group orbit of one (hence any) intersection points
of the orbit with the Cartan algebra. Even this result is a special
case of a convexity theorem for symmetric spaces.

The influence of this
theorem, called \textit{Kostant's Linear Convexity Theorem}
in the development of the relationship between geometric aspects
of Lie theory, symplectic and Poisson geometry, infinite dimensional
differential geometry, affine geometry, and integrable systems cannot
be overstated. We give below a very quick and incomplete synopsis of the
enormous work generated by this result in order to put our paper in this
larger context. For an excellent review of the convexity results in Lie
theory and symplectic geometry see Guillemin and Sjamaar's book
\cite{GuSj2005} and references therein.
We isolate five main areas where Kostant's Linear Convexity
Theorem has important extensions and has led to remarkable results.

\subsection{Infinite dimensional Lie theory}  \hfill

The first infinite
dimensional convexity result, due to Atiyah and Pressley \cite{AtPr1983},
extends Kostant's Convexity Theorem to loop groups of compact connected
and simply connected Lie groups. A more general convexity theorem,
valid for all coadjoint orbits of arbitrary Kac-Moody Lie groups associated
to a symmetrizable generalized Cartan matrix, except for some degenerate
orbits, is due to Kac and Peterson \cite{KaPe1984}.  The direct
generalization of the Schur-Horn Convexity Theorem to the Banach Lie group
of unitary operators on a separable Hilbert space is due to Neumann
\cite{Neumann1999}. The main issue in the formulation of the theorem
is the topology used to close the convex hull and the author discusses
both Schatten classes as well as the operator topology. This work is
continued in \cite{Neumann2002} by considering the infinite dimensional
orthogonal and symplectic groups. Again, the essential issue is the
topology used for the closure of the convex hull.

A different type of
Schur-Horn Convexity Theorem for an appropriate completion of the group
of area preserving diffeomorphisms of the annulus $[0,1] \times S^1$ is
due to Bloch, Flaschka, and Ratiu \cite{BlFlRa1993}. The completion
relative to various topologies is shown to equal the semigroup of not
necessarily invertible measure preserving maps of the annulus. The role
of the maximal torus is played the Hilbert space of $L^2$-functions on
the interval $[0,1]$ and that of the Weyl group by semigroup of maps
$[0,1] \times S^1 \ni (z, \theta)\mapsto(a(z), j(x)\theta)$, where
$a(z)$ is measure preserving and $j(z) = \pm 1$ a.e. It is shown that the
Schur-Horn Convexity Theorem holds: the orthogonal projection of the
set of functions with the same moments, given by integration over the
circle, equals a weakly compact convex set in the Hilbert space of
$L^2$-functions on the interval $[0,1]$ whose extreme points coincide
with the Weyl semigroup orbit. The link of this convexity theorem
with Toeplitz quantization, measure theory, the dispersionless Toda PDE,
and the PDE version of Brockett's double bracket equation is discussed in
\cite{BlFlRa1996, BlEHFlRa1997}. The convexity result has been extended
to the subgroup of equivariant symplectomorphisms of a symplectic toric
manifold in the Ph.D. thesis of Mousavi \cite{Mousavi2012} (unpublished).
Our paper does not explore such infinite dimensional generalizations.

\subsection{``Linear'' symplectic formulations}  \hfill

The so-called ``linear''
convexity theorems have their origin in the study of the behavior of
eigenvalues of matrices under linear operations; the Schur-Horn Convexity
Theorem and Konstant's Linear Convexity Theorem are representatives of
this point of view. The foundational work extending these theorems to
symplectic geometry is due to Atiyah \cite{At1982} and Guillemin and
Sternberg \cite{GuSt-Convexity1982, GuSt1984}. Let $(M^{2n},\omega)$ be a
$2n$-dimensional symplectic manifold endowed with a Hamiltonian
$\mathbb{T}^k$-action, with invariant momentum map $\mathbf{J}: M
\rightarrow \mathbb{R}^k$. Then the fibers of $\mathbf{J}$ are connected
and $\mathbf{J}(M)$ is a compact convex polytope, namely the convex
hull of the image of the fixed point set of the $\mathbb{T}^k$-action;
$\mathbf{J}(M)$ is called the \textit{momentum polytope}.
If the $\mathbb{T}^k$-action is effective (the intersection of all isotropy
subgroups of each point in $M^{2n}$ is the identity), then there must be
at least $k+1$ fixed points of the action and $k \leq n$.

The generalization of the Atiyah-Guillemin-Sternberg theorem  to
compact Lie group actions on compact connected symplectic manifolds
is due to Kirwan \cite{Kirwan1984}: the fibers of the momentum map
are connected and the image of the momentum map intersected with a Weyl
chamber is a connected polytope. For Hamiltonian actions of compact
Lie groups, this intersection is called \textit{momentum
polytope}. The momentum polytope of the restriction of the Hamiltonian
action of a compact Lie group to its maximal torus is the convex hull
of the Weyl group orbit of the momentum polytope of the Hamiltonian
compact Lie group action.

Brion \cite{Brion1986} sharpened Kirwan's Convexity Theorem in the
framework of projective algebraic varieties by finding a more detailed
description of this momentum polytope and linked it to the Kirwan
stratification \cite{Kirwan1984_book}. For an up to date account
of stratified spaces, see Pflaum's book \cite{Pflaum2001}.
Sjamaar \cite{Sjamaar1998} extended Brion's methods to symplectic
manifolds and gave a
description of the polytope in a neighborhood of any of its points $p$
in terms of the action of the stabilizer of the group action on the
manifold at a point $m$ satisfying $\mathbf{J}(m) = p$; this also yields a necessary
condition for $p$ to be a vertex. In addition, Kirwan's Convexity Theorem
was extended to actions on affine varieties and cotangent bundles.

Hilgert, Neeb, and Plank \cite{HiNePl-Convexity1994} extended Kirwan's
Convexity Theorem to non-compact symplectic manifolds with a proper
momentum map and obtained other convexity results.
Lerman \cite{Lerman1995} proved the same result for linear compact
Lie group actions on symplectic vector spaces whose associated momentum
map is not necessarily proper employing his symplectic cuts technique.
Lerman, Meinrenken, Tolman and Woodward \cite{LT-Orbifold1997,LMTW98}
also extended Atiyah--Guillemin--Sternberg--Kirwan's convexity theorems
to the case of symplectic \textit{orbifolds}, gave a local description of the
momentum polytope, and proved the openness of the map from the orbit space to the
momentum polytope.

The momentum polytope for non-compact manifolds has discrete vertices and
is, in general,  an unbounded, convex, locally finite intersection of polyhedral cones,
each of which is determined by local data on the
manifold, conveniently expressed in terms of the Marle-Guillemin-Sternberg
Normal Form \cite{Marle1984, Marle1985, GuSt1984_normal_form} (a
slice theorem for Hamiltonian actions; see also
\cite{GuSt1990}, \cite{OrRa2004} for a presentation of the normal form
and various extensions thereof).

Heinzner and Huckleberry \cite{HeHu1996}
replaced the compact symplectic manifold and compact Lie group in
Kirwan's theorem by an irreducible complex K\"ahler, not necessarily compact,
manifold with an action of the complexification of the compact Lie group and
showed that the image by the momentum map of the open subset
consisting of the Lie group orbits of maximal dimension, intersected with
a Weyl chamber, is convex.

Knop \cite{Knop2002} introduced the concept
of convex Hamiltonian manifold in the following way. Suppose that
a symplectic manifold $M$ admits a Hamiltonian Lie group action. Choose a
Cartan subalgebra $\mathfrak{t}$ in the Lie algebra of the compact
symmetry Lie group and a positive Weyl chamber $\mathfrak{t}^\ast _+$.
For each point in the manifold, the image of the orbit by the momentum
map intersects the positive Weyl chamber  in exactly one point. This
defines a map $\psi: M \rightarrow \mathfrak{t}^\ast _+$. A convex
Hamiltonian manifold is defined then as a Hamiltonian manifold such
that the inverse image by $\psi$ of any straight line segment in
$\mathfrak{t}^\ast _+$ is connected in $M$. Then it is proved that
convexity of $M$ is equivalent to the convexity of $\psi(M)$ together with
the connectedness of the fibers of $\psi$ and openness of $\psi:M
\rightarrow \psi(M)$, the range being endowed with the subspace topology.
It is shown that several convexity results are a consequence of the
fact that the hypotheses of these theorems guarantee that the symplectic
manifold is convex. It is also proved that $\psi(M)$ is locally a
polyhedral cone and that results that required in the hypothesis
properness of the momentum map, generalize to convex symplectic
manifolds.

Kostant \cite{Kostant1973}
already showed that his Linear Convexity Theorem holds also for real
flag manifolds and Atiyah \cite{At1982} generalized it to the symplectic
setting. The general version of this theorem is due to Duistermaat
\cite{Duistermaat1983} who proved that the image of the fixed point
set of an antisymplectic involution by the momentum map of a toral
Hamiltonian action on a compact connected symplectic manifold, coincides
with the momentum polytope. This result has striking consequences, for
example, in the study of the gradient character of the finite Toda lattice
system for compact semisimple Lie algebras, the convexity properties
of the associated momentum map and their relationship the Brockett
double bracket equation (a gradient system in the normal metric),
as well as certain ``non-linear'' convexity theorems (see \cite{BlFlRa1990}),
which will be discussed in the next subsection.

If the symplectic action of a compact Lie group does not admit a
momentum map, there are two generalizations of the convexity theorems.
The first one is based on a construction of Condevaux, Dazord, and Molino
\cite{CoDaMo-Moment1988}.
They introduce a flat connection on the trivial bundle whose structure
group is the underlying additive group of the dual of the Lie algebra of
symmetries over the symplectic manifold, consider the holonomy bundle
in this product containing a given point (whose structure group is the
holonomy group based at that point), and the projection to the dual of
the Lie algebra. The cylinder valued momentum map is the
quotient of this projection from the original symplectic manifold to
the quotient of the dual of the Lie algebra of symmetries by the closure
of the holonomy group; for a detailed presentation of this momentum
map and its properties see \cite[Sections 5.2-5.4, 7.6]{OrRa2004}.

Birtea, Ortega, and Ratiu \cite{BiOrRa2008} proved a convexity theorem
for cylinder valued momentum maps of compact symplectic group actions
in the spirit of Kirwan's Convexity Theorem supposing that the holonomy
group mentioned above is closed and that the cylinder valued momentum map
is a closed map and is tube-wise Hamiltonian (each point admits an open
group invariant neighborhood such that the restriction of the action to
this neighborhood admits a standard momentum map; for connected Abelian
Lie group actions, this hypothesis is not necessary): the intersection
of the range of the momentum map in the quotient of the dual of the
symmetry Lie algebra by the closed holonomy group (i.e., the ``cylinder'')
intersected with the quotient of a Weyl chamber by this holonomy group
is weakly convex. This statement uses the natural length metric
of the cylinder (so it is a geodesic metric space) and weak convexity
of a set means that any two points can be joined by a geodesic
contained in this set, however, not necessarily the shortest one.

The second approach, due to Benoist \cite{Benoist2002} and Giacobbe
\cite{Giacobbe2000, Giacobbe2005}
is to work with the momentum map naturally associated to an appropriate
covering of the symplectic manifold. Benoist works with the universal
covering of the symmetry group acting symplectically on the universal
covering of the symplectic manifold. Suppose that this momentum map
is proper modulo the holonomy group, i.e., the inverse image of
any compact set is included in the holonomy group orbit of a compact
set in the symplectic manifold. Then the image of the momentum map
intersected with a Weyl chamber is convex (and has the other usual
properties in convexity theorems), also extending the Kirwan Convexity
Theorem.

Giacobbe \cite{Giacobbe2000, Giacobbe2005}
showed first that if the group is a torus, there is a
minimal covering on which the given symplectic action admits a momentum
map and that the image of this covering by the momentum map equals a
convex compact polytope times a vector space, both being explicitly
described. Then he showed that, for an non-Abelian compact group $H$
which acts symplectically on $M$, there exists an $H$-equivariant
momentum map $\mathbf{J}:
\widetilde{M} \rightarrow \mathfrak{h}^* \cong \mathfrak{g}^*  \times
\mathfrak{t}^\ast$, where $\mathfrak{h}$ is the Lie algebra of $H$,
$\mathfrak{g}$ is the Lie algebra of the commutator $G$ of $H$
which is semisimple, $\mathfrak{t}$ is the Lie algebra of the center of $H$
which is a torus $T = T_e \times T_c$, where $T_e$ is the maximal subtorus
of $T$ acting in a Hamiltonian fashion on $M$, $T_c$ is a torus complement to
$T_e$ in $T$, and $\widetilde{M}$ is the minimal covering of $\widetilde{M}$
on which the $T$-action can be lifted to a Hamiltonian action.
Then $\mathbf{J}\left(\widetilde{M} \right)
\cap (\mathfrak{s}_+^* \times \mathfrak{t}^\ast ) = P \times
\mathfrak{t}_c^*$, where $\mathfrak{t}_c = Lie\ T_c,
\mathfrak{t}_e = Lie\ T_e$, $\mathfrak{s}_+^*$ is the positive
Weyl chamber of $\mathfrak{g}$,
and $P \subset \mathfrak{s}_+^* \times
\mathfrak{t}_e^\ast$ is the product of the momentum polytope in
$\mathfrak{t}^\ast _e$ given by the Hamiltonian $T_e$-action on $M$
and the $G$-Kirwan polytope in $\mathfrak{s}_+^*$.
Giacobbe also proved any effective symplectic action of a
$n$-dimensional torus on a closed connected symplectic $2n$-manifold
with a fixed point must be Hamiltonian.

There are also convexity theorems for presymplectic manifolds. For torus
actions, this result is due to Ratiu and Zung
\cite{RaZu-Presymplectic2017}. Let
$\mathbf{J}: M^{2d+q} \rightarrow  \mathbb{R}^{q+d}$ be a flat momentum
map of a Hamiltonian torus $\mathbb{T}^{q+d}$-action on a connected
compact presymplectic manifold whose presymplectic form has constant
corank $d$. The flatness condition means that the image
$\mathbf{J}\left(M^{2d+q}\right)$ lies in the intersection of $d$
hyperplanes in $\mathbb{R}^{q+d}$. Then
$\mathbf{J}\left(M^{2d+q}\right)$ is a convex $q$-dimensional polytope
(rational or non-rational)
lying in the $q$-dimensional affine subspace of $\mathbb{R}^{q+d}$
given by the flatness condition. In addition, any such presymplectic
manifold admits a unique equivariant symplectization.
For general compact Lie groups, the convexity result is due to Lin and
Sjamaar \cite{LiSj-PresymplecticConveity2017}, extending to
presymplectic manifolds Kirwan's Convexity Theorem and, of course,
containing the presymplectic convexity theorem just described.

\subsection{``Non-linear'' symplectic formulations} \hfill

``Non-linear'' convexity theorems appear as the result of
the analysis of the behavior of eigenvalues or
singular values of matrices under nonlinear
operations (usually, multiplication). There are far less
theorems of this kind and they are considerably more involved,
although many known convexity
theorems (for example, those presented in Marshall, Olkin, and Arnold's
book \cite{MaOlAr2011}) await a symplectic interpretation, albeit mostly
not in the classical sense, as will be apparent from the presentation
below. Reformulating them in a symplectic, Poisson, Dirac, or groupoid
context, remains a challenge. The matrix case of such theorems goes
back to Weyl \cite{Weyl1949} and Horn \cite{Horn1954b}. Let $P$ be
the set of positive definite Hermitian matrices whose determinant
equals 1 and $\Sigma_\lambda$ the
isospectral subset of matrices in $P$ defined by the given eigenvalues
$(\lambda_1,\dots,\lambda_n) \in \mathbb{R}^n$. The image of the map
$P \ni p \mapsto (\log(\det p_1),\dots,\log(\det p_n)) \in \mathbb{R}^n$,
where $p_k=(p_{ij}),\, i,j=1,\ldots,k$ is the $k\times k$ principal
sub-matrix of $p$, is a convex polytope.

The Lie theoretical generalization
of the above Weyl--Horn theorem is again due to Kostant \cite{Kostant1973}
and is called
\textit{Kostant's Non-linear Convexity Theorem} which has the following
formulation. Let $G$ be a connected semisimple Lie group and $G= KAN
= PK$ its Iwasawa and Cartan decomposition, respectively, with
$A \subset P$. Let $\mathcal{O}_a$ be the $K$-orbit of $a \in A$ in
$P$ (by conjugation) and $\rho_A: G \rightarrow A$ the Iwasawa projection
$\rho_A(kan) = a$. Identify $A$ with its Lie algebra $\mathfrak{a}$ by
the exponential map. Then $\rho_A( \mathcal{O}_a)$ is the convex hull
of the Weyl group orbit through $a$. The Weyl-Horn Convexity
theorem is a special case of Kostant's Non-linear Convexity Theorem by
making the following choices: $G = SL(n,\mathbb{C})$,
$K=SU(n)$, $P$ positive definite Hermitian matrices of determinant 1
(i.e., the Cartan complement of $K$ in $G$), $A$ positive
diagonal matrices of determinant equal to 1, $\Sigma_\lambda$ the orbit
(by conjugation) of $K=SU(n)$ in $P$ through $(\lambda_1, \ldots,\lambda_n)
\in A$, $p = kan$ the Iwasawa decomposition of $p \in P$,
$n$ upper triangular with ones on the diagonal. Then the map
$P \ni p \mapsto (\log(\det p_1),\dots,\log(\det p_n)) \in \mathbb{R}^n$
is the logarithm of $\rho_A$.

A symplectic proof of Kostant's theorem for a large class of
Lie groups was given by Lu and Ratiu \cite{LuRa1991}; in the special
case of the Weyl-Horn Convexity Theorem,
$\Sigma_\lambda$ is a symplectic manifold, the map $P \rightarrow
\mathbb{R}^n$ given above is the momentum map for a Hamiltonian torus
action, and convexity follows from the Atiyah-Guillemin-Sternberg
Convexity Theorem. However, there is a twist:
the symplectic form on the orbit $\Sigma_\lambda$ is not due to
a Lie-Poisson structure on the dual of a Lie algebra, but it is induced
from a Poisson structure compatible with the Lie group structure on
the group $SL(n,\mathbb{C})$. This has far reaching consequences, as
discussed below. The symplectic proof of Kostant's Non-linear Convexity
Theorem given in \cite{LuRa1991} is based on the following idea. Consider
the complexification $G^\mathbb{C}$ of $G$ and write the Iwasawa
decomposition $G^\mathbb{C} = UB$, where $U$ is a maximal compact
subgroup and $B$ is the solvable complement; $B$ is a Poisson-Lie
group (i.e., a Lie group and a Poisson manifold for which multiplication
is a Poisson map) endowed with the so-called Lu-Weinstein Poisson
structure \cite{LuWe1990} whose symplectic leaves are the dressing
orbits (left or right) of $U$ on $B$. The restriction of this action
to a maximal torus of $U$ is Hamiltonian and its associated momentum
map is the Iwasawa projection of $G^\mathbb{C}$, which then
proves Kostant's Convexity Theorem for $G^\mathbb{C}$ viewed as a
real Lie group by invoking the Atiyah-Guillemin-Sternberg Convexity
Theorem.

To get the general case, one would like to apply Duistermaat's
Convexity Theorem for the antisymplectic involution induced by
the Cartan involution, viewing $AN$ in the Iwasawa decomposition
$G = KAN$ as the fixed point set in $B$. However, $B$ is not invariant
(as claimed in \cite{Duistermaat1983} and then used in \cite{LuRa1991})
and the argument holds only if the centralizer of $\mathfrak{a}$ in
$\mathfrak{k}$ is Abelian, as Hilgert and Neeb \cite{HiNe1998} pointed
out. This paper contains several interesting results on nonlinear
convexity theorems. The map $B \ni b \mapsto \log(b^*b) \in \mathfrak{p}$
is a diffeomorphism sending each symplectic leaf of $B$ (a dressing
orbit) to a coadjoint orbit of $K$ in $\mathfrak{p}$; here
$\mathfrak{u}$ is the $+1$ eigenspace of the Cartan involution $\tau$
on $G^\mathbb{C}$ viewed as a real Lie group, so $\mathfrak{u}$ is
the Lie algebra of $U$, $\mathfrak{p}$ is the $-1$ eigenspace of
$\tau$, $b^* = \tau(g ^{-1})$ (so for $G = SL(n, \mathbb{C})$ it is just
the usual transpose conjugate), $\mathfrak{u}^*$ is identified with
$\mathfrak{p}$ via the imaginary part of the Killing form. So, a coadjoint
orbit in $\mathfrak{p}$ carries two symplectic forms: the usual
orbit symplectic form and the push forward of the dressing symplectic form
by the map $B\ni b \mapsto \log(b^*b) \in \mathfrak{p}$. Ginzburg and
Weinstein \cite{GiWe1992} proved that there is a global diffeomorphism
on $\mathfrak{p}$ taking the Lie-Poisson tensor on $\mathfrak{p}$ to
the push forward of the Poisson Lie tensor on $B$ to $\mathfrak{p}$ by
the map defined above. The symplectic proof of the Kostant Non-linear
Convexity Theorem for all connected real semisimple Lie groups was given
in Sleewaegen's 1999 Ph.D. thesis \cite{Sleewaegen1999} (unpublished) and
later by Kr\"otz and Otto \cite{KrOt2006} each extending the
Duistermaat Convexity Theorem in such a way that the symplectic proof
outlined above works in full generality.

The previously described proof of the Kostant Non-linear Convexity Theorem
has sparked interest in other types of non-linear convexity results.
The first such theorem was given by Flaschka and Ratiu
\cite{FlRa-Convexity1996} and is very general. Let $K$ be the compact
real form of a connected complex semisimple Lie group $G$ (hence $K$ is
connected) and denote by $G^\mathbb{R}$ the real underlying Lie group.
Let $G^\mathbb{R}= KAN$ be the Iwasawa decomposition,
$\mathfrak{g}$, $\mathfrak{g}^\mathbb{R}$,
$\mathfrak{k}$, $\mathfrak{a}$, $\mathfrak{n}$ the Lie algebras of $G$,
$G^\mathbb{R}$, $K$, $A$, $N$, respectively, and $T \subset K$
the connected maximal torus in $K$ whose Lie algebra is $\mathfrak{t} =
{\rm i}\mathfrak{a}$. Then $\mathfrak{b}= \mathfrak{a} + \mathfrak{n}$
and $\mathfrak{k}$ are dual to each other via the imaginary part of the
Killing form; thus $\mathfrak{k}$ is isomorphic to $\mathfrak{b}^\ast$.
The Cartan decomposition $G^\mathbb{R} = PK$ defines the Cartan involution
$\tau: G^\mathbb{R} \rightarrow G^\mathbb{R}$, $\mathfrak{k}$ is the
$+1$ eigenspace of $\tau$ and let $\mathfrak{p}$ be the $-1$ eigenspace
of $\tau$. Denote by $\mathfrak{a}_+$ a positive Weyl chamber in
$\mathfrak{a}$. Now suppose that $K$ acts on a compact connected
symplectic manifold $(M, \omega)$. Assume that the restriction of this
action to the maximal torus $T \subset K$ is Hamiltonian with associated
invariant momentum map $\phi: M \rightarrow \mathfrak{a} \cong
\mathfrak{t}^*$. Suppose there exists a map $j:M \rightarrow \mathfrak{p}$
with the following four properties: $j$ is equivariant relative to the
adjoint action of $K$ on $\mathfrak{p}$; $T_mj(T_mM)$ is the annihilator
of the stabilizer subalgebra $\mathfrak{k}_m$ for every $m \in M$;
$\ker T_mj$ coincides with the $\omega$-orthogonal complement of the
tangent space to the $K$-orbit at $m$ for every $m \in M$; the
restriction of $j$ to $j^{-1}(\mathfrak{a}_+)$ is proportional to $\phi$.
(Note that the second and third conditions are the content of the
Reduction Lemma; see, e.g.,  \cite[Lemma 4.3.4]{AbMa1978},
\cite[Propositions 4.5.12 and 4.5.14]{OrRa2004}.) Then the fibers of $j$
are connected and $j(M) \cap \mathfrak{a}_+$ is
a convex polytope. The hypotheses of this theorem are non-trivially
verified if $j$ is the Lu momentum map (\cite{Lu1990, LuWe1990}) of
a Poisson Lie group structure on $K$ \cite{FlRa-Convexity1996}.

All
Poisson Lie group structures on compact Lie groups have been classified
and the dual groups computed (\cite{LeSo1991, FlRa-Convexity1996}),
the most important one being the Lu-Weinstein Poisson Lie structure
\cite{LuWe1990}. If the compact Lie group $K$ in the Flaschka-Ratiu
Convexity Theorem is a Poisson Lie group, Alekseev \cite{Alekseev1997}
has given a different proof
of this theorem, by modifying the symplectic structure on $M$ using
the Poisson Lie group structure on $K$, which then reduces this theorem
for Poisson actions of compact Poisson Lie groups on symplectic manifolds
to the usual Kirwan Convexity Theorem (see also \cite{GuSj2005}
for a discussion of this method).

A non-linear convexity theorem for quasi-Hamiltonian actions is due to
Alekseev, Malkin, and Meinrenken \cite{AlMaMe1998}. The momentum
map for quasi-Hamiltonian actions is group valued, so convexity
means that the projection of the image of the group valued momentum
map to the space of conjugacy classes of the group, identified
with the fundamental Weyl alcove in a choice of a positive Weyl
chamber, is convex; this projection is the momentum polytope for
connected quasi-Hamiltonian spaces of compact, connected, simply
connected Lie group actions.
It should be emphasized  that the Lu momentum map for
Poisson Lie group actions (whose values lie, by definition, in the
dual Poisson Lie group of the Poisson Lie group whose action generates
it, if it exists) is \textit{not} an example of this
group valued momentum map; the two momentum maps coincide only for
Abelian groups; see \cite[Section 5.4]{OrRa2004} for a discussion and its
relationship to the cylinder valued momentum map introduced in
\cite{CoDaMo-Moment1988}.

The existing momentum  map convexity
theorems strongly suggest a convexity theorem for the optimal momentum
map introduced by Ortega and Ratiu \cite{OrRa2002}
(see \cite[Sections 5.5 and 5.6]{OrRa2004} for a study of its properties),
which always exists for any canonical Lie group action on a Poisson
manifold. The target of the optimal momentum map is the quotient
topological space of the Poisson manifold on which the group acts by
the polar pseudo-group of the group of diffeomorphisms given by the action.
The problem is to find an intrinsic definition of the concept of convexity
for this topological space. The conjecture is that this topological space
is intrinsically convex.

Zung \cite{Zung-Proper2006} showed that every proper
quasi-symplectic groupoid $(\Gamma \rightrightarrows P, \omega+\Omega)$
in the sense of Xu \cite{Xu-Momentum2003} (also known as twisted presymplectic
groupoid  \cite{BCWZ-TwistedDirac2003} ) is symplectically linearizable,
and studied the intrinsic affine structures of appropriate quotient spaces
of proper quasi-symplectic groupoids and of their quasi-Hamiltonian manifolds.
He showed that, under some mild conditions, these intrinsic affine structures
are intrinsically convex. Most existing symplectic convexity results,
both linear and nonlinear, and also for group-valued momentum maps,
can be recovered from this convexity result for quasi-Hamiltonian manifolds
of proper quasi-symplectic groupoids as important particular cases. 
In particular, one can cite here a very beautiful result of Weinstein
\cite{Weinstein_Convex2001}, which was one of the original motivations
for the study of proper groupoids and momentum maps:  the set of ordered
frequency $n$-tuples of all possible sums of a pair of positive definite quadratic Hamiltonian functions on $\mathbb{R}^{2n}$ with given frequencies 
is an unbounded closed convex locally polyhedral subset of $\mathbb{R}^{n}$. 
There are remarkable similarities between the Ortega--Ratiu construction of the optimal momentum map and Zung's construction
of the intrinsic transverse structures, so it is likely that these objects are
closely related together, though no one worked it out yet. For another very
recent result on transverse symplectic complexity, see~\cite{Ishida2017}

\subsection{Local-Global Convexity Principle} \hfill

All proofs of the symplectic convexity theorems rely on the following
strategy. First one proves some local convexity properties of the map in
question, using some kind of local normal forms, e.g., the
Marle-Guillemin-Sternberg normal norm
\cite{Marle1984, Marle1985, GuSt1984_normal_form}, which has its roots
in the quadratic nature of the singularities of the momentum
map, first observed by Arms, Marsden, and Moncrief \cite{ArMaMo1980}.
Then one uses, under certain hypotheses, a passage from local
convexity to global convexity.

Originally, this passage from local to global
was done using Morse theory, but it became later
apparent that these Morse techniques do not always apply.
A better tool is the so-called  \textit{Local-Global Convexity Principle}.
The origins of this principle go back to Tietze \cite{Tietze1928} and
Nakajima \cite{Nakajima1928}: \textit{if a connected closed subset
$S \subset \mathbb{R}^n$ is locally convex, then it is convex}. Since then,
there are many extensions and developments of this result,
due to Schoenberg \cite{Schoenberg1942}, Klee \cite{Klee1951},
Sacksteder, Straus, and Valentine \cite{SaStVa1961}, Blumenthal and
Freese \cite{BlFr1962}, Kay \cite{Kay1985}, Cel \cite{Cel1998}, and so on.

Condevaux, Dazord, and Molino \cite{CoDaMo-Moment1988} were the first authors
to use the local-global convexity principle to give a different and much simpler
proof of the Atiyah-Guillemin-Sternberg and Kirwan Convexity Theorems.
Then Hilgert, Neeb, and Plank \cite{HiNePl-Convexity1994} isolated this principle
as a fundamental tool, and gave the following version of it, which is well
adapted for symplectic convexity.

Let $f:S \rightarrow V$ be a continuous map from a connected Hausdorff
topological space $S$ to a vector space $V$. The map $f$ is said to
be locally fiber connected if for each $s \in S$, there is an open
neighborhood $U \subset S$ of $s$ such that $f^{-1}(f(u))\cap U$ is
connected for all $u \in U$. Suppose that there is an assignment
of a closed convex cone $C_s \subset V$ with vertex $f(s)$ to each
$s \in S$. Such an assignment is called local convexity data for $f$
if for each $s\in S$ there is an open neighborhood $U_s \subset S$ of
$s$ such that $f|_{U_s}: U_s \rightarrow C_s$ is an open map and
$f ^{-1}(f(u)) \cap U_s$ is connected for all $u \in U_s$. The
Local-Global Convexity Principle in \cite{HiNePl-Convexity1994} states
that if $f:S \rightarrow V$ is a proper locally fiber connected map
with local convexity data $\{C_s \mid s \in S\}$, then $f(S) \subset V$ is
a closed convex locally polyhedral set, the fibers of $f$ are connected,
$f:S \rightarrow f(S)$ is an open map (with respect to the subspace
topology of $f(S) \subset V$), and $C_s = f(s) + L_{f(s)}(f(S))$, where
$L_{f(s)}(f(S))$ denotes the closure of the cone $\mathbb{R}_+(f(S)-f(s))$.

The above Local-Global Convexity theorem of
Hilgert, Neeb, and Plank \cite{HiNePl-Convexity1994} was
crucial in the proof of the Flaschka-Ratiu Convexity
Theorem \cite{FlRa-Convexity1996}, in particular for all compact Poisson Lie
groups, because the usual Morse theoretic arguments failed to work for the
Poisson Lie group structures different from the Lu-Weinstein one.

Prato \cite{Prato1994} pointed out that the loss of compactness of the
manifold also implies, in general, the loss of the convexity of the
momentum map, even for torus actions. She also showed that if there is
an integral element in the Lie algebra of the symmetry torus such that
the momentum map component for this element is proper and has its
minimum as its unique critical value, then the image of the momentum map
of the Hamiltonian torus action is the convex hull of a finite number
of affine rays starting in the fixed point set of the action. This
convexity theorem has been vastly generalized in several directions.
Birtea, Ortega, and Ratiu \cite{BiOrRa2009} have shown that only closedness
of the momentum map is sufficient to obtain a convexity result, thereby
setting the stage for other convexity theorems in infinite dimensions
than the ones mentioned earlier. These results have been further
generalized in Birtea, Ortega, and Ratiu \cite{BiOrRa2008} where the
target map is a length metric space; this has as consequence the
convexity theorem for cylinder valued momentum maps, mentioned earlier.
Bjorndahl and Karshon \cite{BjKa2010} have provided another Local-Global
Convexity theorem for proper continuous maps from a connected Hausdorff
topological space to a convex set in Euclidean space.

The paper \cite{Zung-Proper2006} by Zung contains another simple
version of the local-global convexity principle,
which is the one that we will use this paper  to show
some positive global convexity results for toric-focus integrable
Hamiltonian systems (see Section \ref{sec_global_convexity}).

\section{Convexity in integrable Hamiltonian systems}
\label{section:ReviewConvex2}

Integrable Hamiltonian systems in the classical sense of Liouville
are a special case of Hamiltonian group actions on symplectic manifolds,
where the group is the non-compact Abelian group $\mathbb{R}^n$,
with $n$ being half of the dimension of the manifold.

Each integrable Hamiltonian system with a proper momentum map has an
associated \textit{base space}, whose points correspond to connected
components of the level sets of the momentum map. Due to the existence
of local \textit{action} coordinates, this base space is equipped
with an intrinsic \textit{integral affine structure}, which is singular
in general, and which plays a key role in the classification of integrable
systems; see, e.g., \cite{Duistermaat-AA1980,Zung-Integrable2003}.

If the integrable system admits hyperbolic singularities, then the base
space has \textit{branching points}, and it does not make much sense to
even talk about local convexity of the affine structure at those points.
However, there are large classes of integrable Hamiltonian systems
without hyperbolic singularities, for which we can study convexity
properties of the base spaces. We present a brief overview of such
systems in this section.

\subsection{Toric systems and their momentum polytopes}
\label{subsection_toric_review}  \hfill

A symplectic toric manifold is a compact connected symplectic
$2n$-dimensional manifold endowed with an effective Hamiltonian
$\mathbb{T}^n$-action. From the geometric point of view of integrable
Hamiltonian systems, they are nothing else but integrable systems
whose singular points are all elliptic
non-degenerate, and which admit a full set of global action variables.
See, e.g., Audin's book \cite{Audin2004} for an introduction to this topic of
toric integrable systems.

As a special case of the Atiyah--Guillemin--Sternberg theorem,
the base space of a symplectic toric manifold is a convex polytope
embedded in $\mathbb{R}^n$ via the momentum map, which satisfies three
additional properties: \textit{rationality} (each facet is given
by a linear equation whose linear coefficients are integers),
\textit{simplicity} (each vertex has exactly $n$ edges),
and \textit{regularity} (near each vertex, the polytope is locally
integral-affinely isomorphic to the orthant
$(x_1,\hdots,x_n) \in \mathbb{R}^n \mid x_1 \geq 0, \hdots, x_n \geq 0\}$
in $\mathbb{R}^n$). Such a convex polytope is called a \textit{Delzant
polytope}, because  Delzant \cite{Delzant1988} proved
a 1-to-1 correspondence between these polytopes and connected
compact symplectic toric manifolds (up to isomorphisms).
For non-compact symplectic toric manifolds, the situation is considerably
more involved; Karshon and Lerman \cite{KaLe2015} analyzed this
situation.

If we consider convex polytopes which are rational simple but not
regular, then they correspond to compact symplectic toric
\textit{orbifolds}, instead of manifolds. According to a result of
Lerman and Tolman \cite{LT-Orbifold1997}, there is a 1-to-1
correspondence between connected compact symplectic toric orbifolds
and \textit{weighted simple rational convex polytopes}: each facet
of the polytope is given an arbitrary natural number $w$, called its
weight, which corresponds to the orbifold type
$D^{2(n-1)} \times (D^2/\mathbb{Z}_w)$  at its preimage under the
momentum map ($D^k$ denotes a $k$-dimensional ball).

If the polytope is irrational then the situation is more complicated,
because there is no symplectic Hausdorff space corresponding to it.
Prato \cite{Prato_Nonrational2001} invented the notion of symplectic
\textit{quasifolds} to deal with this case. Some other authors
\cite{KLMV-NoncommutativeToric2014}
talk about non-commutative toric varieties in this situation.

Ratiu and Zung \cite{RaZu-Presymplectic2017} extended
Delzant's classification theorem to compact
\textit{presymplectic toric manifolds}. It states that
connected compact presymplectic manifolds are classified, up to
equivariant presymplectic diffeomorphisms, by their associated
\textit{framed momentum polytopes}, which are convex and may
be rational or non-rational. Unlike the symplectic case, the
polytope is not enough to characterize Hamiltonian presymplectic
manifolds and additional information, in this case the framing of the
polytope, is needed. Furthermore, they defined a \textit{Morita
equivalence relation} on the set of framed polytopes and, using it,
the classification of connected presymplectic toric manifolds is
alternatively given by the Morita equivalence classes
of their framed momentum polytopes.
Toric orbifolds \cite{LT-Orbifold1997}, quasifolds
\cite{BaPr_SimpleNonrational2001}, and non-commutative toric varieties
\cite{KLMV-NoncommutativeToric2014} can be viewed as quotients of
presymplectic toric manifolds by the kernel isotropy foliation
of the presymplectic form; thus, they are classified by the Morita
equivalence classes of their framed momentum polytopes.
The Lerman-Tolman \cite{LT-Orbifold1997}
classification of symplectic orbifolds by weighted simple rational
convex polytopes can be recovered from this Morita equivalence.
(See also \cite{LiSj-PresymplecticConveity2017}).

\subsection{Toric degenerations} \hfill

There are many interesting integrable Hamiltonian systems on compact
symplectic manifolds with degenerate singularities, whose base spaces
(with the associated integral affine structures)
are still convex polytopes. The most famous examples are probably
the so-called \textit{Gelfand-Cetlin system}, introduced by
Guillemin and Sternberg \cite{GuSt1983}, and the
\textit{bending flows} of
polygons in $\mathbb{R}^3$, introduced by Kapovich and Millson \cite{KaMi1996}.

It turns out that these two famous examples, and other systems
with convex momentum polytopes but degenerate singularities, have
some very similar topological and geometrical characteristics, namely:

- Their symplectic manifolds are not toric, but admit \textit{toric degenerations}.
See, e.g., \cite{AlBr-Degeneration2004,
KoMi-GC2005,HaKa2015} for toric degenerations related to integrable systems.
It's interesting to note that the momentum polytope corresponds to
that of a (singular) toric variety at the degeneration, see, e.g.,
\cite{BGL-SingularToric2005}.

- Even though their smooth momentum maps have degenerate
singularities, the inverse image of every point of the momentum polytope
under the momentum map is not a singular variety but rather a
\textit{smooth} isotropic manifold or orbifold. See, e.g.,
\cite{Bouloc2017,BMZ-GC2017} for some results in this
direction.

\subsection{Log-symplectic convexity} \hfill

The Delzant correspondence extends to certain classes of toric Poisson
manifolds which are generically symplectic but their Poisson structure
admits a degeneracy locus. If this degeneracy locus is a smooth
hypersurface, Guillemin, Miranda, Pires, and Scott \cite{GuMiPiSc2015}
proved a generalization of the Delzant correspondence. If, however,
this degeneracy locus has singularities of the type of
normal crossing configurations of smooth hypersurfaces, the manifolds
are called log symplectic, to emphasize the nature of the singularities.
For the formulation of the analogue of the Delzant correspondence,
several new ideas need to be introduced. This was done by Gualtieri,
Li, Pelayo, and Ratiu \cite{GuLiPeRa2017} and required the appeal
to the Mazzeo-Melrose \cite{MaMe1998} decomposition for manifolds with
corners, free divisors appearing in algebraic geometry, and certain
ideas of tropicalization. The notion of the momentum map is extended
to the global tropical momentum map whose range is constructed from
partial compactifications of affine spaces intimately linked to extended
tropicalizations of toric varieties.

The analogue of the Delzant theorem
to this log symplectic case states that there is a bijective
correspondence between
equivariant isomorphism classes of oriented compact connected toric
Hamiltonian log symplectic $2n$-manifolds and equivalence classes of
pairs $(\Delta, M)$, where $\Delta$ is a compact convex log affine
polytope of dimension $n$, satisfying the Delzant condition, and $M$
is a principal $n$-torus bundle over $\Delta$ with vanishing obstruction
class. From the point of view of integrable systems, which is the subject
of this paper, this convexity result yields a classification of a large
family of toric integrable systems having a base whose integral affine
structure degenerates in a very precise manner along a stratification.

\subsection{Non-Abelian integrability}  \hfill

Non-Abelian integrability appears for the first time, in a very general context,
in Abraham and Marsden's book \cite[Exercise 5.2I]{AbMa1978}: a
Hamiltonian action of a symmetry group is completely integrable if, generically,
all reduced spaces are zero dimensional, or, if the Hamiltonian is not
a functional of the momentum map, they should be two dimensional.
This idea is developed further in Marsden's lectures \cite{Marsden1981},
where even a model of action-angle coordinates are proposed, based on a
local product structure of the symplectic manifold that takes into account
the reduced space. These ideas were vague at the time because the
Marle-Guillemin-Sternberg Normal Form did not exist yet.

An identical definition of non-Abelian integrability as that in Abraham
and Marsden's book \cite{AbMa1978}, both appearing in the same year,
but aimed at linking  non-Abelian to Abelian, i.e., standard Liouville,
integrability, is due to Mishchenko and Fomenko \cite{MiFo1978}: the
action of a Hamiltonian
symmetry group $G$ on a symplectic manifold $M$ is completely integrable
if the dimension of the manifold is the sum of the dimension and the
rank of $G$. By definition, the rank of $G$ is the rank of
$\mathfrak{g}^\ast$, which, in turn is defined to be the dimension of
the minimal coadjoint isotropy subalgebra in $\mathfrak{g}$; by
the Duflo-Vergne Theorem \cite{DuVe1969} the set of points whose coadjoint
isotropy algebra is minimal is Zariski open and the coadjoint isotropy
algebras at all of these points are Abelian. Note that the
Mishchenko-Fomenko condition means that the reduced spaces associated
to every point in this Zariski open set are zero dimensional, which
is exactly the condition imposed in \cite{AbMa1978}.

The link of non-Abelian integrability with Kirwan's convexity theorem is
the following: the inverse image of a coadjoint orbit by the momentum map is a
$G$-orbit in $M$, i.e., the Kirwan polytope can be identified,
in this case, with $M/G$.

In \cite{MiFo1978} it is shown that if a system is integrable in the
non-Abelian sense defined above, then the
system is also integrable in the Abelian sense (i.e., there is
another Abelian Lie algebra of functions relative to the Poisson
bracket whose dimension
is half of the dimension of the manifold) and its trajectories are straight
lines on a torus whose dimension is the rank of the group (which is
strictly smaller than half of the dimension of the manifold); see
also \cite[Subsections 3.1--3.4]{FoTr1988} for a detailed discussion.

Due to the fact that the solutions of a non-Abelian integrable system
take place on tori of dimension strictly smaller than half of the dimension
of the manifold, these systems are also often called super-integrable
systems.

The natural question then arises to classify all compact connected
symplectic manifolds admitting a completely integrable action of a compact
connected Lie group. Iglesias \cite{Iglesias1991} carried out the classification
for rank one groups. Delzant \cite{Delzant1990} did it for
rank two groups: if $G$ is a compact connected Lie group of rank two
acting on a compact connected symplectic manifold in a completely
integrable manner, then the image of the momentum map and the principal
isotropy group classify the manifold up to equivariant symplectic
diffeomorphisms. The classification of multiplicity free
Lie group actions, i.e., non-Abelian completely integrable
actions, of a compact connected Lie group was finished
by Knop \cite{Knop2011} (for previous local results,
see Losev \cite{Losev2009}).

\subsection{Convexity in the presence of focus-focus singularities} \hfill
\label{subsection:review_convexity_with_focus}
As we mentioned earlier, if an integrable Hamiltonian system admits only
elliptic singularities, then it  is a toric system and enjoys convexity
properties. If it also has hyperbolic singularities then it does not make
much sense to speak about convexity of the base space at hyperbolic points
because of the branching at these points.

If the system admits no hyperbolic or degenerate singularities, only
non-degenerate elliptic and focus-focus singularities, then are often called
\textit{almost-toric} in the literature, but we prefer to call them 
\textit{toric-focus} because this terminology emphasizes the
role of focus-focus singularities in these systems. 

A subclass of toric-focus systems are the so-called semitoric
systems, for which one asks that all but one component of the momentum map
yield a periodic flow, plus some properness conditions. 
For semitoric systems on a connected four-manifold,
V\~{u} Ng\d{o}c \cite{San-Polytope2007} obtained a result similar to
Atiyah-Guillemin-Sternberg convexity theorem. More precisely, 
the momentum map $\textbf{F}$ has connected fibers and a simply 
connected image;  the set of critical values of $\textbf{F}$ consists of 
the points on the boundary of its image $\textbf{F}(M)$ plus a finite number $d$
of interior points corresponding to the focus-focus singular fibers; 
there are $2^d$ natural different ways to map the image $\textbf{F}(M)$ 
to $2^d$ different convex "momentum polygons" of the system. The non-uniqueness
of these "momentum polygons" is due to the monodromy, in stark contrast to 
the case of toric systems.

However, as far as we know, convexity properties of general toric-focus integrable
systems have never been treated in a systematic way in the literature before, and
our present paper is the first one to deal with it.

\section{Toric-focus integrable Hamiltonian systems}
\label{section:ToricFocus}

\subsection{Integrable systems} \hfill

We consider {\bfi integrable Hamiltonian system} in the Liouville sense. That
means, for $n$ degrees of freedom a Hamiltonian function $H= H_1$ together with
a so-called {\bfi momentum map} ${\bf H} = (H_1,\hdots, H_n): M^{2n} \to
\mathbb{R}^n$ consisting of $n$ smooth functions on a symplectic manifold
$(M^{2n},\omega)$ which are functionally independent (i.e., ${\rm d}
\mathbf{H}:={\rm d}H_1 \wedge \hdots \wedge {\rm d}H_n \neq 0$ an open dense set) and pairwise commuting
($\{H_i, H_j\} = 0$) under the Poisson bracket.
\textit{Throughout this paper it is assumed that $M^{2n}$ is paracompact,
connected, and that $\mathbf{H}$ is a proper map.}

At each point $p \in M^{2n}$, the number
\[
\rk p := \rk {\rm d}{\bf H} (p):=
\dim \operatorname{span}_\mathbb{R}\{{\rm d}H_1(p), \ldots,
{\rm d}H_n(p)\}
\]
is called the
{\bfi rank} of the system at $p$. If $\rk p < n$ then we say that $p$
is a {\bfi singular point}  of the system,
and ${\rm corank}\ p := n - \rk p$ is called the {\bfi corank} of $p$.
 If $p$ has
rank $0$, i.e. ${\rm d}H_1 (p) = \hdots = {\rm d}H_n(p) = 0$, we say $p$ is a
{\bfi fixed point. }

Given an integrable system, the  partition
\begin{equation*}
\mathcal{F}:=\{\Lambda \text{ connected component of } {\bf H}^{-1}(c),\, c \in {\bf H}(M^{2n}) \},
\end{equation*}
together with the space $\cB := \nicefrac{M^{2n}}{\sim_\mathcal{F}}$ where $x \sim_\mathcal{F} y$
if $x$ and $y$ are in the same element of the partition,
is called the {\bfi associated singular Lagrangian torus fibration} of the system, because the regular fibers are
Lagriangian tori. $\cB$ equipped with the quotient topology is a Hausdorff topological space because $\mathbf{H}$
is proper, and is called the \textit{\textbf{base space}} of the system. A fiber (i.e., an element of the partition 
$\mathcal{F}$) is called \textit{\textbf{regular}} if it does not contain any singular point, otherwise it is called \textit{\textbf{singular}}.

The classical action-angle variables theorem (due to Liouville and Mineur
\cite{Mineur-AA1937}, but often called Arnold-Liouville theorem; see, e.g.,
\cite[Appendix A2]{HoZe1994} for a classical proof and \cite{Zung-AA2017} for a new conceptual approach to it) says that every regular fiber of the associated Lagrangian
torus fibration is indeed a Lagrangian torus which admits a neighborhood $\cU(N)
\cong \mathbb{T}^n \times D^n$ with coordinates $(\theta_1,\hdots, \theta_n,
I_1,\hdots , I_n)$, where $(I_1,\hdots I_n)$ are coordinates on a small disk
$D^n \subset \mathbb{R}$, and $(\theta_1 \pmod 1, \hdots \theta_n \pmod1)$ are
periodic coordinates on the torus $\mathbb{T}^n$, in which the symplectic form
is canonical,
\begin{equation*}
\omega = \sum_{i=1}^n d\theta_i \wedge dI_i,
\end{equation*}
and the components of the
momentum map depend only on the variables $(I_1,\hdots, I_n)$,
\begin{equation*}
H_i = h_i(I_1,\hdots,I_n) .
\end{equation*}
The variables $(I_1,\hdots, I_n)$ are called \textbf{\textit{actions}}
while $(\theta_1, \hdots, \theta_n)$ are called \textbf{\textit{angles}}.

From a geometric point of view, we can forget about the momentum map, and only
look at the associated Lagrangian fibration. So we will adopt the following
notion of geometric integrable Hamiltonian systems \cite{Zung-Integrable2003}:
\textit{a \textbf{\textit{geometric integrable Hamiltonian system}}
is a singular \textbf{\textit{Lagrangian torus fibration}} such that
each fiber admits a saturated  neighborhood (i.e., consisting of whole
fibers) and an integrable Hamiltonian system in this neighborhood
whose associated Lagrangian fibration coincides with the given fibration.}

\subsection{Local normal form of non-degenerate singularities} \hfill
\label{subsection:local_normal_form_nondeg}

In this subsection, we recall necessary definitions and results about
non-degenerate singularities of integrable Hamiltonian systems, which will be
used in the paper.

Recall that a fixed point, i.e., a singular point of rank 0, of the
momentum map ${\bf H} = (H_1,\hdots, H_n): M^{2n} \to \mathbb{R}^n$, is a point
$p \in M^{2n}$ such that
${\rm d}H_1 (p) = \hdots = {\rm d}H_n(p) = 0.$
At a fixed point $p$, the quadratic parts of the functions $H_1, \hdots, H_n$
are well-defined and span an Abelian subalgebra $\mathcal{A}_p$ of the Lie
algebra $Q(T_pM)$ of homogeneous quadratic functions on the tangent space
$T_pM$: the Lie bracket is the Poisson bracket given by the symplectic form at
$p$. The Lie algebra $Q(T_pM)$ is isomorphic to the simple Lie algebra
$\mathfrak{s}\mathfrak{p}(2n,\mathbb{R})$. The fixed point is called
\textbf{\textit{non-degenerate}} if $\mathcal{A}_p$ is a Cartan subalgebra of
$Q(T_pM)$, i.e., it has dimension $n$ and all its elements are semi-simple.

In case $p$ is a singular point of corank $\cok <n$, we may assume, without 
loss of generality, that $dH_1 (p) = \hdots = dH_\kappa(p) = 0$ and 
$dH_{\cok+1} \wedge \hdots \wedge dH_n (p) \neq 0$.
The Hamiltonian vector fields of the functions $H_{\cok+1}, \hdots, H_n$
commute and give rise to a free local $\mathbb{R}^{n-\cok}$-action. The 
quotient of the local coisotropic codimension-$(n-\cok)$ submanifold 
$\{H_{\kappa+1} = const, \hdots, H_n = const \}$ containing $p$
by the orbits of this free local $\mathbb{R}^{n-\cok}$-action is  
the {\bfi local Marsden-Weinstein reduction} of the system (with respect 
to $(H_{\cok+1}, \hdots, H_n)$). This local reduced space is a $2\cok$-dimensional
symplectic manifold with an integrable system given by the momentum map 
$(H_1,\hdots, H_\cok)$, and the image
of $p$ in this local reduced space is a fixed point $p_{reduced}$ for the reduced system. 
We say that $p$ is
\textbf{\textit{non-degenerate of corank $\cok$}} if and only if its image 
$p_{reduced}$ is a non-degenerate fixed point for the
local reduced integrable Hamiltonian system.

According to a classical theorem of Williamson \cite{Williamson1936}, 
at each non-degenerate fixed point $p$, there exist a triple
of nonnegative integers $\Bbbk (p) = (k_e, k_h, k_f)$, called the \textbf{\textit{Williamson type}} of $p$,
$k_e + k_h +  2k_f  =n$ is the corank of $p$,
such that, after the local Marsden-Weinstein reduction, the quadratic part 
of each first integral function whose differential vanishes at $p$ is a linear combination of the  quadratic functions
${\rm e}_i (1 \leq i \leq k_e), {\rm h}_i (1 \leq i \leq k_h), {\rm f}^1_i, {\rm f}^2_i (1 \leq i \leq k_f)$ given by

 \begin{itemize}
  \item ${\rm e}_i = x^2_i + \xi^2_i, \; \forall\;  1 \leq i \leq k_e$   , 

  \item ${\rm h}_i = x_{i+k_e} \xi_{i + k_i}, \; \forall\;  1 \leq i \leq k_h$
  \item $\begin{cases}
           {\rm f}^2_i = x_{2i -1 +k_e +k_h}\xi_{2i + k_e+k_h} + x_{2i +k_e +k_h}\xi_{2i -1 + k_e+k_h} \\
         {\rm f}^1_i = x_{2i -1 +k_e +k_h}\xi_{2i-1 + k_e+k_h} + x_{2i  +k_e +k_h}\xi_{2i + k_e+k_h} 
        \end{cases} \; \forall\;  1 \leq i \leq k_f$.
 \end{itemize}
In addition, $(x_1,\xi_1,\hdots, x_n,\xi_n)$
 is a canonical linear coordinate system on the tangent space at $p$.
The numbers $k_e, k_h$, and $ k_f$ are called the numbers of \textbf{\textit{elliptic, hyperbolic}},
and \textbf{\textit{focus-focus}} components of (the system at) $p$, respectively.

If $p$ is a non-degenerate singular point of corank $\cok < n$ and
$n-\cok > 0$, we define its {\bfi  Williamson type} $\Bbbk (p) = 
(k_e, k_h, k_f)$ to be the numbers of elliptic, hyperbolic,
and focus-focus components of the image of $p$ in the
local Marsden-Weinstein reduction at $p$; in this case we have
$k_e + k_h + 2k_f = \cok$.

We will need the following local symplectic linearization (i.e.,  
normal form) theorem for non-degenerate singular points,
which is due to Vey~\cite{Vey1978} in the analytic case and to Eliasson 
\cite{Eliasson1984, Eliasson1990} in the smooth case 
(see also \cite{Chaperon2013,DuMo-Elliptic1991, Miranda-Thesis2003,
VNWa2013,Zung-Focus2001,Zung-Birkhoff2005}).

\begin{theorem}[Vey-Eliasson \cite{Vey1978, Eliasson1984, Eliasson1990}]
\label{thm:LocalLinearization}
If $p \in M^{2n}$ is a non-degenerate singular point
of rank $n-\cok$ and Williamson type $\Bbbk (p) = (k_e, k_h, k_f)$
of an integrable Hamiltonian system
${\bf H}=(H_1,\,\ldots,H_n): M \rightarrow \mathbb{R}^n$,
then there exists a local symplectic
coordinate system  $(x_1,\, \ldots,x_n,\, \xi_1,\,\ldots,\, \xi_n)$
about $p$, in which $p$ is represented as
$(0,\,\ldots,\, 0)$, such that $\{H_i,\,q_j\}=0$, for all $i,j$,  where

 \begin{itemize}
  \item $q_i = {\rm e}_i = x^2_i + \xi^2_i$ ($1 \leq i \leq k_e$)   are elliptic components,

  \item $q_{k_e+i} = {\rm h}_i = x_{i+k_e} \xi_{i + k_i}$ ($1 \leq i \leq k_h$)
  are hyperbolic components,

  \item $\begin{cases}
      q_{2i -1 +k_e +k_h}  =  {\rm f}^1_i = x_{2i -1 +k_e +k_h}\xi_{2i + k_e+k_h} - x_{2i +k_e +k_h}\xi_{2i -1 + k_e+k_h} \\
       q_{2i +k_e +k_h} =  {\rm f}^2_i = x_{2i -1 +k_e +k_h}\xi_{2i-1 + k_e+k_h} + x_{2i  +k_e +k_h}\xi_{2i + k_e+k_h}
        \end{cases} $ \\
        ($1 \leq i \leq k_f$) are focus-focus components,
    \item  $q_{\cok+i} = \xi_{i}$ ($1 \leq i \leq n-\cok$)
are regular components.
 \end{itemize}
Moreover, if $p$ does not have any hyperbolic component, then
the functions $H_i$ $(i=1,\dots,n)$ can be written as functions
depending only on $q_1,\hdots,q_n$:
$$H_i = g_i(q_1,\hdots,q_n) + H_i(p),$$ 
where $g_1,\hdots,g_n$ are smooth functions such that
$\det (\partial g_i/\partial q_j)_{i=1,\dots,n}^{i=j,\dots,n} (0) \neq 0.$
\end{theorem}

For example, when the symplectic manifold $M$ is of dimension $4$ and the
non-degenerate singular point $p$ has no hyperbolic component, then one
of the following three cases occur:

\begin{itemize}
\item[{\rm (i)}] (elliptic-elliptic singularity: $k_e = 2$, $k_h=k_f = 0$)
$q_1 = (x_1^2 + \xi_1^2)/2$ and  $q_2 = (x_2^2 + \xi_2^2)/2$;
\item[{\rm (ii)}] (focus-focus singularity: $k_e=k_h = 0$, $k_f =1$)
$q_1=x_1\xi_2 - x_2\xi_1$ and
$q_2 =x_1\xi_1+x_2\xi_2$;
\item[{\rm (iii)}] (transversally-elliptic singularity: $k_e = 1$,
$k_h=k_f = 0$) $q_1 = (x_1^2 + \xi_1^2)/2$ and  $q_2 = \xi_2$.
\end{itemize}

If $p$ is a point of highest corank in a
singular fiber of an integrable system, then the orbit $\cO(p)$
of the Poisson $\mathbb{R}^n$-action generated by the momentum map
through $p$ must be compact (otherwise it will contain in its boundary
singular points of even higher corank), and hence is diffeomorphic
to the $(n-\cok)$-dimensional torus $\mathbb{T}^{n-\cok}$,
where $\kappa$ is the corank of $p$. According to a theorem by
Miranda and Zung \cite{MirandaZung-NF2004},
under the non-degeneracy condition of $p$, the system can be linearized
not only at $p$ but also in a neighborhood of the orbit $\cO(p)$.
To formulate this theorem precisely, we need
first to construct the linear model around such an orbit, which is
what we do next.

Take a symplectic manifold of direct product type
\begin{equation*}
V =  D^{2\cok} \times \bbT^{n-\cok} \times D^{n-\cok}
\end{equation*}
with a canonical coordinate system
\[ (x_1,\, \ldots,x_{\cok},\, \xi_1,\,\ldots,\, \xi_{\cok},
\theta_{\cok+1},\hdots, \theta_n , I_{\cok+1},\hdots,I_n)\]
where $(x_1,\, \ldots,x_{\cok},\, \xi_1,\,\ldots,\, \xi_{\cok})$
are canonical coordinates
on the standard symplectic disk $D^{2\kappa}$ and
$(\theta_{\cok+1},\hdots, \theta_n , I_{\cok+1},\hdots,I_n)$ are action-angle
coordinates on $\bbT^{n-\cok} \times D^{n-\cok}$. They look the same as the
coordinate system in the statement of Theorem \ref{thm:LocalLinearization},
except for the
fact that the coordinates $\theta_{\kappa+1} \pmod 1,\hdots,\theta_n \pmod 1$
are not local but periodic coordinates  of period 1 on  $\bbT^{n-\kappa}$.
Take the  functions $q_1,\hdots,q_n$ with the same expression as in
Theorem \ref{thm:LocalLinearization}. We get a direct product linear
model around a compact orbit $\cO(p)$ of Williamson type $(k_e,k_h,k_f)$.

Let $\Gamma$ be a finite group with a  free symplectic action
$\rho:\Gamma\times V \rightarrow V$, preserving the momentum map
$(q_1,\hdots,q_n)$, and is \textbf{\textit{linear}} in the following sense:
{\it $\Gamma$ acts on the product $V = D^{2\cok} \times D^{n-\cok} \times
\bbT^{n-\cok}$ component wise; the action of $\Gamma$ on $D^{n-\cok}$ is
trivial, its action on $\bbT^{n-\cok}$ is by translations relative
to the coordinate system $(\xi_1,\hdots,\xi_{n-\cok})$, and its action on
$D^{2\kappa}$ is linear with respect to the coordinate system
$(x_1,\xi_1,\hdots,x_{\cok},\xi_{\cok})$}.
Then we can form the quotient symplectic manifold
$V/\Gamma$, with an integrable Hamiltonian system on it given by the same
momentum map
$(q_1,\hdots,q_n)$ as before. We call it an \textbf{\textit{almost direct product linear
model}}. It is a direct product model if $\Gamma$ is trivial.

\begin{theorem}[\cite{MirandaZung-NF2004}]
\label{thm:OrbitNormal}
The associated Lagrangian fibration of an integrable Hamiltonian system near
a compact orbit $\cO(p)$ through a non-degenerate singular point $p$ of
Williamson type $(k_e,k_h,k_f)$ is isomorphic to one of
the above almost direct product linear models.
\end{theorem}

\subsection{Semi-local structure of singularities}\label{subsection:semi-local}
 \hfill

A singular fiber $N$ of an integrable Hamiltonian system with proper
momentum map  is called
{\bfi non-degenerate} if its singular points are non-degenerate.
By a {\bfi singularity}, we mean a germ of a saturated open
neighborhood of a single singular fiber $N$ (together with its singular Lagrangian
fibration). We say that a singularity is {\bfi non-degenerate} if the singular
fiber $N$ in question is non-degenerate, plus an additional technical condition
called the {\bfi non-splitting condition}: the bifurcation diagram (i.e.,
the set of singular values of the momentum map) for the system in a
neighborhood of $N$ does not split at $N$, i.e., this bifurcation diagram
coincides (locally) with the bifurcation diagram of the system in a
neighborhood of a singular point of maximal corank in 
$N$\footnote{A smooth map $f: M\rightarrow N$
is said to be \emph{locally trivial} at $n_0 \in f(M)$, if there is an
open neighborhood $U \subset N$ of $n_0$ such that $f^{-1}(n)$ is a
smooth submanifold of $M$ for each $n\in U$ and there is a smooth map
$h: f^{-1}(U)\rightarrow f^{-1}(n_0)$ such that $f\times
h:f^{-1}(U)\rightarrow U \times f ^{-1}(n_0)$ is a diffeomorphism. The
\emph{bifurcation set} $\Sigma_f$ consists of all the points of $N$
where $f$ is not locally trivial. The set of critical values of $f$ is
included in the bifurcation set; if $f$ is proper, this inclusion is an
equality (see \cite[Proposition 4.5.1]{AbMa1978} and the comments
following it).}. This  condition was first introduced by Zung
\cite{Zung-Integrable1996} under a different name; later Bolsinov and Fomenko
\cite[Section 1.8.4 and Section 9.9]{BoFo-IntegrableBook2004} called it
``the non-splitting condition'', a term which is more intuitive and will be adopted in this
paper. All natural examples of integrable systems from mechanics and physics
that we know satisfy this non-splitting condition.

As was shown in \cite{Zung-Integrable1996}, all non-degenerate singular points
of maximal corank in a singular fiber $N$ of an integrable system (even without
the non-splitting condition) have the same Williamson type $\Bbbk = (k_e, k_h,
k_f)$, which is called the {\bfi Williamson type of} the singular
fiber, or of the singularity. The corank $\cok = k_e + k_h +
2k_f$ of singular points of highest corank in $N$ is also called the
{\bfi corank} of the singular fiber, or of the singularity.

{\bfi Elementary} non-degenerate singularities
of integrable systems are those  with
$k_e+k_h+k_f = 1$ and $n-\cok = 0$, i.e., they have rank 0 and
only one component, either elliptic, hyperbolic, or focus-focus.
They live on 2-dimensional (the elliptic and hyperbolic case) and
4-dimensional (the focus-focus case) symplectic manifolds.

In the following theorem, we denote a neighborhood of an elementary
non-degenerate singular fiber $N$ (with some upper index, in dimension
2 or 4), together with the associated Lagrangian foliation, by
$(P^2(N^e), \cL^e)$ in the elliptic case, $(P^2(N^h), \cL^h)$ in the
hyperbolic case,  and $(P^4(N^f), \cL^f)$ in the focus-focus case.
$(D^{n-\cok} \times {\bbT}^{n-\cok}, \cL^r)$ denotes a regular Lagrangian
torus fibration in dimension $2(n-\cok)$ in the neighborhood of a
regular fiber.

\begin{theorem}[Semi-local structure of non-degenerate singularities,
\cite{Zung-Integrable1996}] \label{thm:Semilocal}
Let $N$ be a non-degenerate singularity of corank $\kappa$ and
Williamson type $\Bbbk = (k_e, k_h, k_f)$ in an integrable
Hamiltonian system given by a proper momentum map ${\bf H}: M^{2n} \to
\mathbb{R}^n$ on a symplectic manifold $(M^{2n},\omega)$. Then there
is a neighborhood $\cU(N)$ on $N$ in $M^{2n}$, saturated by the fibers
of the system, with the following properties:

{\rm (i)} There exists an effective Hamiltonian action of
$\mathbb{T}^{k_e +k_f + (n-\cok)}$ on $\cU(N)$ which preserves the system.
This number $k_e +k_f + (n-\cok)$ is the maximal possible. There is a
(non-unique, in general) locally free  $\mathbb{T}^{n-\cok}$-subaction
of this torus action.

{\rm (ii)}  $\cU(N)$ together with the associated Lagrangian torus
fibration is homeomorphic
to the quotient of a direct product of elementary
non-degenerate singularities (2-dimensional
elliptic, 2-dimensional hyperbolic, and/or 4-dimensional focus-focus)
and a regular Lagrangian torus foliation of the type
\begin{multline}
({\mathcal U}({\bbT}^{n-\cok}), \cL^r) \times
(P^2(N^e_{1}), \cL^e_{1}) \times \cdots \times (P^2(N^e_{k_e},
{\mathcal L}^e_{k_e}) \times \\
\times  (P^2_h(N^h_{1}, {\mathcal L}^h_{1}) \times\cdots \times
(P^2(N^h_{k_h}), {\mathcal L}^h_{k_h}) \times
(P^4(N^f_{1}), {\mathcal L}^f_{1}) \times \cdots \times
(P^4(N^f_{k_f}), {\mathcal L}^f_{k_f})
\end{multline}
by a free action of a finite group $\Gamma$ with the following property:
$\Gamma$ acts on the above product
component-wise (i.e., it commutes with the projections onto the
components), it preserves the system, and, moreover, it acts trivially
on elliptic components.
\end{theorem}

\begin{remark}
{\rm Item (ii) of Theorem~\ref{thm:Semilocal} does not hold in the symplectic
category, and not even in the differentiable category as soon as we go beyond
the most simple cases. See for instance~\cite{BoIz2017} for a smooth invariant
of focus-focus singularities with more than one pinch in the differentiable
category, and~\cite{VuNgoc2003},\cite{Wacheux-Asymptotics2015}, \cite{PeTa2018}
for symplectic invariants of focus-focus singularities with one or more pinches. }
\end{remark}

\subsection{Topology and differential structure of the base space} \hfill

Recall that the \textbf{\textit{base space}} of an integrable Hamiltonian the system with a proper
momentum map $\textbf{H}: M^{2n} \to \mathbb{R}^n$
 is a Hausdorff space with the quotient topology from $M^{2n}$, i.e., the weakest topology
for which the projection map $\pi: M^{2n} \to \cB$
of the system is continuous. 
According to Theorem \ref{thm:Semilocal}, under the non-degeneracy
assumption, locally $\cB$ is homeomorphic to an almost direct product of
elementary components, which are either regular, elliptic, hyperbolic, or
focus-focus: each regular component is just an open interval, each
elliptic component is a half-closed interval with one
end point (which corresponds to the elliptic singular fiber), each
hyperbolic component is a bouquet of half-closed intervals (a star-shaped
graph, with the vertex in the center corresponding  to the hyperbolic
singular fiber), each focus-focus component is a 2-dimensional disk whose
center corresponds to the focus-focus fiber.

The phrase ``almost direct'' means that we may have to take the quotient of a
direct product, as described above, of 1-dimensional and 2-dimensional
components by a diagonal action of a finite group $\Gamma$, which acts
non-trivially only on hyperbolic components. (The action of $\Gamma$ on local
regular, elliptic, and focus-focus components of the base space is trivial.)

In particular, if there is no hyperbolic component, then the base space is
locally homeomorphic to a direct product of intervals (maybe half-closed)
and disks, and so it is an $n$-dimensional  topological
manifold (which is in fact a smooth manifold, see Proposition
\ref{prop:BisManifold}), possibly with boundaries and corners.
 The points on $\cB$
which correspond to singularities of
the system with at least one elliptic component $(k_e \geq 1)$ 
are the boundary points of $\cB$, and those points with at least
two elliptic components $(k_e \geq 2)$ lie on the corners of $\cB$. 
When there are hyperbolic singularities, the base space is not a manifold, 
but rather a \textit{branched manifold} in the sense of Williams
\cite{Williams1974}, or more precisely, a topological 
space which is locally an almost-direct product of graphs.

It is not clear what does it mean for a branched manifold
to be convex. Hence, in this paper we will not consider hyperbolic
singularities when talking  about convexity. (However,  in the hyperbolic
case, one can still talk about (quasi) convexity of the closure
of each regular region of the base space and the (quasi) convexity of a
function on $\cB$; see \cite{Zung-Integrable1993}.)
So we adopt the following definition.

\begin{definition}
\label{def:toric-focus}
An integrable Hamiltonian system with non-degenerate singularities
is of \textbf{\textit{toric-focus}} type if its singular points have no
hyperbolic components, only elliptic and/or focus-focus components.
\end{definition}

The base space $\cB$ of an integrable system also admits 
a natural \textit{\textbf{differential structure}} induced from $M^{2n}$: 
a function $f: \cB \to \mathbb{R}$ is called a \textit{\textbf{smooth function}} on $\cB$ 
if its pull-back $f \circ \pi$ is a smooth function on $M^{2n}$. 
It is quite obvious that the set $C^\infty(\cB)$ of all 
so-defined smooth functions on $\cB$ is a differential 
structure in the sense of \cite{Sikorski1967},\cite[Section
2.1]{Sniatycki2013}, i.e. satisfies the following two conditions:

\begin{itemize} 
\item[{\rm (i)}] $C^\infty(\cB)$ is a $\mathcal{C}^\infty$-ring, i.e., we have 
$g \circ (u_1,..,u_N) \in C^\infty(\cB)$
$\forall u_1, ..,u_N \in C^\infty(\cB)$ and 
$ \forall g \in \mathcal{C}^\infty(\RR^N,\RR)$;

\item[{\rm (ii)}] Locality property: 
Let $\tau_\infty$ be the weakest topology
that makes all functions in $C^\infty(\cB)$ 
continuous. If $f$ is a function on $\cB$
such that for all $b \in \cB$
there exists an open set $\cV_b \ni b$ with respect to $\tau_\infty$ and $u_b \in
C^\infty(\cB)$ such that $f|_{\cV_b} = (u_b)|_{\cV_b}$, then $f \in
C^\infty(\cB)$. \end{itemize} 

For a general projection $\pi: M \to \cB$ from a manifold $M$
to a space $\cB$, it might happen that the quotient topology of 
$\cB$ is strictly stronger than the "smooth topology" $\tau_\infty$ of $\cB$ induced from $M$. 
However, in our case of proper integrable Hamiltonian systems with nondegenerate singularities,  
it follows from local normal forms that these
two induced topologies on $\cB$ are the same (see the proof of
proposition \ref{prop:BisManifold}). In fact, moreover, 
as it was observed in \cite[Section 3.6]{Zung-Integrable2003}, for base spaces of integrable systems with nondegenerate singularities, the sheaf of local smooth functions is a fine sheaf (which admits partition functions), and the Poincar\'e lemma and DeRham cohomology still work, just like in the
regular case.

The following fact about toric-focus base spaces has been known for a long time, apparently without anyone bothering to write it down explicitly:

\begin{proposition}
\label{prop:BisManifold}
Let $\cB$ be the base space of a toric-focus integrable Hamiltonian system on a symplectic manifold $M^{2n}$. Then $\cB$ together with its
induced differential structure is a smooth manifold, possibly with boundary and corners, whose manifold
topology is the same as the quotient topology from
$M^{2n}$.
\end{proposition}

\begin{proof}
Take an arbitrary point  $x \in \cB$
and denote by $N_x \subset M^{2n}$ the corresponding fiber of the system, and $p \in N_x$
an arbitrary point on $N_x$. Take a small neighborhood $U(p)$ of $p$ in the symplectic manifold together with a canonical local coordinate system  $(x_1,\, \ldots,x_n,\, \xi_1,\,\ldots,\, \xi_n)$ provided by the Vey-Eliasson local normal form Theorem \ref{thm:LocalLinearization}. In particular,
the quadratic and linear functions $q_1,\hdots,q_n$, whose formulas are given in
Theorem \ref{thm:LocalLinearization}, is a
complete set of first integrals near $p$, and moreover the momentum map ${\bf H} = (H_1,\dots,H_n)$ can be expressed via these functions near $p$:
$$H_i = g_i(q_1,\hdots,q_n) + H_i(p),$$
where $\textbf{g}: (g_1,\hdots,g_n): (\mathbb{R}^n,0) \to (\mathbb{R}^n,0)$
is a smooth local diffeomorphism. The number of
elliptic components of $p$ does not depend on the choice of the point $p$ in $N_x$; let us call this number $e$ ($0 \leq e = k_e \leq n$), so that $q_i = x_i^2 + \xi_i^2$ for $i=1,\hdots, e$.
The image $U(x) = \pi(U(p))$ of $U(p)$ in $\cB$ under the projection map $\pi: M^{2n} \to \cB$
is a small open neighborhood of $x$ (with respect to the quotient topology) because of the 
nondegeneracy condition on the momentum map. The functions 
$q_1,\hdots,q_n$ can be pushed forward to become a map
$$(\pi_*q_1,\hdots\pi_*q_n)|_{U(x)} = \textbf{g}^{-1} \circ (\textbf{H} - \textbf{H}(p))|_{U(x)}: U(x) \to \mathbb{R}^n_{e+},$$ 
which is a homeomorphism from $U(x)$ onto a neighborhood of the origin in the "quadrant"
$$\mathbb{R}^n_{e+} := \{(q_1,\hdots,q_n) \in \mathbb{R}^n\ | \ q_1 \geq 0, \hdots, q_e \geq 0 \}.$$ 
($\mathbb{R}^n_{e+} = \mathbb{R}^n$ if $e=0$ and is a half-space if $e=1$.) We can take these models $(U(x), (\pi_*q_1,\hdots\pi_*q_n)|_{U(x)})$
for smooth local charts on $\cB$. The fact that these charts are smoothly compatible and turn $\cB$ into a smooth manifold, possible with boundary and corners, follows immediately from the fact that the maps $\textbf{g}$ (exchanging between an original momentum map and linear-quadratic momentum maps in canonical local models) are local diffeomorphisms. The fact that the induced smooth functions on $\cB$ are exactly those which are smooth in this smooth manifold structure, and that the smooth topology coincides with the quotient topology, is also clear from the construction.
\end{proof}

The image of a singular fiber of Williamson type
$(k_e,0,k_f)$, with $k_f > 0$, is called a
\textbf{\textit{singular point of type focus power
$k_f$}}, or, more concisely, $\textbf{\textit{focus}}^{k_f}$ point,
on $\cB$. This terminology does not
distinguish the number of elliptic components of the singularities, and as we shall see, only 
focus-focus components influence convexity properties of the base space. Note that, $\cB$
is still a smooth manifold at these focus singular points; it is the \textit{associated affine structure}, discussed in the next section, 
which is singular there. On the other hand,
the image of elliptic singular fibers (without any focus-focus component)  on $\cB$ are considered not as singular points on $\cB$, but rather as boundary or corner points, because the associated affine structure is still regular at these boundary and corner points.

Let us remark that, as was pointed out 
in \cite{Zung-Integrable2003}, toric-focus
integrable systems admit 
\textbf{\textit{local smooth sections}} to the associated Lagrangian fibrations, not only at regular fibers but also at singular fibers, and
the projection maps from these local smooth
sections to the base space are local diffeomorphisms. This fact simplifies the 
theory of characteristic classes, 
classification and integrable surgery  \cite{Zung-Integrable2003} in the case of toric-focus systems. (One can require that a section in the neighborhood of a focus-focus singularity without elliptic components contains only regular points of the system; a local section
near a singular fiber with elliptic components should be understood as the image under the
blown-down map of a blown-up picture in which 
each point with $k_e$ elliptic components 
becomes a $k_e$-dimensional torus.)

\subsection{Integral affine structure on the base space} \hfill

We say that a local function on the base space $\cB$ 
of a proper integrable Hamiltonian system is a local \textbf{\textit{action
function}} if its pull-back to the symplectic manifold generates a
Hamiltonian flow whose time-one map is the identity, i.e., if it can be viewed as the momentum map
of a Hamiltonian $\bbT^1$-action which preserves the system.

By Arnold-Mineur-Liouville theorem on local action-angle variables, near each regular point $x\in \cB$ (corresponding to a regular Lagrangian torus of the system), there is a smooth local
coordinate system consisting of $n$ actions functions $(p_1,\hdots,p_n)$
in a neighborhood of $x$. If $p'$ is another
local action function near $x$, then, up to a constant,
it is a linear combination of the action functions $p_1,\hdots,p_n$ with
integer coefficients:
\begin{equation}
\label{eq:action1}
p' = \sum_{i=1}^n a_ip_i + c,\ \; a_i \in \mathbb{Z},\; c \in \mathbb{R}.
\end{equation}

It follows that the sum or difference of two local action functions is
again a local action function, even near a singular fiber. Indeed, near
each regular fiber the statement is true because of \eqref{eq:action1}.
Near a singular fiber it is also true because of the continuity of the
time-one flow map of a vector field: if the map is the identity outside
the singular fibers, then it is also the identity at a singular fiber.
Thus the sheaf $\cA$ of local action functions on $\cB$ is an Abelian
sheaf which contains the constant functions. (Constant functions
correspond to the trivial torus action.) The quotient $\cA/\mathbb{R}$ of
$\cA$ by constant functions is a sheaf of free Abelian groups called
the sheaf of local \textbf{\textit{action 1-forms}}. Theorem
\ref{thm:Semilocal} yields the following formula for the rank of the
stalk $(\cA/\mathbb{R}) (x)$ of $\cA/\mathbb{R}$
at every point $x \in \cB$ of corank $\cok$ and Williamson type
$(k_e,k_h,k_f)$: $\rk(\cA/\mathbb{R}) (x) = (n-\cok) + k_e + k_f =
n - k_h - k_f$, i.e.
 \begin{equation*}
(\cA/\mathbb{R}) (x) \cong \mathbb{Z}^{n - k_h - k_f}.
 \end{equation*}

Notice that elements of $\cA$ are smooth functions on $\cB$.

If we forget about the focus singular points of $\cB$ and restrict $\cA$
to its regular part $\cB_{reg}$, then $\cA|_\cB$ is a regular 
\textit{\textbf{integral affine structure}} in the usual sense:
local charts (for the manifold structure of $\cB_{reg}$, possibly
with boundary and corners) are given by local action coordinate systems
(near boundary and corner points which correspond to elliptic fibers, 
these action functions are provided in semi-local normal forms 
\cite{MirandaZung-NF2004}), such that each corner of $\cB_{reg}$ is locally polyhedral,
and such that  the transformation maps between different local
action coordinate systems are elements of $GL(n,\mathbb{Z}) \rtimes
\mathbb{R}^n $, i.e., they are  affine maps whose linear part is in the group
$GL(n,\mathbb{Z})$ of linear isomorphisms of the lattice $\mathbb{Z}^n$. 

In our paper, however, we will keep the whole base space $\cB$, and call
$\cA$ the \textit{\textbf{associated singular integral affine structure}} 
on $\cB$. One may take, for example, the following as a general definition of singular affine structures (which can of course be generalized further, or can be made more restrictive): 

\begin{definition}
An Abelian sheaf $\cA$ of smooth functions on a manifold $\cB$ is called a singular (integral) affine structure if there is an open dense subset $\cB_{reg}$ of $\cB$ on which $\cA$ is the sheaf of (integral) affine functions of a regular (integral) affine structure.
\end{definition}

In this paper, we are mainly interested in integral affine structures coming toric-focus 
integrable systems. 

\section{Base spaces and affine manifolds with focus singularities}
\label{sec_ffsingularities}

 \subsection{Monodromy and affine coordinates near elementary focus points} \hfill
\label{subsec_monodromy}
The monodromy (of the Gauss-Manin flat connection on the bundle  of
homology groups of the fibers) of a regular Lagrangian torus fibration
was first introduced in 1980 by Duistermaat \cite{Duistermaat-AA1980}
as one of the obstructions to the existence of global action-angle variables.
This discovery has generated much work, both among mathematicians
and physicists; see, e.g., \cite{CuBa2015} and references therein.

It is now well-known (see \cite{Zung-Focus1997,Zung-Focus2001})
that, each elementary focus-focus singular fiber in dimension 4 is a "\textit{pinched torus}" 
with $k$ pinches, and the monodromy around such a  fiber is non-trivial and is given by the matrix
$
\begin{pmatrix}
1 & k \\ 0 & 1
\end{pmatrix}
$,
where $k\geq 1$ is the number of singular focus-focus points on the singular fiber. 

Concretely, this means the following. Take a simple loop
on the 2-dimensional base space $\cB_{reg}$ with the focus base point $O$ 
(corresponding to the focus-focus singular fiber) lying in
its interior. For a chosen point $c$ on this path, take two appropriate
1-cycles $\gamma_1$ and $\gamma_2$ on the corresponding Lagrangian fiber
$N_c \cong \mathbb{T}^2$, which form a basis of $H_1(N, \mathbb{Z})$.
Move $c$ along this closed path. Then $N_c$ and $\gamma_1, \gamma_2$
move together with it. When $c$ has moved once along this closed path,
coming back to its initial position in $\mathcal{B}$, $N_c$ also
comes back to itself, but the homology classes of $\gamma_1, \gamma_2$
change by the following formula:
 \begin{equation*}
\begin{pmatrix}
[\gamma_1^{new}] \\  [\gamma_2^{new}]
 \end{pmatrix}
 =
\begin{pmatrix}
1 & k \\ 0 & 1
 \end{pmatrix}
 \begin{pmatrix}
[\gamma_1] \\  [\gamma_2]
 \end{pmatrix}.
 \end{equation*}

More precisely, with respect to the basis $\gamma_1$ and $\gamma_2$, the
parallel transport along the fixed loop of the Gauss-Manin connection on the
bundle of first homology groups is given by above matrix.

The number $k$ in the above monodromy formula is also called the \textbf{\textit{index}}
of the focus singularity $O$ on $\cB$. 

For the above formula to hold, one must choose $\gamma_1$ to be the
cycle represented by the orbits of the system-preserving
Hamiltonian $\bbT^1$-action given by Theorem \ref{thm:Semilocal}
and choose appropriate orientations for the path on the base  space and
the cycles.

\begin{figure}[!ht]
\vspace{-25pt}
\centering
 {\mbox{} \hspace*{0cm }\includegraphics[width=1 \textwidth]{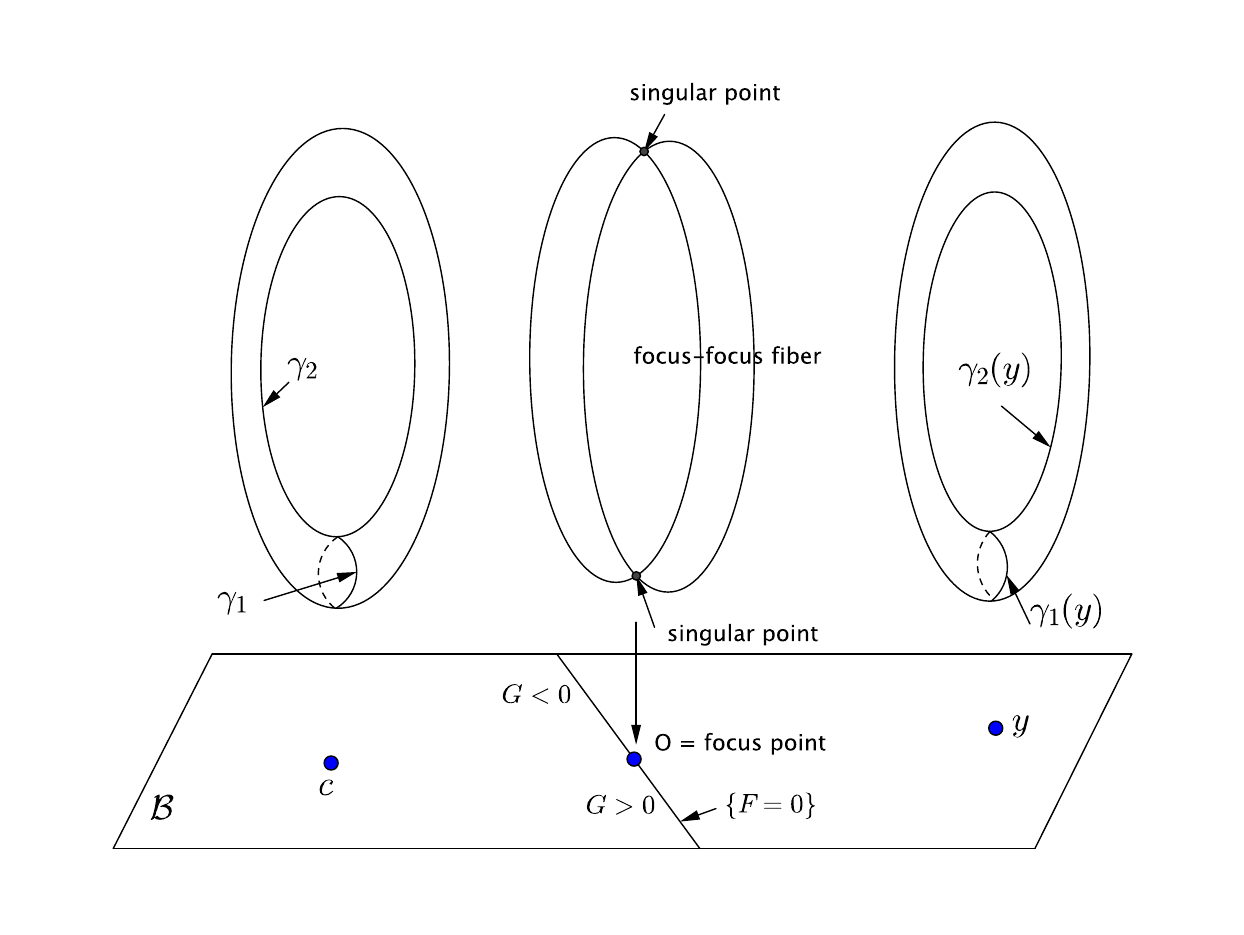}}

\vspace{-15pt}
  \caption{Monodromy around focus-focus singular fiber.}
  \label{fig:Focus1}
\end{figure}

In order to avoid confusion about the orientations,
it is be better to look at the monodromy on action functions or action
1-forms, instead of 1-cycles on the tori; these are essentially the same,
because the action 1-forms generate 1-cycles on the tori.
In terms of action functions, the monodromy around an elementary
focus-focus singularity can be described as follows.

For each point $y \in \cB$, denote by $\gamma_i (y)$ ($i = 1,2$)
the result of parallel transport of $\gamma_i$ from $c$ to $y$
by the Gauss-Manin connection along a simple path on $\cB$ avoiding
the focus point. Let
\begin{equation}
\label{eq:ActionIntegral}
F(y)= \int_{\gamma_1 (y)} \alpha,\quad G(y)= \int_{\gamma_2 (y)} \alpha,
\end{equation}
where $\alpha$ is a primitive 1-form of the symplectic 2-form
$d\alpha = \omega$.
Such a primitive 1-form exists near the focus-focus
singular fiber, because the cohomology class of the symplectic
form in a neighborhood of the focus-focus fiber is 0, due to the
fact that the fiber is Lagrangian (see~\cite{VuNgoc2003,Wacheux-Asymptotics2015}
for a proof of this last fact).
Then $F$ and $G$ are two {\color{red} (local)} action functions. $F$ is
single-valued and is, in fact, a smooth action function on $\cB$ given by
Theorem \ref{thm:Semilocal}. The function $G$ is multi-valued: it depends on the
homotopy class of the path from $c$ to $y$ (See Figure \ref{fig:Focus1} for the
case of 2 pinches). Single-valued action functions are defined only on a
complement of a half-line starting from the focus-focus point, along the
eigendirection of the monodromy matrix. With two of them we have a complete set
of affine charts.

The first homology group with integral coefficients of the focus-focus fiber
is isomorphic to $\mathbb{Z}$. When $c$ tends to the focus point $O$ 
on $\cB$ then $\gamma_1$
tends to 0 while $\gamma_2$ tends to a generator of this homology group
(independently of the path taken), hence $G$ admits a continuous extension to
the focus point $O$ on $\cB$, and we may arrange so that $F(O) = G(O) = 0$ and
choose $c$ such that $F(c) < 0$. The line $\{F=0\}$ 
cuts $\cB$ into two parts and $G$
is single-valued and strictly monotonous on that line.

We choose two different branches of $G$, denoted by $G_l$ and
$G_r$, by specifying the homotopy of the path from $c$ to $y$ in
$\cB \setminus \{O\}$ as follows. For $G_l$, if $F(y) \leq 0$
then the path from $c$ to $y$ lies in the region $\{F \leq 0\}$,
if $F(y) > 0$ then the path from $c$ to $y$ cuts the ray $\{F=0,
G \leq 0\}$ but does not cut the ray $\{F=0, G \geq 0\}$. For
$G_r$, if $F(y) \leq 0$ then the path from $c$ to $y$ lies in the
region $\{F \leq 0\}$, if $F(y) > 0$ then the path from $c$ to
$y$ cuts the ray $\{F=0, G \geq 0\}$ but does not cut the ray
$\{F=0, G \leq 0\}$. 

\begin{definition}
The set of affine functions $(F,G_l,G_r)$ defined as above is called a {\bfi multivalued affine coordinate system } around the focus point $O$.
\end{definition}


The monodromy formula around the focus-focus singularity,
or the focus point on $\cB$, can now be expressed as a relation between the
two branches $G_l$ and $G_r$ of the multi-valued action function $G$:
\begin{equation}
\label{eq:Monodromy1}
G_r = G_l + kF \; \text{when}\;  F > 0;
\quad G_r = G_l  \; \text{when}\;  F \leq 0.
\end{equation}

In turn, Formula \eqref{eq:Monodromy1} can be seen as a special case of the
Duistermaat-Heckman formula \cite{DuHe1982} for the Hamiltonian
$\mathbb{T}^1$-action generated by $F$; see \cite{Zung-Focus2001}.

\begin{figure}[!ht]
\vspace{-15pt}
\centering
 {\mbox{} \hspace*{0cm }\includegraphics[width= 0.7 \textwidth]{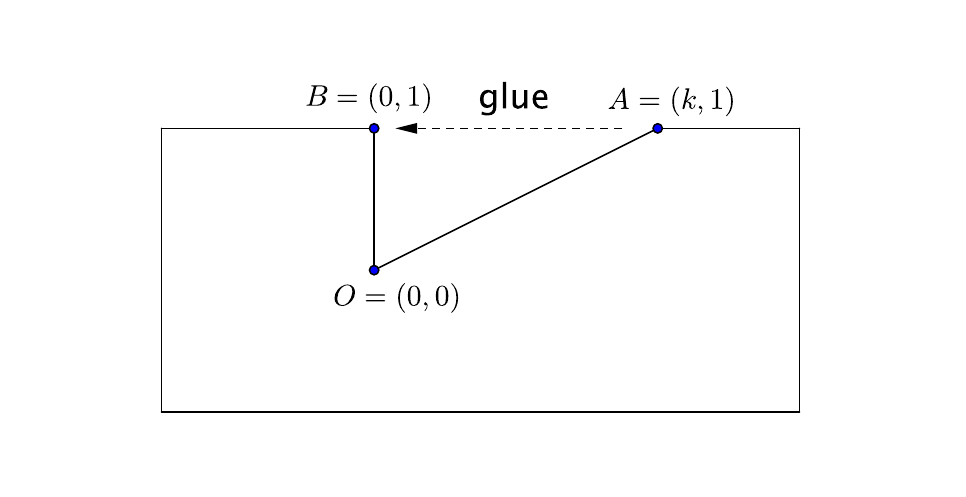}}

\vspace{-15pt}
  \caption{Affine structure near the focus point.}
  \label{fig:Focus2}
\end{figure}

Visually, the singular affine structure near the focus-focus
base point $O$ on $\cB$
can be obtained from a Euclidean plane with a standard coordinate system
by cutting out the angle given by the two rays $OA, OB$,
where $O = (0,0), A = (k,1), B = (0,1)$, and then gluing the ray $OB$
to the ray $OA$ by the transformation matrix $\begin{pmatrix}
1 & k \\ 0 & 1
\end{pmatrix}
 $. (See Figure \ref{fig:Focus2}).

Since we have both a smooth regular differential
structure and a singular integral affine structure on $\cB$,
it is natural to ask how the affine structure behaves in terms
of the smooth differential structure near
the singular focus points. In particular,
how does $G$ behave on the line $\{F=0\}$ with respect to the
smooth differential structure?

The construction we present here that results in Formula
\eqref{eqn:AsymptoticG} and Lemma \ref{lem:AsymptoticG} below are well-known to
people working on symplectic invariants of integrable Hamiltonian systems, and
go back at least to Dufour, Molino, and Toulet \cite{DuMoTo-Hyperbolic1994};
see also V\~u Ng\d{o}c~\cite{San-SemiglobalFF2003} and
Wacheux~\cite{Wacheux-Asymptotics2015} for the focus-focus case, with some
correction given in~\cite{SeVN2018} by V\~u Ng\d{o}c and Sepe.

Recall that we have a smooth coordinate system $(F,H)$ near $O$, where $F$ is
the action function. Their pull backs to the symplectic manifold are given by
the following formulas in a local canonical coordinate system near a focus-focus
point, according to Theorem \ref{thm:LocalLinearization}:
\begin{equation*}
F = x_1y_2 - x_2y_1; \quad H =x_1 y_1 + x_{2} y_{2}.
\end{equation*}

Without loss of generality (using some rescaling if necessary),
we may assume that the point $p: (x_1,y_1,x_2,y_2) = (1,0,0,0)$ is in a
small neighborhood of the focus-focus
point where the normal form is well-defined.
Take a small 2-dimensional disk $D =\{x_1 = 1, x_2 = 0\}$ which intersects
the local plane $\{ y_1=y_2 =0\}$ (which is part of the local focus-focus
fiber) transversally at $p$. On $D$ we have $H=y_1, F=y_2$, which means
that the local smooth structure of $\cB$ near $O$ can be projected from
the smooth local structure of the disk $D$ which intersects all the
fibers near the focus-focus fiber transversally.  This implies, in
particular, that the smooth structure on $\cB$ near $O$ does not depend
on the choice of the singular point on the singular fiber where one takes
 the normal form.

\begin{figure}[!ht]
\vspace{-15pt}
\centering
 {\mbox{} \hspace*{0cm }
 \includegraphics[width=1 \textwidth]{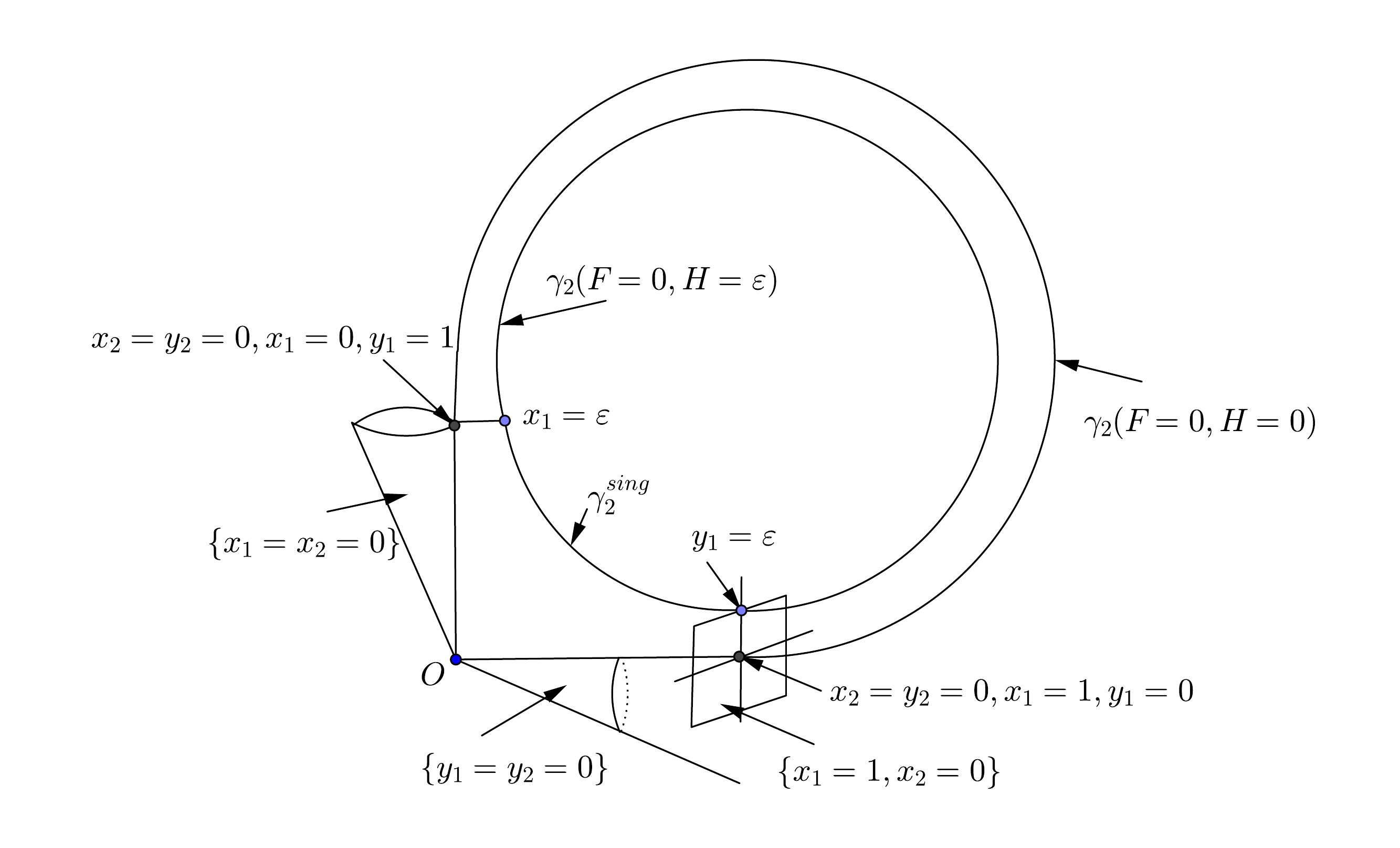}}

\vspace{-15pt}
  \caption{Cycle $\gamma_2$.}
  \label{fig:ActionFunction1}
\end{figure}

We can assume that the 1-form $\alpha$  in formula
\eqref{eq:ActionIntegral} for action functions is given in the
local canonical coordinates $ (x_1,y_1,x_2,y_2)$ by
\begin{equation*}
\alpha = x_1 dy_1 + x_2 dy_2
\end{equation*}
and that the part $\gamma_2^{sing}$
of the 1-cycle $\gamma$ on the Lagrangian torus
$\{F = 0, H = \varepsilon\}$, which
is inside the ball with the canonical coordinate system
$ (x_1,y_1,x_2,y_2)$, is given by the equations
$\{x_2= y_2 = 0,\, x_1 y_1  = \varepsilon = H\}$ and the inequalities
$0 \leq x_1 \leq 1$, $0 \leq y_1 \leq 1$.

If the focus-focus fiber contains only one singular point (see Figure
\ref{fig:ActionFunction1}), then at $F=0$, $H= \varepsilon$ we have:
\begin{align} \label{eqn:AsymptoticG}
G &= \int_{\gamma_2} \alpha = \int_{\gamma_2^{sing}} \alpha
+  \int_{\gamma \setminus \gamma_2^{sing}} \alpha
= \int_{H}^1 (H/y_1) dy_1 + g(H)\\
&= H \log (1/H) + g(H), \nonumber
\end{align}
where $g$ is a smooth function.

The case when there are many ($k > 1$) singular points on the focus-focus fiber
is absolutely similar, except for the fact that there are many ``singular'' pieces
of $\gamma_2$ (each piece near one singular point) which by integration contribute 
to the singular part of $G$ instead of just one. 
For example, assume that there is another focus-focus singular point with
a canonical local normal form in which we have two functions $(F, H')$ such as above
(with the same $F$ as near the first focus-focus point, but with another function $H'$
which is a priori different from $H$). Then $H'$ can be written as a smooth function
of two variables $F$ and $H$: $H' = h'(F,H)$, with $\partial h' / \partial H \neq 0$.
This second focus-focus point contributes the singular term $H' \log (1/H')$ to the
expression of $G$, which is not equal to $H \log (1/H)$ when $F=0$ but which is of the type
$f(H) H \log (1/H) + l(H)$, where $f$ and $l$ are smooth functions and $f(0) > 0$. By summing
up all these singular terms in the case $k > 1$, we have proved the following  lemma
(see V\~{u} Ng\d{o}c \cite{VuNgoc2003}, Pelayo and Tang \cite{PeTa2018}, 
and Bolsinov and Izosimov \cite{BoIz2017} for more
detailed asymptotic formulas and smooth symplectic invariants):

\begin{lemma} \label{lem:AsymptoticG}
With the above notations and assumptions,
$G$ is strictly increasing
with respect to $H$ on the line $\{F=0\}$ on the base space $\cB$,
has real positive derivative $\partial G/\partial H > 0$
on this line, except for the point
$O$ where $\partial G/\partial H = +\infty$.
\end{lemma}

 \subsection{Affine coordinates near focus points in higher dimensions} \hfill

 A singularity with 1 focus-focus component in a
 toric-focus integrable system with $n$ degrees of freedom ($n \geq 3$)
 may be viewed as a parametrized version of elementary focus-focus singularities.

In this case, topologically we have a direct product of an elementary
focus-focus singularity with a regular Lagrangian torus fibration
(i.e., the finite group $\Gamma$ in Theorem \ref{thm:Semilocal}
can be chosen to be trivial). 

On the base space $\cB$  near the focus point $O$,
we have a local smooth coordinate system
$(F, H, L_1,\hdots, L_{n-2})$
which consists of $n-1$ action functions
$L_1 = x_3, \hdots, L_{n-2} = x_n,$
$F = x_1y_2 - x_2y_1$, and one additional first integral
$H =x_1 y_1 + x_{2} y_{2}$, where $(x_1,y_2,\hdots, x_n,y_n)$
is a local canonical coordinate system on the symplectic manifold.

Some of the action functions $L_1,\hdots, L_{n-2}$ may eventually
correspond to the elliptic components of the singularity (if it also has
elliptic components). For our discussions, due to the normal forms, 
these elliptic singular components are harmless; 
they simply create boundary and corners for the base space $\cB$
and $\cB$ becomes a smooth manifold with boundary and corners. 

The set of focus points on $\cB$ near $O$, or equivalently, the
family of focus-focus fibers near the fiber corresponding to $O$,
is smooth $(n-2)$-dimensional and given by the equations $F=H=0$.
The functions $(L_1,\hdots, L_{n-2})$ form a smooth coordinate system
on this local submanifold of focus points on $\cB$.

Similarly to the case with 2 degrees of freedom, there is a multi-valued
action function $G$ in a neighborhood of $O$ in $\cB$. We can
choose, as in Subsection \ref{subsec_monodromy}, two branches
of $G$, denoted by $G_l$ and $G_r$,
by specifying the homotopy of the path from
a point $c\in \cB$, satisfying $F(c) < 0$, to
$y \in \cB \setminus \{F=H=0\}$ as follows. For $G_l$:
if $F(y) \leq 0$, then the path from $c$ to $y$ lies in the
region $\{F \leq 0\}$, whereas if $F(y) > 0$, then the path from
$c$ to $y$ cuts half-line $\{F=0, H \leq 0\}$ but does not cut
the half-line $\{F=0, H \geq 0\}$. For $G_r$:
if $F(y) \leq 0$, then the path from $c$ to $y$ lies in the
region $\{F \leq 0 \}$, whereas if $F(y) > 0$, then the path from
$c$ to $y$ cuts the half-line $\{F=0, H \geq 0\}$ but does not
cut the half-line $\{F=0, H \leq 0\}$.

The monodromy formula around the local codimension two submanifold
of focus points in $\cB$ can be expressed as a relation between the
two branches $G_l$ and $G_r$ of the multi-valued action function $G$,
just as in the case of an elementary focus-focus singularity:
\begin{equation*}
G_r = G_l + kF \; \text{when}\;  F > 0;
\quad G_r = G_l  \; \text{when}\;  F \leq 0.
\end{equation*}

On $\{F=0\}$, the multi-valued action function $G$ is, in fact,
continuous, single-valued, and has the same behavior with
respect to $H$ as in the case of an elementary focus-focus singularity:
$\partial G /\partial H > 0$ when $H \neq 0$, and
$\partial G /\partial H = + \infty$ when $H = 0$.

The local Marsden-Weinstein reduction with respect to the
$\mathbb{R}^{n-2}$-action generated by $(L_1,\hdots, L_{n-2})$
yields a $(n-2)$-dimensional family of elementary focus-focus
singularities, each for one level of $(L_1,\hdots, L_{n-2})$.
It means that we can slice a neighborhood of $O$ in $\cB$
by the functions $(L_1,\hdots, L_{n-2})$ into a $(n-2)$-dimensional
family of 2-dimensional base spaces containing each one a focus point.
The pair $(F,H)$ is a local smooth coordinate system
and the pair $(F,G)$ is a multi-valued system of action coordinates
on each such 2-dimensional base space. The restriction of $G$
to the local $(n-2)$-dimensional submanifold $\{F=H=0\}$ of focus points
in $\cB$ is not constant, in general, but it is a smooth function
in the variables $(L_1,\hdots,L_{n-2})$. The reason is that the family
of 1-cycles $\gamma_2$, as shown in Figure~\ref{fig:ActionFunction1},
consists of  a ``singular part'' $\gamma_2^{sing}$ which does not depend on
$(L_1,\hdots,L_{n-2})$ (with respect to a canonical
system of coordinates) and a ``regular part'' which depends smoothly on
$(L_1,\hdots,L_{n-2})$. We denote the restriction of $G$ to the
submanifold $\{F=H=0\}$ by $G^{critical}(L_1,\hdots,L_{n-2})$
and call it the (function of) \textbf{\textit{critical values}} of $G$.
(This function is only unique up to a constant; one can  put,
for example, $G(O) = 0$ to fix it).

\begin{figure}[!ht]
\vspace{-15pt}
\centering
 {\mbox{} \hspace*{0cm }
 \includegraphics[width=0.75 \textwidth]{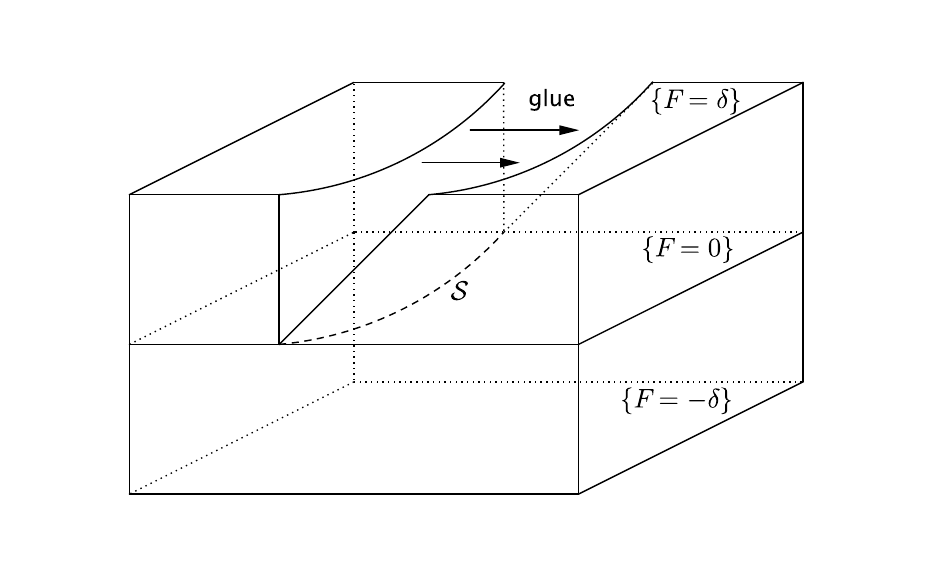}}

\vspace{-15pt}
\caption{Construction of a base space with focus points in higher
dimensions.}
  \label{fig:Focus3}
\end{figure}

To visualize the affine structure of $\cB$ near $O$, one can proceed as
follows. In a neighborhood of the origin in the Euclidean space
$\mathbb{R}^n$ with a local coordinate system $(F,G, L_1,\hdots, L_{n-2})$
(if there are elliptic singular components then just take a corner of this
space, e.g., $L_1, L_2 \geq 0$ if the number of elliptic components is
equal to 2), draw the ``critical'' $(n-2)$-dimensional submanifold
\begin{equation*}
\cS = \{F=0, \, G= G^{critical}(L_1,\hdots,L_{n-2})\}.
 \end{equation*}
Dig out a ``ditch''
\begin{equation*}
\cD = \{F \geq 0 , \, G^{critical}(L_1,\hdots,L_{n-2}) \leq G \leq
G^{critical}(L_1,\hdots,L_{n-2}) + kF\},
 \end{equation*}
which has $\cS$ as its ``sloping bottom'' and glue the two ``walls''
of the ditch together, slice by slice (glue the two edges of each
$\cD \cap \{L_1 = const., \hdots, L_{n-2} = const.\}$ in the same way
as in the case of a 2-dimensional base space). The result is the local
affine model for $\cB$; see Figure \ref{fig:Focus3} for the case $n=3$.
(See also \cite{Wacheux-Local2015}.)
Notice that our ditch is ``curved'': its base $\cS$
is a submanifold but not an affine submanifold, in general. However,
$\cS$ lies on the affine hypersurface $\{F = 0\}$.

 \subsection{Behavior of the affine structure  near $\text{focus}^m$ points} \hfill
\label{subsection:FocusMBehavior}

 Consider a singular point $O$ of Williamson type
 $\Bbbk = (k_e,0,k_f)$ with $2 \leq k_f = m \leq n/2$
 in the base space $\cB$ of a toric-focus  integrable Hamiltonian system
 with $n$ degrees of freedom.

For simplicity, we assume, for the moment, that the singularity
corresponding to $O$ of the Lagrangian fibration
is a topological direct product of elementary singularities, i.e.,
the finite group $\Gamma$ in Theorem \ref{thm:Semilocal}
is trivial.

Similarly to the situation with only one focus component treated in the
previous subsections, Theorems \ref{thm:LocalLinearization},
\ref{thm:OrbitNormal}, and \ref{thm:Semilocal} provide a local smooth
coordinate system $(F_1,H_1,\hdots, F_m,H_m, L_1,\hdots, L_{n-2m})$
and a set of affine real-valued functions \\ 
$(F_1,G_1,\hdots, F_m,G_m, L_1,\hdots, L_{n-2m})$ in a neighborhood
$\cU(O)$ of the focus$^m$ point $O$, with the following
properties:

\begin{itemize}

\item All coordinate functions  vanish at $O$:
\begin{equation*}
F_i (O) = H_i(O) = G_i(O) = L_j (O) = 0 \; \forall \; i=1,\hdots,m; j =1,\hdots, n-2m.
\end{equation*}

\item The set of singular points in $\cU(O)$
is the union of $m$ codimension two submanifolds
\begin{equation*}
\cS_i = \{F_i = H_i = 0\} ,\; i=1, \hdots, m.
\end{equation*}

\item
The fundamental group of the local set of regular points in $\cU(O)$ is
isomorphic to $\mathbb{Z}^m$:
\begin{equation*}
\pi_1 \left(\cU(O) \setminus (\cup_{i=1}^m \cS_i)\right)
\cong \mathbb{Z}^m.
\end{equation*}

\item The linear monodromy representation
 $\rho: \mathbb{Z}^m \to GL(n,\mathbb{Z})$ of $\pi_1 (\cU(O) \setminus (\cup_{i=1}^m \cS_i))$
for the regular integral  affine structure on
$\cU(O) \setminus (\cup_{i=1}^m \cS_i)$, or equivalently,
for the Lagrangian torus fibration over
$\pi_1 (\cU(O) \setminus (\cup_{i=1}^m \cS_i))$,
is generated by the matrices
\begin{equation*}
M_i = I_{2i-2} \oplus \begin{pmatrix}
1 & k_i \\ 0 & 1
\end{pmatrix}
\oplus I_{n-2i}, \quad \forall \, i =1, \hdots, m,
\end{equation*}
where $I_d$ means the identity matrix of size $d \times d$, and $k_i > 0$
is the index of the focus points on $\cS_i \setminus \cup_{j  \neq i} S_j$.

\item If the fiber over $O$ has no elliptic component, then
$\cU(O)$ together with the coordinate functions $(F_i,H_i, L_j)$
looks like an $n$-dimensional cube $]-a, a[^n$ in $\mathbb{R}^n$
for some $a > 0$. If there are $k_e > 0$ elliptic components then
the functions $L_{n-2m-k_e+1}, \hdots, L_{n-2m}$ admit only nonnegative
values, and $\cU(O)$ looks like $1/2^{k_e}$ part
 of a cube
\begin{equation*}
 ]-a, a[^{n-k_e} \times [0,a[^{k_e}.
\end{equation*}

\item For each $i = 1,\hdots , m$, the action function $G_i$ will change to
$G_i + k F_i$ if one goes one full circle around $S_i$ in an appropriate
direction. To be more precise, we can think of $G_i$ as having two branches
$(G_i)_l$ and $(G_i)_r$: $(G_i)_l$ is a smooth action function on
$\cU(O) \setminus \{F_i = 0, H_i \geq 0\}$ with a continuous extension to
the whole $\cU(O)$ but not smooth on $\{F_i = 0, H_i \geq 0\}$;
$(G_i)_r$ is a smooth action function on
$\cU(O) \setminus \{F_i = 0, H_i \geq 0\}$ with a continuous extension to
the whole $\cU(O)$ but not smooth on $\{F_i = 0, H_i \leq 0\}$. The monodromy can be
expressed in terms of the following relations:
\begin{equation*}
\begin{aligned}
(G_i)_r &= (G_i)_l + k_iF_i \; \text{when}\;  F_i > 0\\
(G_i)_r &= (G_i)_l  \; \text{when}\;  F_i \leq 0
\end{aligned}
\qquad \forall\ i =1, \hdots, m.
\end{equation*}

\item For each $i =1, \hdots, m$,  we have
\begin{equation*}
\frac{\partial G_i}{\partial H_i} > 0 \quad \text{is real positive on} \quad \{F_i =0\},
\end{equation*}
except on $S_i$ where  $\frac{\partial G_i}{\partial H_i}  = + \infty$.
\end{itemize}

\begin{definition}
The set of affine functions $(F_1,G_1,\hdots, F_m,G_m,L_1,\hdots,
L_{n-2m})$ defined as above is called a {\bfi multi-valued action
coordinate system} in a neighborhood $\cU(O)$ of the focus$^m$
point $O$.
\end{definition}

In the case when we have only a topological almost direct product, i.e.,
the finite group $\Gamma$ in Theorem \ref{thm:Semilocal} is non-trivial,
its action preserves every fiber of the singular Lagrangian fibration, and
so the base space of the quotient fibration by $\Gamma$ is exactly the
same as the base space of the direct product model. Thus $\cB$ is locally
still the same as in the direct product case. It still has a local smooth
coordinate system  $(F_1,H_1,\hdots, F_m,H_m, L_1,\hdots, L_{n-2m})$
consisting of $n-m$ action functions
$(F_1,\hdots, F_m, L_1,\hdots, L_{n-2m})$ and $m$ additional functions
$(H_1,\hdots, H_m)$; it still has $m$ complementary multi-valued
action functions $G_1,\hdots, G_m$, with the same formulas and
monodromy as above. The only difference is that, unlike the
direct product case, the $n$-tuple
$(dF_1,dG_1,\hdots, dF_m,dG_m, dL_1,\hdots, dL_{n-2m})$
\textit{does not}
generate the whole lattice of action 1-forms at a regular point near
$O$ in general, but only a sub-lattice of finite index.

 For example, consider the almost direct
product
\begin{equation*}
(\cF_1 \times \cF_2)/\mathbb{Z}_2
\end{equation*}
where $\cF_1$ and $\cF_2$ are elementary focus-focus singularities  of index $2$,
$\mathbb{Z}_2 = \mathbb{\mathbb{Z}}/2\mathbb{Z}$ acts freely on each of them
in such a way that $\cF_1/\mathbb{Z}_2$ and $\cF_2/\mathbb{Z}_2$ are elementary
focus-focus singularities  of index 2. For this almost direct product, the function
\begin{equation*}
(G_1 + G_2)/2
\end{equation*}
is a multi-valued action function, though such a function cannot be a multi-valued
action function in the direct product situation.
$(G_1 + G_2)$ is a multi-valued action function in the direct product situation.

\begin{remark}{\rm
V\~u Ng\d{o}c \cite{San-SemiglobalFF2003} and Wacheux
\cite{Wacheux-Asymptotics2015} obtained more refined asymptotic formulas for the
functions $G_i$ in the particular case when $k_i = 1\, \forall\, i$, using the
complex algorithm function. More precisely, they showed that in that case
Formula \eqref{eqn:AsymptoticG} becomes $G = \Re(- (H+ \sqrt{-1} F)\log(H+
\sqrt{-1}F)) + g(H,F)$ (where  $\Re$ means the real part, and $g$ is a smooth
function) (or a similar formula for higher corank), which holds also for $F \neq
0$, whereas Formula  \eqref{eqn:AsymptoticG} only holds for $F =0$. We will not
need these refined formulas in the present paper.
}
\end{remark}

\subsection{Definition of affine structures with focus points}  \hfill

Guided by the properties of the base spaces of toric-focus integrable
systems, we define the following notion of singular integral
affine structures with focus points on manifolds with or without
boundary and corners.

\begin{definition} \label{def:IntegralAffine1}
A (singular) \textbf{integral affine structure with focus points} on a smooth
manifold $\cB$ (possibly with boundary and corners) is a free Abelian
subsheaf $\cA_\mathbb{Z}$ of the sheaf of local smooth functions on $\cB$, 
called the sheaf of local \textbf{integral affine functions},
which satisfies the following properties:

{\rm(i)} $\cA$ contains all constant functions. We will denote by 
$\cA_\mathbb{Z}/\mathbb{R}$
the quotient of $\cA_\mathbb{Z}$ by the constant functions 
and call it the sheaf of local \textbf{integral affine 1-forms} on $\cB$

{\rm(ii)} For $x$ in a open dense set in $\cB$, called the set of 
\textbf{regular points}, $\cA_\mathbb{Z} (x)/\mathbb{R} \cong \mathbb{Z}^n$ 
(where $\cA_\mathbb{Z} (x)$ denotes the stalk
of $\cA_\mathbb{Z}$ at $x$), and a basis of $\cA_\mathbb{Z} (x)$ forms a 
local coordinate system on $\cB$, called a \textbf{local integral affine 
coordinate system}.

{\rm(iii)} If $x \in \cB$ is singular (i.e., non-regular), then 
$\rk \cA_\mathbb{Z} (x) < n$. The set $\cF_m = \{ x \in B\  |\  
\rk \cA_\mathbb{Z} (x) = n-m\}$ is a smooth submanifold
of codimension $2\cok$ in $\cB$ (whose intersection with the boundary and 
corners of $\cB$ is transversal).  Each point $x \in \cF_m$ is called a 
\textbf{singular focus point of corank $m$}, or  \textbf{of type focus 
power $m$}, or an \textbf{$F^m$-point} for short. 

{\rm(iv)} For every point $x \in \cF_m$, there is a small neighborhood 
$V$ of $x$ and $\cok$ codimension two submanifolds $S_1,\hdots S_m$,
with transversal intersections, such that $\cF_k \cap V = \cap_{i=1}^m S_i$.
Moreover, there is a family of functions 
$$
(F_1, (G^l)_l, (G^1)_r, \hdots,  F_m, (G_m)_l, (G_m)_r, L_1,\hdots, L_{n-2m})$$
in $V$, called a \textbf{multi-valued system of integral affine coordinates} near $x$, 
which satisfies the following conditions:

\begin{itemize}
\item $F_1,\hdots, F_k, L_1, \hdots, L_{n-2m} \in \cA_\mathbb{Z}(V)$ and 
together with some other smooth functions $H_1,\hdots H_m$ form a smooth 
coordinate system of $V$.

\item $\cS_i = \{F_i = H_i = 0\}$ for every $i = 1,\hdots, m$.

\item $(G_i)_l$ is continuous on $V$ and $(G_i)_l |_{V \setminus \{F_i = 0, H_i \geq 0\}} 
\in  \cA_\mathbb{Z}(V \setminus \{F_i = 0, H_i \geq 0\})$. 

\item
$(G_i)_r$ is continuous on $V$ and $(G_i)_r |_{V \setminus \{F_i = 0, H_i \leq 0\}} \in  
\cA_\mathbb{Z}(V \setminus \{F_i = 0, H_i \leq 0\}).$ 

\item $(G_i)_r = (G_i)_l$ when $F_i \leq 0$ and 
$(G_i)_r = (G_i)_l + k_iF_i$
when $F_i \geq 0$, for some positive constant $k_i \in \mathbb{Z}_+$. 
Moreover, $\left(\partial (G_i)_l/\partial H_i\right) (y)  = 
\left(\partial (G_i)_r/\partial H_i\right) (y) > 0$
with respect to the coordinate system $(F_1, H_1, \hdots, F_m,H_m, 
L_1,\hdots, L_{n-2m})$  
for every $y \notin \cS_i$.
\end{itemize}
\end{definition}

Of course, if  $\cB$ is the base space of a toric-focus integrable Hamiltonian
system, then the sheaf of local action function on $\cB$ is an integral affine
structure with focus singularities. Local integral affine functions on $\cB$ are
those functions whose pull back on $M^{2n}$ give rise to Hamiltonian vector
fields whose time-one flow is the identity map, i.e., they are generators of
$\mathbb{T}^1$-Hamiltonian actions whose images preserve the system. We do not
consider the points of $\cB$ corresponding to elliptic singularities of the
system as singular points of $\cB$, but rather as regular points lying on the
boundary and corners.

We also define {\bfi affine structures with focus points} on a manifold $\cB$
(without the adjective \textit{integral}) by simply removing all the words
``integral'' and ``$\mathbb{Z}$'' from Definition \ref{def:IntegralAffine1}, and
by considering the sheaf $\cA$ of all local affine functions (instead of just
integral affine functions). At a regular point, the stalk of $\cA/\mathbb{R}$ is
isomorphic to $\mathbb{R}^n$ instead of $\mathbb{Z}^n$.  In the monodromy
formulas $(G_i)_r = (G_i)_l + k_iF_i$, the numbers $k_i$ are still positive but
are not required to be integers. 

Of course, an integral affine structure with focus point is a special case of
affine structures with focus points: in that case, the relation between $\cA$
and $\cA_\mathbb{Z}$ is $\cA = \cA_\mathbb{Z} \otimes \mathbb{R}$.

There are many ways to modify the above definitions to extend them to more general
situations by weakening or changing some conditions. For example, one may require 
$\cB$ to be a more general differential space instead of a manifold, and/or define
the affine structure just outside a subset of $\cB$, etc. 
See, for example, \cite{CaMa_Lagrangian2009,GrSi-AffineComplex2011}.
However, we will stick with the above definition, 
because it comes from toric-focus integrable systems and is convenient for our study of convexity. 

From now on, when we talk about a singular affine manifold $\cB$, we always mean
an affine manifold with focus singularities in the sense of the above definitions.

\begin{remark}
{\rm
The direct product of two manifolds $\cB_1$ and $\cB_2$ with singular affine
structures is of course again a manifold with a singular affine structure
which is the "direct product" of the two affine structures in question, but
there are in general (infinitely) many non-isomorphic singular affine structures on $\cB_1 \times \cB_2$ with the same singular set and indices as this product affine structure.
}
\end{remark}

\section{Straight lines and convexity}
\label{sec_straight_lines_convexity}

\subsection{Regular and singular straight lines} \hfill

\emph{Affine lines}, a.k.a. \emph{straight lines}, in the regular part of a singular affine manifold 
$\cB$ can be defined in an obvious way:  a curve $\gamma$  lying entirely
in the regular part of a singular affine manifold $\cB$ is called an
{\bfi affine line} or {\bfi straight line}, if it is locally affine in every local affine chart of the
regular part of $\cB$. In other words, for any two local affine functions there is a non-trivial
affine combination of them which vanishes on $\gamma$. Here, 
by a (parametrized) curve in $\cB$, we mean a continuous map from an interval
$I \subset \mathbb{R}$ (which can be bounded or unbounded, 
with or without end points) to $\cB$.

Any non-trivial regular straight line can be naturally (affinely)
parametrized by local affine functions: the parametrization is unique
up to transformations of the type $t \mapsto a t +b$ ($a$ and $b$ are constants, $a \neq 0$).

We also want to study straight lines which contain singular points of
$\cB$.  They will be called \textit{singular straight lines}, or
\textit{singular affine lines}.
We can define them  by using limits of regular straight lines.

\begin{definition}
A non-constant parametrized continuous curve $\gamma: I \to \cB$ in a singular affine manifold
 $\cB$ ($I$ is an interval in $\mathbb{R}$), which contains singular points of $\cB$, is called a
\textbf{singular straight line} or \textbf{singular affine line}  if
for every point $t_0$ in the interior of $I$ there exists a subinterval neighborhood  $[s,t] \subset I$
of $t_0$ in $I$, and
a sequence of affinely parametrized regular straight lines $\gamma_n: [s,t] \to \cB$
which converges to $\gamma|_{[s,t]}$ (in the standard compact-open topology).
\end{definition}

Of course, the regular part of a singular straight line $\gamma$ (i.e.,
$\gamma$ minus singular points of $\cB$), if not empty, consists of
regular straight lines, because in the regular region, the limit of a
sequence of straight lines is again a straight line.

\subsection{Singular straight lines in  dimension 2 and branched extension} \hfill

\label{subsection:StraightFFa}

We will see that, unlike regular affine lines, singular affine lines exhibit
unusual behavior at singular points. To understand their nature, let us first
consider the simplest case of a single focus point $O$ in a 2-dimensional
singular affine manifold $\cB^2$.

Let $F,G$ be a multi-valued affine coordinate system near $O$ in $\cB^2$:
$F(O) = G(O) = 0$, $F$ is single valued, $G$ is multivalued.  $G$ has two
branches $G_l$ and $G_r$ (extensions of an affine function $G$ from the
region $\{F \leq 0 \}$ to the region $\{F > 0\}$,
from the left or from the right of the singular point $O$), with the
monodromy formula
 \begin{equation*}
G_r = G_l + kF\quad \text{when}\quad F > 0\quad \text{and}\quad
 G_r = G_l\quad \text{when}
\quad F \leq 0,
 \end{equation*}
where $k > 0$ is the index of the focus point $O$.

Let $\gamma: I \to \cB^2$ be a parametrized singular straight line in
$\cB^2$, with $\gamma(t_0) = O$ for some $t_0 \in I$. Let $\gamma_n: I \to \cB^2$ be parametrized regular
straight lines in $\cB^2$ which converge to $\gamma$ when $n\to \infty$.
Since $F$ is single-valued and is affine on every $\gamma_n$ with
respect to the parameter $t \in I$, it follows that $F$ is also affine
on $\gamma$ with respect with to the parameter $t \in I$. Moreover,
$F(t_0) = F(O) = 0$ on $\gamma$. We distinguish two possible situations.

1) $F$ is identically zero on $\gamma$, i.e., $\gamma(I)  \subset
\{F = 0\}$. On  $\{F = 0\}$ the function $G$ is single-valued and becomes
an affine parametrization for $\gamma$, i.e., $G$ is affine with respect
to $t$ on $\gamma$. The curves $\{F = c\}$ (where $c \neq 0$ is a
constant) are straight lines affinely parametrized by $G$ (or $G_l$
or $G_r$ when $c  > 0$) and they converge to the singular affine curve
$\{F = 0\}$ (also affinely parametrized by $G$) when $c$ tends to $0$.

2) $F$ is not identically on  punctured open neighborhood of $t_0$. Then we can take $F$ as an
affine (re)parametrization of $\gamma$. Let $\gamma_n$ be a family of
straight lines which converges to $\gamma$ when $n$ goes to infinity.
Without losing generality, we may assume that $F(\gamma_n(t)) =
F(\gamma(t))$ for every $n \in \mathbb{N}$ and every $t \in I$.
Put $X_n = \gamma_n(t_0)$. Then $X_n \neq O$, but the
sequence of points $(X_n)$ lies on the straight line $\{F = 0\}$ and
converges to the singular point $O$ when $n \to \infty$. Put
$c_n = G(X_n)$, then $c_n \neq 0$ and $c_n
\stackrel{n \to \infty}{\longrightarrow} 0$ (remind that $G$ is well-defined on the submanifold $\{ F=0 \}$. By taking a subsequence
of $(\gamma_n)$, if necessary, we may assume that either $c_n < 0$
for all $n$ or $c_n > 0$ for all $n$.

\begin{figure}[!ht]
\vspace{-15pt}
\centering
 {\mbox{} \hspace*{-1.8cm }\includegraphics[width=1.2 \textwidth]{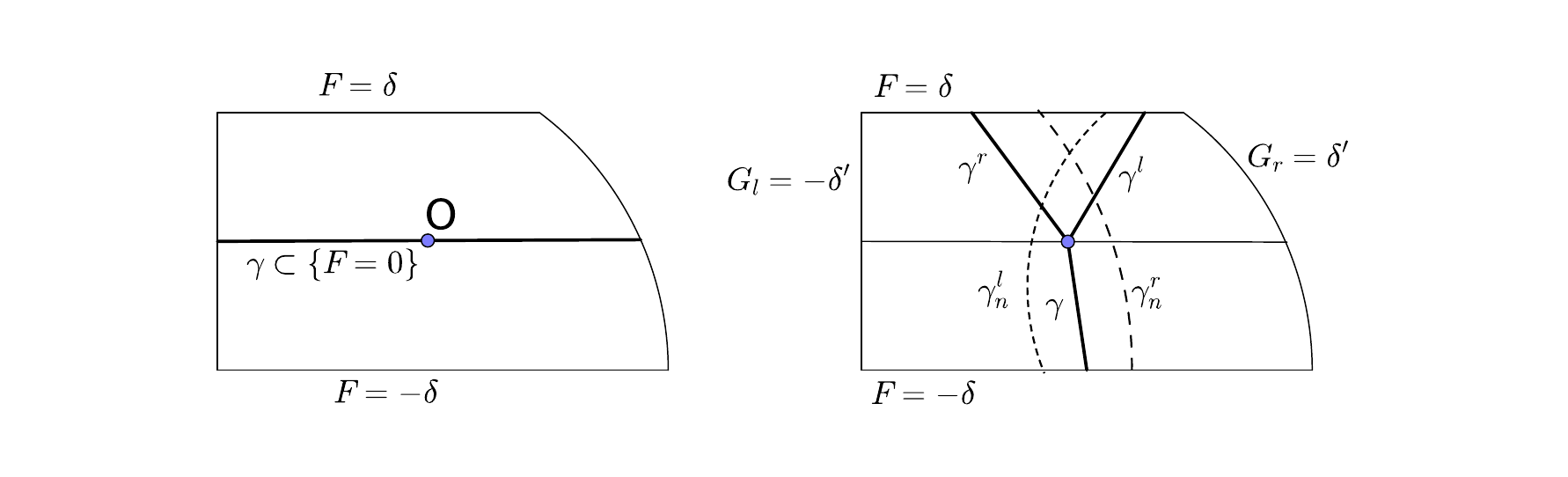}}

\vspace{-15pt}
  \caption{Extension of a straight line through a focus point.}
\end{figure}

If $c_n < 0$ for all $n$, then we say that the straight lines
$(\gamma_n)$ tend to $\gamma$ {\bfi from the left}, and if  $c_n > 0$
for all $n$, then we say that the straight lines $(\gamma_n)$ tend
to $\gamma$ {\bfi from the right} (with respect to the singular
multi-valued affine coordinate system $(F,G)$).

Assume now that $t_0$ lies in the interior of the interval $I$, and cut $I$
by $t_0$ into two parts: $I_1 = I \cap\, ]-\infty, t_0]$ and
$I_1 = I \cap [t_0, +\infty[$. We will say that $\gamma|_{I_1}$ is a
straight line coming from the region $\{F < 0\}$ which {\bfi hits $O$},
and $\gamma|_{I_2}$ is a {\bfi straight extension of $\gamma|_{I_1}$
after hitting $O$}.

 Our main observation in this subsection is the following:

\begin{proposition}
\label{prop:extension1}
In a neighborhood of a focus point $O$, every straight line hitting
$O$ and not lying on $\{F = 0\}$ has exactly two different straight
extensions after hitting $O$.
\end{proposition}

\begin{remark}{\rm
The straight line has only a unique straight extension if it lies
on $\{F = 0\}$}.
\end{remark}

\begin{proof}
Indeed, without loss of generality,  working locally,
we may assume for simplicity that $I =  [- \delta , \delta]$ for some
small $\delta > 0$, $t_0 =0$, $F(\gamma(t)) = t$ and $G(\gamma(t)) = at$
for some constant $a \in \mathbb{R}$ on $\gamma|_{I_1}$. We can define
straight lines $\gamma^l_n: I \to \cB^2$ and $\gamma^r_n: I \to \cB^2$
by $F(\gamma^l_n (t)) = t, G_l(\gamma^l_n (t)) = at - \varepsilon_n$ and
$F(\gamma^r_n (t)) = t, G_r(\gamma^r_n (t)) = at + \varepsilon_n$, where
$\varepsilon_n$ are small positive numbers which tend to $0$ when $n$
goes to infinity.

Then $(\gamma^l_n)$ tend to a singular straight line $\gamma^l$ defined by
$$
F(\gamma^l_n (t)) = t, \quad G_l(\gamma^l (t)) = at
$$
(for every $t \in I = [- \delta , \delta]$), while $(\gamma^r_n)$
tend to a singular straight line $\gamma^r$ defined by
$$
F(\gamma^l_n (t)) = t, \quad G_r(\gamma^l (t)) = at.
$$

Recall that when $F \leq 0$, then $G_l = G_r = G$ are the same function,
but when $F > 0$, then $G_l$ and $G_r$ are two different functions
(two branches of $G$) related by the monodromy formula $G_r = G_l + kF$.
Because of that, $\gamma^l |_{I_1} = \gamma^r |_{I_1} = \gamma$, where
$\gamma: I_1 \to \cB$ is a straight line that hits $O$, but
$\gamma^l |_{I_2} \neq \gamma^r |_{I_2}$.

It is easy to see that any sequence of regular straight lines converging to
a singular straight line which contains $O$ and which does not lie on 
$\{F=0\}$ must contain a subsequence such that either all of its elements "lie
on the left of $O$" or all of its elements "lie on the right of $O$"; in the "left"
case the limit will be $\gamma^l$ and in the right case the limit will be 
$\gamma^r$. So it means that we have exactly two affine extensions
$\gamma^l$ and $\gamma^r$. 
\end{proof}

We will say that $\gamma^l$ is the (local) {\bfi extension from the left}
of $\gamma$ after hitting $O$, while  $\gamma^r$ is the (local)
{\bfi extension from the right} of $\gamma$; together they form the
{\bfi branched, double-valued extension} of $\gamma$.

The difference between $\gamma^l$ and $\gamma^r$ in terms of $G$ is as
follows:
$$
G_l(\gamma^l (t)) - G_l(\gamma^r(t)) =
G_r(\gamma^l (t)) - G_r(\gamma^r(t)) = k F(t),
$$
when $F(t) > 0$. The difference is 0 when $F(t) \leq 0$. Geometrically,
it means that \textit{the extension from the left lies on the right of
the extension from the right}.

\subsection{Straight lines in dimension $n$ near a focus point} \hfill

\label{subsection:StraightFFb}

Let $O$ be a $F^1$ singular point (i.e., with just one focus
component) in a singular $n$-dimensional affine manifold $\cB$,
with $n \geq 3$. ($O$ can lie on the boundary of $\cB$.) Near $O$,
we have a local  multi-valued affine coordinate system
$(G,F, L_1,\hdots,L_{n-2})$, where $F, L_1,\hdots,L_{n-2}$ are
single valued affine functions, $F$ is the ``angular momentum'' for
the focus-focus singularities,
$F(O) = 0$, $G$ is single valued when $F \leq 0$ and admits a double-valued affine extension, denoted by
$G_l$ and $G_r$, to the region $F > 0$, which are related to each other by the same formula as in the 2-dimensional case:
 \begin{equation*}
 G_r = G_l + kF\;
\text{when}\; F > 0, \quad G_l=G_r =G\; \text{when}\; F \leq 0,
 \end{equation*}
where $k$ is some positive constant.

The set of all singular points of $\cB$ near $O$ is a local
$(n-2)$-dimensional hypersurface $\cS$ containing $O$, lying in the local
affine hyperplane  $\{F = 0\}$ of $\cB$, on which $(L_1,\hdots,L_{n-2})$
is a regular local coordinate system.  We  consider
$G_{critical} = G |_\cS$ as a function of $(L_1,\hdots,L_{n-2})$,
the {\bfi function of critical values} of $G$ near $O$.
(This function is smooth but not constant in general.)
$G_l$ (respectively, $G_r$)  is obtained from $G$ by affine extension
from $F \leq 0$ to $F > 0$ via the paths ``on the left'' (respectively,
``on the right'') of the critical set $\cS$, i.e., paths which cut the
subspace $\{F = 0\}$ at points whose value of $G$ is less than
(respectively, greater than) the critical value of $G$ for the same
level of $(L_1, \hdots, L_{n-2})$.

Similarly to the 2-dimensional situation, for a local singular
parametrized straight line $\gamma$ in dimension $n$
which contains $O$, we can distinguish two cases (see Figure \ref{fig_2cases}).

\underline{Case 1}. $F$ is identically zero on the line, i.e., the
straight line lies on the hypersurface $\{F = 0\}$. On this hypersurface
the function $G$ is single valued and, together with
$(L_1,\hdots,L_{n-2})$, form a local single-valued affine coordinate
system on $\{F = 0\}$. In other words, if we restrict our attention to
$\{F = 0\}$, then we can forget about the singular points; $\{F = 0\}$
admits a regular affine structure compatible with the singular affine
structure on $\cB$, and (singular) affine straight lines on  $\{F = 0\}$
are simply straight lines with respect to the regular affine structure
on it. Such a straight line can cut the singular set $S \subset \{F = 0\}$
at one, or many, or an infinite number of points, but there is no
branching. Such a (local) straight line  can, of course, be constructed
as a limit of (local) regular straight lines in $\cB$:
to generate a regular straight line $\gamma_c$, where $c \neq 0$ is some
small constant, just keep the same values of $G, L_1,\hdots,L_{n-2}$
but change the value of $F$ from $0$ to $c$. Then
$\displaystyle \lim_{c \to 0} \gamma_c = \gamma$ is a singular straight
line.

\begin{figure}[!ht]
\centering
{\mbox{} \hspace*{-0cm }\includegraphics[width=1 \textwidth]{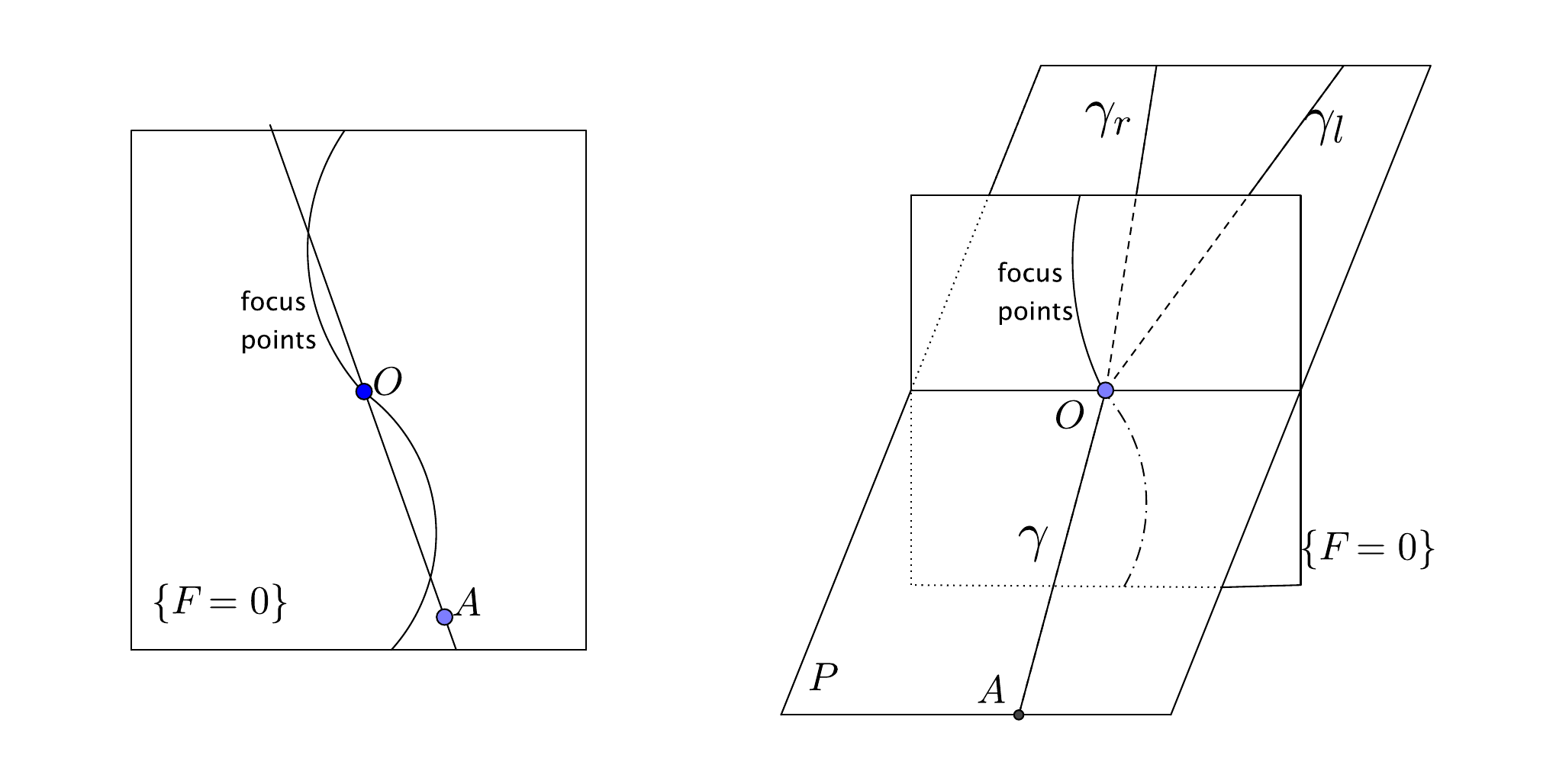}}

\vspace{-15pt}
\caption{The two cases of straight lines going through focus points.}
\label{fig_2cases}
\end{figure}

\underline{Case 2}. $F$ is not identically zero on the line. Then we can
assume that $\gamma$ is  parametrized by $F$, i.e.,
$F(\gamma(t)) = t$ for every $t \in I = [-\delta,\delta]$. By passing
affine equations from regular straight line to our singular
straight line  via the limit, we see that for each $j =1,2,\hdots n-2$,
there exists  two real numbers $a_j, b_j$ such that
 \begin{equation*}
a_jF + L_j = b_j
 \end{equation*}
identically on $\gamma$.
Define the following local 2-dimensional subspace
$P$ of $\cB$:
 \begin{equation*}
 P = \{x \in \cB \; |  \; a_jF(x) + L_j(x) = b_j, \;
 \forall\,
 j = 1,\hdots,n-2 \}.
 \end{equation*}
$P$ is given by $n-2$ independent linear equations and it intersects
the  $(n-2)$-submanifold $\cS$ of focus points transversally
at $O$. $P$ inherits from $\cB$ the structure of a local affine manifold
which contains $O$ as the only singular focus point.
The singular straight line $\gamma$ lies in $P$. By doing this reduction,
we fall back to the 2-dimensional case with a focus singular point.
So, similarly to the 2-dimensional case, $\gamma |_{[-\delta,0]}$ is a
straight line which hits a singular point $O$ and which admits
exactly two different straight line extensions $\gamma^l$ and $\gamma^r$
after hitting $O$. The difference between
$\gamma^l$ and $\gamma^r$ in terms of the local multi-valued affine
function $G$ is the same as in the 2-dimensional case:
 \begin{equation*}
 G_l(\gamma^l (t)) - G_l(\gamma^r(t)) =
 G_r(\gamma^l (t)) - G_r(\gamma^r(t)) = k F(t)
 \end{equation*}
when $F(t) > 0$. The difference is 0 when $F(t) \leq 0$.

\subsection{Straight lines  near a  $\text{focus}^m$ point} \hfill

The case when a singular straight line $\gamma: [-\delta, \delta] \to \cB$
contains a $F^m$ point $O$ of type (focus power $m$)
is similar.

We have a local multi-valued affine coordinate system
 \begin{equation*}
 (F_1,G_1,\hdots, F_m,G_m, L_1,\hdots,L_{n-2m})
 \end{equation*}
centered at $O$ ($F_i(O) = G_i(O) = L_j(O) = 0$),   such
that $F_1, \hdots, F_m,L_1,\hdots, L_{n-2m}$ are single-valued,
$G_i$ is single-valued when $F_i \leq 0$ and is double-valued
with branches $(G_i)_l$ and $(G_i)_r$ when $F_i > 0$ for each
$i =1,2,\hdots,m$. The relation between $(G_i)_l$ and $(G_i)_r$ is
$(G_i)_r = (G_i)_l + k_iF_i$  when $F_i > 0$ for some positive
constants $k_i >0$. (These numbers $k_i$ are
the monodromy indices of the singular point $O$.)
We also have a local smooth coordinate system
 \begin{equation*}
 (F_1,H_1,\hdots, F_m,H_m, L_1,\hdots,L_{n-2m}).
 \end{equation*}
 The local singular set on $\cB$ is $\cup_{i=1}^m \cS_i$ where $\cS_i =\{F_i=H_i=0\}.$

In the generic case, when none of the functions $F_1,\hdots, F_m$ is
identically zero on $\gamma$,  then $\gamma$ intersects the set of
singular points of $\cB$ only at $O$, and the straight line
$\gamma |_{[-\delta,0]}$ admits exactly $2^m$ different straight
line extensions after hitting $O$ (see Figure \ref{fig4branches}).

\begin{figure}[!ht]
\vspace{-15pt}
\centering
 {\mbox{} \hspace*{-1cm}\includegraphics[width=1.1 \textwidth]{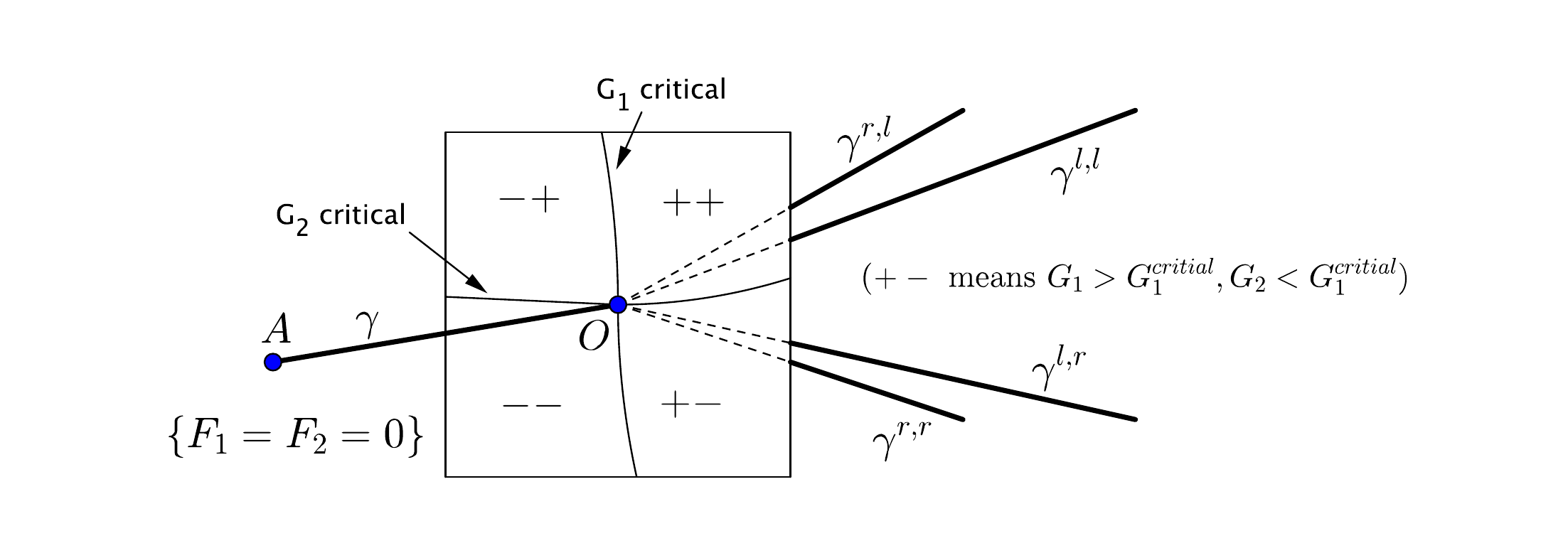}}

\vspace{-15pt}
  \caption{Four branches of straight extension through a $\text{focus}^2$ point $O$.}
  \label{fig4branches}
\end{figure}

Without loss of generality, we may assume that $\gamma(-\delta) = A$
is a point such that $F_1(A) < 0, \hdots, F_m(A) < 0$. For each
multi-index ${\bf d} = (d_1,\hdots,d_m)$, where each $d_i$ is either
$l$ (left, minus) or $r$ (right, plus), and each small positive number $c$,
we can define a regular straight line
$\gamma_{c, {\bf d}}: [-\delta,\delta] \to \cB$ near $O$
which satisfy the following conditions for every $t  \in [-\delta,0]$:
 \begin{equation*}
L_i (\gamma_{c, {\bf d}}(t)) =  L_i(\gamma(t)), \quad \forall \
i=1,\hdots, n-2m
\end{equation*}
and
\begin{equation*}
F_i (\gamma_{c, {\bf d}}(t)) =  F_i(\gamma(t)),\quad
G^{d_i}_i (\gamma_{c, {\bf d}}(t)) =  G_i(\gamma(t)) + G_i(p_{c, {\bf d}}),
\quad \forall \ i=1,\hdots, m,
 \end{equation*}
where $p_{c, {\bf d}}$ is the point with coordinates $F_1 (p_{c, {\bf d}}) =\hdots=F_m (p_{c, {\bf d}}) =
L_1 (p_{c, {\bf d}})= \hdots = L_{n-2m} (p_{c, {\bf d}})= 0$,
$H_i (p_{c, {\bf d}})= -c$  if $d_i$ is $l$ and $H_i (p_{c, {\bf d}}) = c$
if $d_i$ is $r$. Taking the limit of $\gamma_{{\bf c}, {\bf d}}$ when
$ c$ tends to $0$, we  get
$2^m$ branches $\gamma_{\bf d}$ of straight extensions of $\gamma$, one
for each $\bf d$.

If, among the values $F_1(A), \hdots F_m(A)$, there are only $s$ values
different from 0, and the other $m-s$ values are equal to 0
$(0 \leq s < m)$, then we only have $2^s$ branches of straight extensions
of $\gamma$ instead of $2^m$ branches.
In particular, if $F_1(A) = \cdots = F_m(A) = 0$, then $\gamma$ lies
entirely on the local  $(n-m)$-dimensional  subspace
$\{F_1(x) = \cdots = F_m(x) = 0\}$ of $\cB$ with a flat affine
structure (we can ignore the singular points when we restrict our
attention to this subspace because all the multi-valued affine
coordinate functions are single-valued there), and so $\gamma$ admits
a unique straight extension in this case.

\subsection{The notions of convexity and strong convexity}  \hfill

We begin with a definition.

\begin{definition}
\label{def_convexity}
 A singular affine manifold  $\cB$ (or a subset $\cC \subset \cB$)
is called (globally) \textbf{convex}
if for any two points $A,B \in \cB$ (or $A,B \in \cC$)
there exists a (regular or singular) straight line
from $A$ to  $B$ in $\cB$ (or in $\cC$, respectively).
We  say that $\cB$  is \textbf{locally convex} if any neighborhood of any
point $x \in\cB$   admits a sub-neighborhood of $x$
which is convex.  Similarly, a subset $\cC \subset \cB$ is called
is \textbf{locally convex} if any neighborhood of any
point $x \in\cC$   admits a sub-neighborhood of $x$ whose intersection
with $\cC$ is convex.
\end{definition}

Of course, any globally convex singular affine manifold is locally
convex, but the converse is not true in general.

The above notion of convexity using straight lines is very natural
and is a generalization of the notion of \textbf{\textit{intrinsic
convexity}} (see \cite{Zung-Proper2006}) from the regular case to the
singular case. When $\cB$ is affinely immersed in $\RR^n$ with its
canonical affine structure, this Definition \ref{def_convexity}
is equivalent to the usual definition of convexity.

\begin{figure}[!ht]
\vspace{-15pt}
\centering
 {\mbox{} \hspace*{-0cm }\includegraphics[width=0.6 \textwidth]{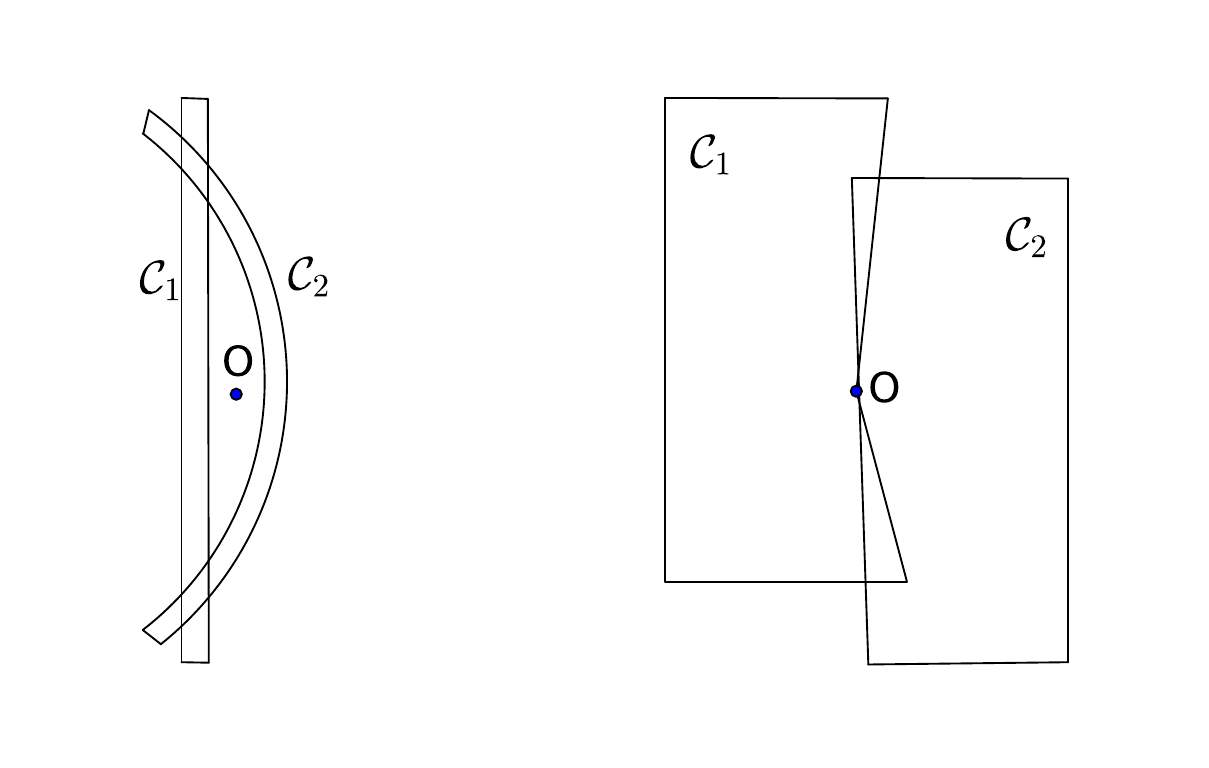}}

\vspace{-15pt}
  \caption{``Bad'' intersections of convex sets near a focus point.}
  \label{fig:BadIntersections}
\end{figure}

In a Euclidean space, every (locally) convex set is automatically
(locally) path-connected, and the intersection of two (locally)
convex sets is again a (locally) convex set.
In a singular affine manifold, a (locally) convex set is again
automatically (locally) path-connected (because we can go from one
point to another by a straight line). However, due to singular points
and monodromy, the intersection of two convex sets is neither
necessarily locally convex nor connected, as shown in Figure
\ref{fig:BadIntersections}.

So, in order to make sure that the intersection of two sets with
convexity properties still inherit convexity properties, we need a
stronger notion of convexity for sets in singular affine manifolds.

\begin{definition}
\label{def_strong_convexity}
{\rm(i)} A subset $\cC$  of a singular convex affine manifold $\cB$
is called \textbf{strongly convex}  in  $\cB$ if $\cC$ is convex,
i.e., any two distinct points $A,B\in \cC$ can be joined by a
straight line which lies in $\cC$, and, moreover, any straight line
going from $A$ to $B$ in $\cB$  also lies $\cC$.

{\rm(ii)} We say that $\cC \subset \cB$ is \textbf{strongly locally
convex} in  $\cB$ if any neighborhood $U$ of any point $x \in \cC$
admits a sub-neighborhood $V$ of $x$ such that $\cC \cap V$ is strongly
convex in $V$. In other words, for any two distinct points
$A, B \in \cC \cap V$ there is a straight line
going from $A$ to $B$ in $\cC \cap V$, and, moreover, any straight line
going from $A$ to $B$ in $V$ also lies in $\mathcal{C}$.
\end{definition}

\begin{remark}{\rm
Recall that, because of monodromy, given two points in $\mathcal{B}$,
there could be more than one straight line segment in the affine
structure of $\mathcal{B}$ linking these two points. So Definition
\ref{def_convexity} requires only that at \textit{least one} such
segment lies in $\mathcal{C}$, whereas  Definition
\ref{def_strong_convexity} requires that \textit{all} such segments lie
in $\mathcal{C}$.} \hfill $\lozenge$
\end{remark}

\begin{proposition}
{\rm (i)} The intersection of two strongly convex sets in a singular
affine manifold $\cB$ is again a  strongly convex set in $\cB$.

{\rm (ii)} The intersection of two strongly locally
convex sets in a singular affine manifold $\cB$
is again a  strongly locally convex set in $\cB$.
\end{proposition}

\begin{proof}
The proof is straightforward.

(i) Let $\cC_1, \cC_2$ be two strongly convex subsets of $\cB$, and
$A, B \in \cC_1 \cap \cC_2$. Let $\gamma$ be a straight line from
$A$ to $B$ in $\cC_1$ ($\gamma$ exists because $\mathcal{C}_1$ is convex).
Then $\gamma \subset \cC_2$ because $\cC_2$ is strongly convex, so
$\gamma \subset \mathcal{C}_1 \cap \cC_2$,
i.e., there is at least one straight line from $A$ to
$B$ in $\mathcal{C}_1 \cap \cC_2$. Moreover, if  $\gamma'$ is any other
straight line from $A$ to $B$ in $\cB$, then $\gamma' \subset
\mathcal{C}_i$ ($i=1,2$) because $\cC_i$ is strongly convex, and so
$\gamma' \subset \mathcal{C}_1 \cap \mathcal{C}_2$.

(ii) Let $\cC_1, \cC_2$ be two strongly locally convex subsets of $\cB$,
$x \in \cC_1 \cap \cC_2$, and $U$ is an arbitrary neighborhood of
$x$ in $\cB$. By definition, there is a neighborhood $V_1 \subset U$
of $x$ such that $\cC_1 \cap V_1$ is strongly convex in $V_1$, and
a neighborhood $V \subset V_1$ of $x$ such that
$\cC_2 \cap V$ is strongly convex in $V$. Then
$(\cC_1\cap\cC_2) \cap V$ is strongly convex in $V$.

Indeed, let  $A, B \in (\cC_1 \cap \cC_2) \cap V$ be two arbitrary
distinct points. Since $\cC_2 \cap V$ is convex, there is a straight
line $\gamma$ going from $A$ to $B$ in $\cC_2 \cap V$.
Since $\gamma \subset V \subset V_1$, $A, B \in \cC_1 \cap V_1$ and
$\cC_1 \cap V_1$ is strongly convex in $V_1$, we have that
$\gamma \subset \cC_1 \cap V_1 \subset \cC_1$. Hence
$\gamma \subset \cC_1 \cap \cC_2 \cap V$. Let $\gamma'$ be any other
straight line going from $A$ to $B$ in $V \subset V_1$.
Then, since $\mathcal{C}_1 \cap V_1$ is strongly convex in $V_1$
and $\mathcal{C}_2 \cap V$ is strongly convex in $V$, we have that
$\gamma' \subset \cC_1 \cap V_1$ and $\gamma' \subset \cC_2 \cap V$,
and hence $\gamma' \subset (\cC_1 \cap \cC_2) \cap V$.
\end{proof}

\begin{remark}{\rm
For \textit{regular} affine manifolds there is another different 
notion of convexity (which is not equivalent to ours),
which requires the universal covering of the manifold to be convex
(see, e.g., \cite{Barbot2000}).
}
\end{remark}

\section{Local convexity at focus points}
\label{section:LocalConvexity}

\subsection{Convexity of focus boxes in dimension 2} \hfill
\label{subsection:Local1}

Let $O$ be a focus singular point of index $k >0$
 in a singular 2-dimensional affine manifold $\cB^2$ with a local
 multi-valued  affine coordinate system
$F,G$: $F(O) = G(O) = 0$, $F$ is single valued,
$G$ is multivalued: when $F \leq 0$ then $G$ is single valued but when
$F > 0$ then $G$ has 2 branches $G_l$ and $G_r$ (extension of $G$ from
the region $\{F \leq 0 \}$ to the region $\{F > 0\}$
by the left or by the right of the singular point $O$), with the monodromy
formula  $G_r = G_l + kF$ when $F > 0$.

Let  $\delta, \delta' >0$ be two 
positive numbers, which are small enough if necessary for things to be well defined. The closed neighborhood $Box = Box(\delta,\delta')$
of $O$ defined by the inequalities
 \begin{equation}
Box = \{ x \ \text{such that}\, F(x) \in [-\delta,\delta] \text{ and }
G_l(x), G_r(x) \in [- \delta', \delta'] \}
 \end{equation}
is called a \textbf{\textit{focus box}}, i.e., a box with one focus
point in it.

\noindent\textbf{Remark}.
  The box is a quadrilateral figure
  (i.e., a figure whose boundary consists of  four straight lines)
  if $2\delta' > k\delta$; for simplicity of the exposition
we assume that this is the case. If $2\delta' \leq k\delta$
then the box is a triangle, but our subsequent arguments
based on the quadrilateral figure are easily seen to apply to this case
too.

\begin{theorem} \label{thm:FF1Local}
Any 2-dimensional focus box is convex.
\end{theorem}

\begin{proof}
We give two simple proofs.

\underline{Proof 1}. Let $A$ be an arbitrary point in the box. We show
that the box is $A$-star shaped, i.e., every other  point
can be connected to $A$ by a straight line in the box.

If $A$ lies on the line $\{F = 0\}$, then cut the box in two closed
parts by the line $\{F = 0\}$ (see Figure \ref{fig_cut_the_box}).
Each part is affinely isomorphic to a convex
polygon which contains $A$ on the boundary.
(The part  $\{F \leq 0\}$ is a rectangle and the part $\{F \geq 0\}$
is a trapezoid.) So both parts are $A$-star shaped and their union
is also $A$-star shaped, which proves the claim.

\begin{figure}[!ht]
\vspace{-15pt}
\centering
 {\mbox{} \hspace*{-2.3cm }\includegraphics[width=1.3 \textwidth]{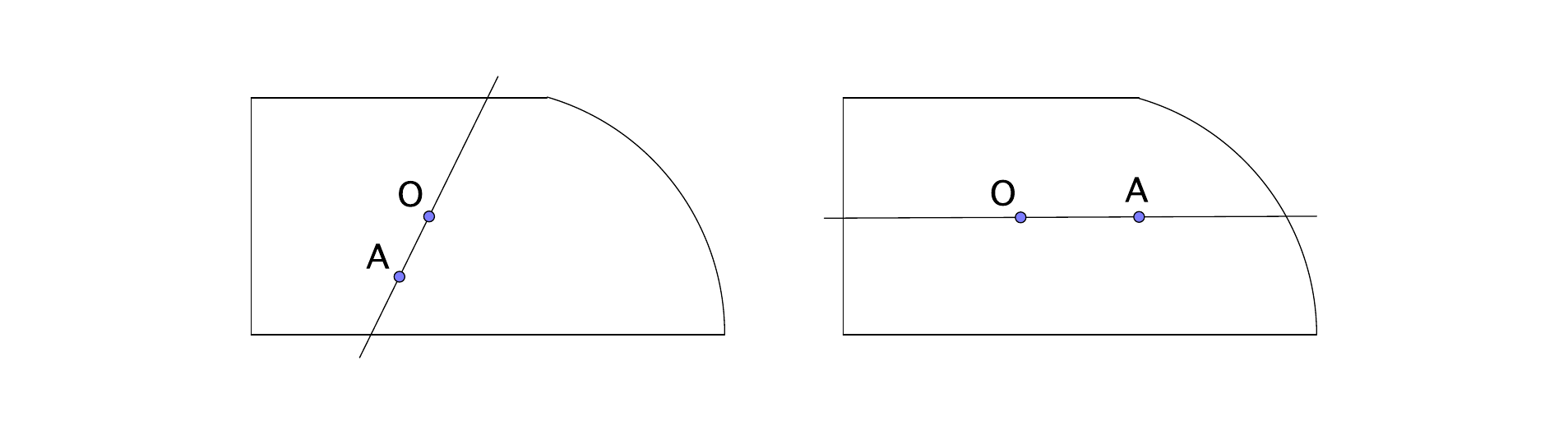}}

\vspace{-20pt}
  \caption{Cutting the box into two convex polygons with $A$ (and $O$) on the boundary.}
  \label{fig_cut_the_box}
\end{figure}

If $F(A) \neq 0$ then the straight line $AO$ admits two different straight
extensions at $O$. Just take any of these two and keep extending this line on
both ends so that the obtained straight line cuts the box in two parts (see
Figure \ref{fig_cut_the_box}). Again each part is a convex polygon with $A$ on
the boundary, which proves the claim.

\underline{Proof 2}. Let $A$ and $B$ be two arbitrary points in the box.
We have to find a straight line in the box going from $A$ to $B$.

- If $F(A) \leq 0$ and $F(B) \leq 0$, then $A$ and $B$ lie in the rectangle
$\{F \leq 0\}$ and in this case, there is a unique
straight line going from $A$ to $B$ consisting of points
$\gamma(t), \; 0 \leq t \leq 1$, such that
$F(\gamma(t)) = t F(B) + (1-t) F(A)$ and $G(\gamma(t)) = t G(B) +
(1-t) G(A)$. (The uniqueness follows from the fact that $F$ is single
valued and affine on the line, so $F$ is always smaller than or equal
to $0$ on the line, and when $F \leq0$, then there is no monodromy,
so $G$ is single valued and is affine on the line).

- Similarly, if $F(A) \geq 0$ and $F(B) \geq 0$, then $A$ and $B$
lie in the trapezoid $\{F \geq 0\}$, and in this case there is a unique
straight line going from $A$ to $B$, consisting of points
$\gamma(t), \; 0 \leq t \leq 1$, such that
$F(\gamma(t)) = t F(B) + (1-t) F(A)$ and $G_l(\gamma(t)) = t G_l(B) + (1-t) G_l(A)$ (or, equivalently,
$G_r(\gamma(t)) = t G_r(B) + (1-t) G_r(A)$).

- The more tricky situation is when $F(A) < 0$, $F(B) > 0$, or vice versa.
In this case, there exists at least one, but maybe  two different
straight lines from $A$ to $B$ in the box. We construct
up to two different straight lines from $A$ to $B$,
denoted by $\gamma_l : [0,1] \to Box$ and
$\gamma_r : [0,1] \to Box$.
The equations for the points of $\gamma_l : [0,1] \to Box$ are
$$ F(\gamma_l(t)) = t F(B) + (1-t) F(A)\; ; \quad
G_l(\gamma_l(t)) = t G_l(B) + (1-t) G_l(A)$$
and the  equations for the points of $\gamma_r : [0,1] \to Box$ are
$$
F(\gamma_r(t)) = t F(B) + (1-t) F(A)\; ; \;G_r(\gamma_r(t)) =
t G_r(B) + (1-t) G_r(A)
$$
for all $t \in [0,1]$; when $F \leq 0$ then $G_l = G_r = G$.

In particular, at the point $t_0 = \dfrac{-F(A)}{F(B)-F(A)}$\,, we have
$F(\gamma_r(t_0)) = 0$, $G_l(\gamma_l(t_0)) = t_0 G_l(B) + (1-t_0) G(A)$,
$G_r(\gamma_r(t_0)) = t_0 G_r(B) + (1-t_0) G(A).$

The lines $\gamma_l$ and $\gamma_r$ can be defined on the interval $[0,t_0]$
(for $F$ going from $F(A)$ to $0$) without any problem.
Indeed, since
the values of $G$ remain in the interval $[-\delta',\delta']$,
this line cannot leave the box. However, at $t_0$
we may run into problems.

\noindent $\bullet$ $G_l$ is the extension of $G$ from the left of $O$,
i.e., through points on $\{F=0\}$, where $G$ has negative values. Thus,
if $ t_0 G_l(B) + (1-t_0) G(A) >0$, then $G_l$ is the wrong function
to use and the equations $F(\gamma_l(t)) = t F(B) + (1-t) F(A)$,
$G_l(\gamma_l(t)) = t G_l(B) + (1-t) G_l(A)$  do not give a straight line,
but rather a broken line at $t_0$. So, in order for $\gamma_l$ to be a
straight line from $A$ to $B$, we need the following necessary and sufficient condition:
$$t_0 G_l(B) + (1-t_0) G(A) \leq 0$$

\noindent $\bullet$ Similarly,
in order for $\gamma_r$ to be a straight line from $A$ to $B$, we need the condition
$ t_0 G_r(B) + (1-t_0) G(A) \geq 0$.

Note that $G_r(B) = G_l(B) + kF(B) > G_l(B)$, and so
$ t_0 G_r(B) + (1-t_0) G(A) >  t_0 G_l(B) + (1-t_0) G(A)$. We have three
different possibilities (see Figure \ref{fig_3_possibilities}).

\begin{figure}[!ht]
\vspace{-15pt}
\centering
 {\mbox{} \hspace*{-2.3cm }\includegraphics[width=1.3 \textwidth]{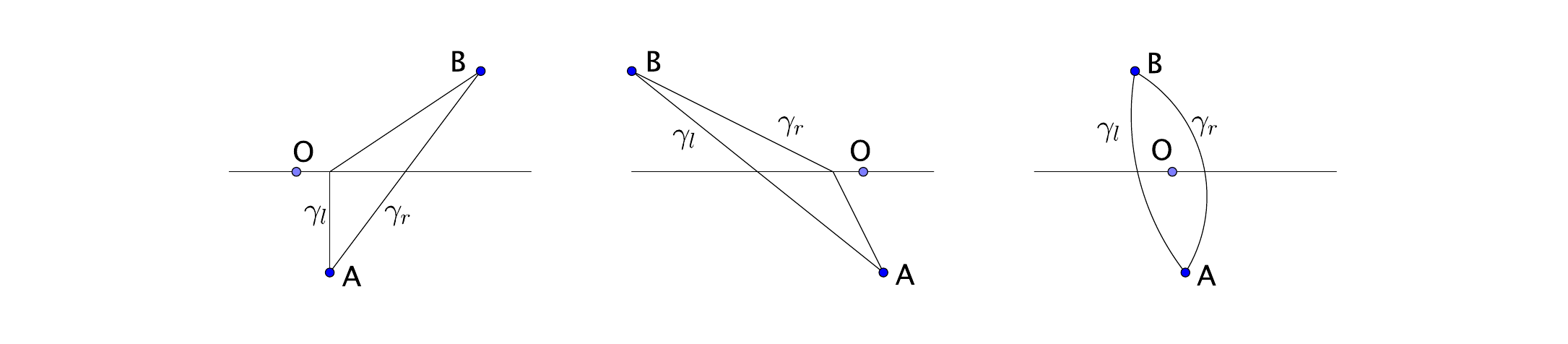}}

\vspace{-15pt}
  \caption{The three cases in a focus box.}
  \label{fig_3_possibilities}
\end{figure}

i) $ t_0 G_l(B) + (1-t_0) G(A) > 0$: then, automatically
$ t_0 G_r(B) + (1-t_0) G(A) > 0$, and we have exactly one straight line
from $A$ to $B$ in the box, which is $\gamma_r$.

ii) $ t_0 G_r(B) + (1-t_0) G(A) < 0$: then automatically
$ t_0 G_l(B) + (1-t_0) G(A) < 0$, and we have exactly one straight line
from $A$ to $B$ in the box, which is $\gamma_l$.

iii) $ t_0 G_l(B) + (1-t_0) G(A) \leq 0$ and $ t_0 G_r(B) + (1-t_0) G(A) \geq 0$:
then we have two different straight lines from $A$
to $B$ in the box, which are  $\gamma_l$ and $\gamma_r$.
\end{proof}

\begin{remark}{\rm
The number $k$ in the monodromy formula ($G_r = G_l + kF$ when $F > 0$)
can be any positive number (not necessarily an integer) and the focus
box is still convex, as the proof above shows. However,
the requirement $k > 0$ is essential: if $k <0$ then the box is
non-convex, because there exist points $A$ and $B$ in the box such that
$F(A) < 0$, $F(B) > 0$, $ t_0 G_l(B) + (1-t_0) G(A) > 0$, and
$ t_0 G_r(B) + (1-t_0) G(A) < 0$, where $t_0 = \dfrac{-F(A)}{F(B)-F(A)}$. See Figure \ref{ConvexFF2_figure}.
  }
\end{remark}

\begin{figure}[!ht]
\vspace{-15pt}
\centering
 {\mbox{} \hspace*{-1.5cm}\includegraphics[width=1.2 \textwidth]{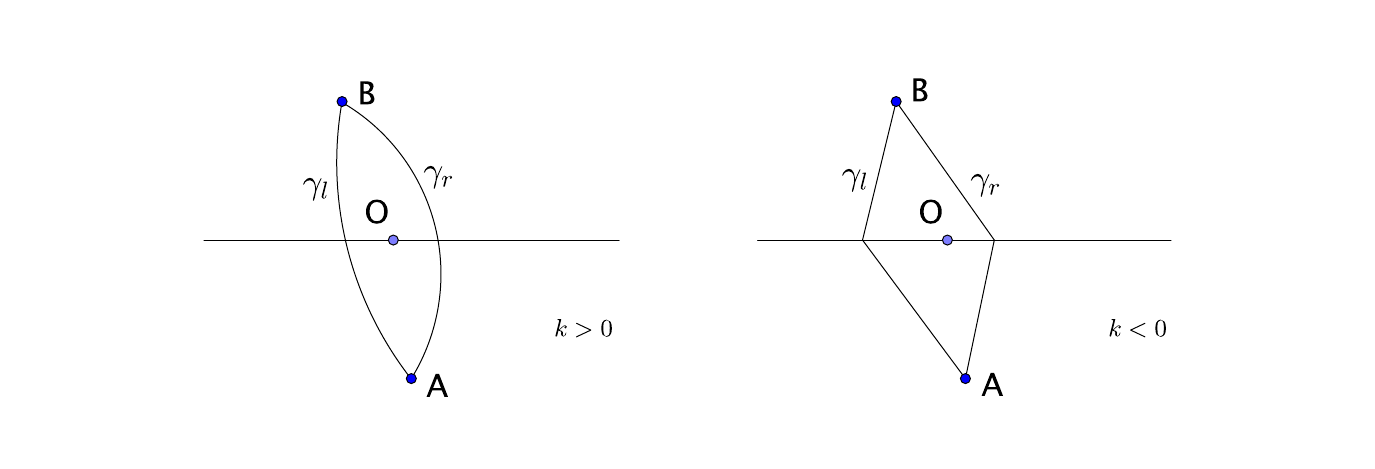}}
\vspace{-15pt}
  \caption{Two straight lines (the case $k > 0$) vs. no straight line ($k < 0$)
  from $A$ to $B$.}
  \label{ConvexFF2_figure}
\end{figure}

\subsection{Convexity of focus boxes in higher dimensions} \hfill

Theorem \ref{thm:FF1Local}
can be extended to the $n$-dimensional case in a straightforward manner.

Let $O$ be a focus singular point of index $k >0$ in a singular
$n$-dimensional affine manifold $\cB$ ($n \geq 3$)
with a local multi-valued  affine coordinate system
$F,G, L_1, \hdots, L_{n-2}$; $F(O) = G(O) = L_1(O) = \cdots =
L_{n-2}(O) = 0$,  $F,L_1,\hdots,L_{n-2}$ are single valued, and
$G$ is multivalued: when $F \leq 0$ then $G$ is single valued but when
$F > 0$ then $G$ has 2 branches $G_l$ and $G_r$ (extensions of $G$ from
the region $\{F \leq 0 \}$ to the region $\{F > 0\}$
on the left or on the right of the  $(n-1)$-dimensional manifold
$\cS \subset \{F=0\}$ of focus points, $O \in \cS$), with the monodromy
formula  $G_r = G_l + kF$ when $F > 0$ and $G_r = G_l = G$ when $F \leq 0$.

Similarly to the 2-dimensional case, we define a
\textbf{\textit{focus box}} around
the focus points $O$ to be the set of all points $x$
 in a neighborhood of $O$ in $\cB$ which satisfies the inequalities
 $$
 -\delta \leq F(x) \leq \delta,\, - \delta' \leq G_l(x),\, G_r(x)
 \leq \delta',
 $$
 $$
 -\delta_1 \leq L_1(x) \leq \delta_1,\hdots,  - \delta_{n-2} \leq L_{n-2} \leq \delta_{n-2} ,
 $$
where $\delta, \delta',\delta_1,\hdots, \delta_{n-2}$
are arbitrary sufficiently small positive numbers. We do this if
$O$ is an interior point of $\cB$.

If $O$ lies on the boundary of $\cB$, where $\cB$ comes from an
integrable Hamiltonian system whose singularities are non-degenerate and
have only elliptic and/or focus-focus components, then the singularity of
the integrable system at $O$ has $e$ elliptic components for some
positive number $e$, and we can assume that $L_1,\hdots,L_{e}$ are
the action functions corresponding to the $e$ elliptic components at $O$,
$L_1, \hdots, L_e \geq 0$ near $O$ (and the other functions
$L_{e+1}, \hdots, L_{n-2}$ can admit both negative and positive values
near $O$). On this boundary (i.e., $\text{elliptic}^e$-focus-focus case),
the above inequalities  for $F,G, L_i$ still define a box around $O$,
and for $i=1,\hdots,e$ we can replace the inequalities
$- \delta_i \leq L_i(x) \leq \delta_i$ by the inequalities
$0\leq L_i(x) \leq \delta_i$.

\begin{theorem} \label{thm:FF2Local}
Any $n$-dimensional focus box is convex.
\end{theorem}

\begin{proof}
We simply repeat the second proof of Theorem \ref{thm:FF1Local}, with
$L_1, \hdots, L_{n-2}$ added to the picture.
Let $A$ and $B$ be two arbitrary points in the box.
We have to find a straight line in the box going from $A$ to $B$.

- If $F(A) \leq 0$ and $F(B) \leq 0$ then $A$ and $B$ lie in the
(affinely flat) convex polytope $\{x \in Box, F(x) \leq 0\}$. In this
case, there is a unique straight line going from $A$ to $B$, consisting of
points $\gamma(t), \; 0 \leq t \leq 1$, such that
$F(\gamma(t)) = t F(B) + (1-t) F(A)$,  $G(\gamma(t)) = tG(B)+(1-t)G(A)$,
and $L_i(\gamma(t)) = t L_i(B) + (1-t) L_i(A)$ for all $i=1,\hdots, n-2$.
(This straight line is unique for the same reasons as in the
2-dimensional case).

- Similarly, if $F(A) \geq 0$ and $F(B) \geq 0$, then $A$ and $B$
lie in the affinely flat convex polytope $\{ x \in Box, F(x) \geq 0\}$,
and, in this case, there is a unique straight line going from $A$ to $B$,
consisting of points $\gamma(t), \; 0 \leq t \leq 1$ such that
$F(\gamma(t)) = t F(B) + (1-t) F(A)$, $L_i(\gamma(t)) =
t L_i(B) + (1-t) L_i(A)$ for all $i=1,\hdots, n-2$, and
$G_l(\gamma(t)) = t G_l(B) + (1-t) G_l(A)$
(or, equivalently, $G_r(\gamma(t)) = t G_r(B) + (1-t) G_r(A)$).

- The more tricky situation is when $F(A) < 0$ and $F(B) > 0$ or vice
versa. In this case, there exists at least one,
but maybe  two different straight lines from $A$ to $B$ in the box.
We construct up to two different straight lines from $A$ to $B$,
denoted by $\gamma_l : [0,1] \to \cB$ and $\gamma_r : [0,1] \to Box$.
The equations for the points of $\gamma_l : [0,1] \to Box$ are
(for $t \in [0,1]$):
$$
L_i(\gamma_l(t)) = t L_i(B) + (1-t) L_i(A),\quad
\forall\; i=1,\hdots, n-2,
$$
$$
F(\gamma_l(t)) = t F(B) + (1-t) F(A), \quad
G_l(\gamma_l(t)) = t G_l(B) + (1-t) G_l(A),
$$
and the  equations for the points of $\gamma_r : [0,1] \to Box$ are
$$
L_i(\gamma_l(t)) = t L_i(B) + (1-t) L_i(A)\; \forall\; i=1,\hdots, n-2,
$$
$$
F(\gamma_r(t)) = t F(B) + (1-t) F(A), \quad
G_r(\gamma_r(t)) = t G_r(B) + (1-t) G_r(A).
$$

In particular, at the point $t_0 = \dfrac{-F(A)}{F(B)-F(A)}$\,, we have
$F(\gamma_r(t_0)) = 0$, $G_l(\gamma_l(t_0)) = t_0 G_l(B) + (1-t_0) G(A)$,
$G_r(\gamma_r(t_0)) = t_0 G_r(B) + (1-t_0) G(A)$,
and $L_i(t_0) = c_i \in [-\delta_i,\delta_i]$ ($i=1,\hdots, n-2$) are some small
numbers.

The line $\{x\mid F(x) = 0, L_1(x) = c_1, \hdots, L_{n-2}(x) = c_{n-2}\}$
contains exactly one focus point, denoted by $O_{A,B}$, and the value
$g^{critical}_{A,B} = G(O_{A,B})$ is called
the critical value of $G$ (on this line, or with respect to $A$ and $B$).

The lines $\gamma_l$ and $\gamma_r$ can be defined on the interval
$[0,t_0]$ (for $F$ going from $F(A)$ to $0$) without any problem,
because the values of the functions $F,G, L_i$ always stays in the
intervals of allowed values for the box; hence these
values cannot leave the box. However, at $t_0$ we may run into problems.

\noindent $\bullet$ $G_l$ is the extension of $G$ from the left of
$O_{A,B}$, i.e., through points on $\{F=0, L_1 = c_1, \hdots,
L_{n-2} = c_{n-2}\}$, where $G$ has values less than the critical value.
Thus, if $ t_0 G_l(B) + (1-t_0) G(A) > g^{critical}_{A,B}$, then $G_l$
is the wrong function to use: the equations
$F(\gamma_l(t)) = t F(B) + (1-t) F(A)$, $G_l(\gamma_l(t)) =
t G_l(B) + (1-t) G_l(A)$  do not give a straight line
but rather a line which is broken at $t_0$. So in order for
$\gamma_l$ to be a straight line from $A$ to $B$, we need the following necessary and sufficient condition:
$$ t_0 G_l(B) + (1-t_0) G(A) \leq g^{critical}_{A,B} .$$

\noindent $\bullet$ Similarly,  in order for $\gamma_r$ to be a straight line from $A$ to $B$, we need the condition
$ t_0 G_r(B) + (1-t_0) G(A) \geq g^{critical}_{A,B}$.

Note that $G_r(B) = G_l(B) + kF(B) > G_l(B)$, and so
$ t_0 G_r(B) + (1-t_0) G(A) >  t_0 G_l(B) + (1-t_0) G(A)  $. We have three different possibilities:

i) $ t_0 G_l(B) + (1-t_0) G(A) > g^{critical}_{A,B}$:
then, automatically, $ t_0 G_r(B) + (1-t_0) G(A) > g^{critical}_{A,B}$,
and we have exactly one straight line
from $A$ to $B$ in the box, which is $\gamma_r$.

ii) $ t_0 G_r(B) + (1-t_0) G(A) < g^{critical}_{A,B}$:
 then, automatically, $ t_0 G_l(B) + (1-t_0) G(A) < g^{critical}_{A,B}$,
 and we have exactly one straight line
from $A$ to $B$ in the box, which is $\gamma_l$.

iii) $ t_0 G_l(B) + (1-t_0) G(A) \leq g^{critical}_{A,B}$
and $ t_0 G_r(B) + (1-t_0) G(A) \geq g^{critical}_{A,B}$:
then we have two different straight lines from $A$
to $B$ in the box, which are  $\gamma_l$ and $\gamma_r$.
\end{proof}

\begin{remark}
{\rm
Another proof of Theorem \ref{thm:FF2Local} can be given by the use of
appropriate $n-2$ independent affine equations
to define a 2-dimensional singular affine ``subbox''
$$
P = \{x \in Box \; |  \; a_jF(x) + L_j(x) = b_j \; \forall\ j = 1,\hdots,n-2 \}
$$
of $Box$, which contains
$A$, $B$, and $O_{A,B}$
as the only focus point in $P$. This reduces
the problem to the 2-dimensional
case already treated in the previous subsection.
}
\end{remark}

\begin{remark}
{\rm
The fact that the local $(n-2)$-submanifold of  focus points near $O$
lies on $\{F =0\}$ is  important in the proof of Theorem \ref{thm:FF2Local}.
If focus points were able to move more freely, one would be
able to construct counter-examples (where both choices $\gamma_l$
and $\gamma_r$ in the construction of a straight line from $A$ to $B$
are bad choices).
}
\end{remark}

\subsection{Existence of non-convex  $\text{focus}^m$ boxes}  \hfill
\label{subsection:NonConvexBox}

Toric-focus systems may lose convexity due to the presence of a
$\text{focus}^m$ singularity.

\begin{theorem}
\label{thm:NonConvexFocus2}
For any $m \geq 2$ and $n \geq 2m$, there exists a toric-focus
integrable system with $n$ degrees of freedom whose base space $\cB^n$
admits an interior singular point $O$ with $m$ focus components
(and no elliptic component) with a corresponding multi-valued
system of action coordinates $(F_1,G_1,\hdots, F_m,G_m,L_1,\hdots, L_{n-2n})$
(where  $F_1, \hdots, F_m, L_1,\hdots, L_{n-2n}$ are smooth single valued
and each function $G_i$ has two branches $(G_i)_l$ and $(G_i)_r$), and a
number $\delta > 0$ such that the $\text{focus}^m$  box
\begin{equation*}
\{ |F_i|, |L_j| \leq \delta \; \forall \; i,j; \; - \delta \leq (G_i)_l\; \& \; (G_i)_r \leq \delta,\; \forall\; i\}
\end{equation*}
around $O$ is locally convex at its boundary points, but is not convex.
\end{theorem}

\begin{proof}
We prove for the case $m = 2$ and $n=4$. The other cases are similar.

Let $O$ be a singular point of type focus square in
the four dimensional base space $\cB$
of a toric-focus integrable system with four degrees of freedom.
We assume that the singularity of the system over $O$
is  topologically a direct product of two elementary
focus-focus singularities of index 1. So we have a local smooth
coordinate system $(F_1, H_1, F_2, H_2)$ and a local
multi-valued system of action coordinates $(F_1,G_1,F_2, G_2)$
as described in Subsection \ref{subsection:FocusMBehavior}.

Take a number $\delta > 0$ and consider the box
\begin{equation*}
Box = \{x \in \cB \;|\; |F_i(x)| < \delta;\; -\delta \leq (G_i)_l(x);\;
(G_i)_r (x) \leq \delta,\; \forall \, i = 1,2 \}
\end{equation*}
containing $O$ in $\cB$. According to  Theorem \ref{thm:FF2Local},
this box is locally convex at its boundary points. We want to show
that the system can be chosen in such a way that our box is
\textit{not} convex.

If the  decomposition of the given $(\text{focus-focus})^2$ singularity
of the integrable Hamiltonian system
into a product of two elementary focus-focus singularities
is not only topological but also symplectic, then
the local base space near $O$ is also a direct product of two affine
2-dimensional manifolds (each one with a focus-focus
singular point), and then the convexity of $\cB$ near $O$ is obvious,
because it is clear that the direct product of two  convex (singular)
affine manifolds is convex. However, in general, this local product
decomposition of $\cB$ is only topological but not affine and that's
why we may run into non-convexity.

Let us recall that the set of singular points in the box
is the union $\cS_1 \cup \cS_2$ of two 2-dimensional quadrilateral
disks intersecting each other transversally at $O$,
each disk being given by the formula
\begin{equation*}
\cS_i = \{x \in Box\; |\; F_i(x) = H_i(x) = 0\}.
\end{equation*}
For each $i= 1,2$, the function $G_i$ is single-valued when
$F_i \leq 0$ but it has two branches $(G_i)_l$ and $(G_i)_r$
when $F_i >0$; $(G_i)_l = (G_i)_r = G_i$ when $F_i \leq 0$. These two
branches of $G_i$ are related by the monodromy formula $(G_i)_r =
(G_i)_l + F_i$ when $F_i >0$.

Take two different points $A$ and $B$ in the box, with
$F_1(A) < 0$, $F_2(A) < 0$ and $F_1(B) > 0$, $F_2(B) > 0$. Then,
potentially, there may be up to four different
straight lines in the box going from $A$ to $B$, each line corresponding to
one of the four choices of the value of the couple $(G_1,G_2)$ at $B$.
(This couple is single-valued at $A$.) Define four (straight or broken
piecewise straight) lines
\begin{equation*}
\gamma^{l,l}, \gamma^{l,r}, \gamma^{r,l}, \gamma^{r,r}: [0,1] \to Box
\end{equation*}
which go from $A$ to $B$ by the following formulas:
\begin{align*}
F_1(\gamma^{l,r}(t)) &= tF_1(B) + (1-t) F_1(A), \\
 F_2(\gamma^{l,r}(t)) &= tF_2(B) + (1-t) F_2(A), \\
(G_1)_l(\gamma^{l,r}(t)) &= t(G_1)_l(B) + (1-t) G_1(A), \\
(G_2)_r(\gamma^{l,r}(t)) &= t(G_2)_r(B) + (1-t) G_2(A)
\end{align*}
for $\gamma^{l,r}$, and similarly for the other three  lines.

It is easy to see that these four lines lie inside the box, because the
inequalities defining the box are satisfied on all of them. If at least one of
these four lines is a straight line (without any breaking, i.e., changing of
direction in the middle of the line) then it means there is a straight line from
$A$ to $B$. But similarly to the 2-dimensional cases, some of them may be
broken. Actually, as we will see, it may happen that all of them are broken and,
in that case, we cannot join $A$ and $B$ by any straight line.

Look at $\gamma^{r,l}$, for example. It is broken at the intersection
with the hyperplane $\{F_1 = 0\}$, if it intersects
$\{F_1 = 0\}$, on the ``left side'' (i.e., the side given by
$\{H_1 < 0\}$, which is the ``wrong side'' for  $ \gamma^{r,l}$)
of the surface  $\cS_1$ of focus-focus points on $\{F_1 = 0\}$ (i.e.,
the side with lower $G_1$-values than on $\cS_1$). It can also
be broken at the intersection with
the hyperplane $\{F_2 = 0\}$, if it intersects
$\{F_2 = 0\}$, on the ``right side'' (i.e., the side given by
$\{H_2 > 0\}$, which is again the ``wrong side'' for $ \gamma^{r,l}$)
of the surface  $\cS_2$ of focus-focus points on $\{F_2 = 0\}$
(i.e., the side with higher $G_2$-values than on $\cS_2$). See
Figure \ref{fig:4wrong}.

\begin{figure}[!ht]   
\vspace{-15pt}
\centering
 {\mbox{} \hspace*{0cm }\includegraphics[width=1 \textwidth]{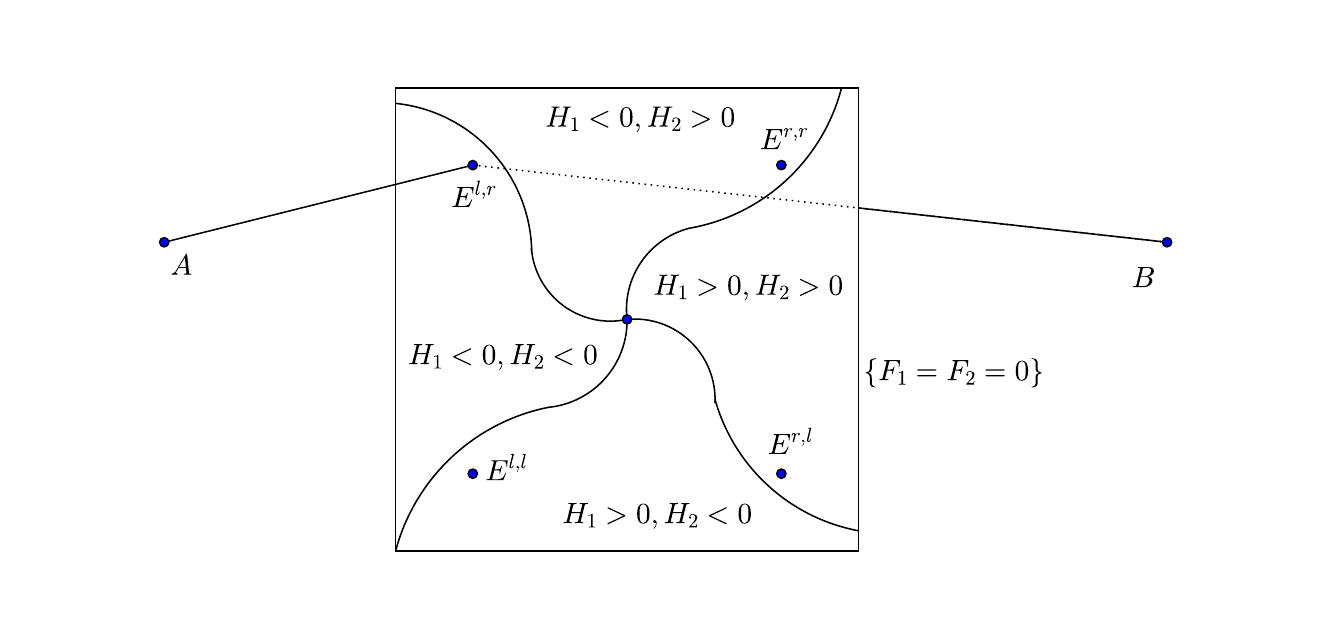}}
\vspace{-15pt}
\caption{A situation when all the choices are wrong choices.}\label{fig:4wrong}
\end{figure}

Let us assume, for example, that $\gamma^{l,l}$ is broken at $\{F_1=0\}$.
Then take $\gamma^{r,l}$ instead of  $\gamma^{l,l}$: by arguments
absolutely similar to the 2-dimensional situation, $\gamma^{r,l}$ is not
broken at $\{F_1=0\}$ in this case. But $\gamma^{r,l}$ can be broken
at  $\{F_2=0\}$. If  $\gamma^{r,l}$ is broken at  $\{F_2=0\}$, then
$\gamma^{r,r}$ is not broken at $\{F_2=0\}$, but then $\gamma^{r,r}$
can be broken at $\{F_1=0\}$. If $\gamma^{r,r}$ is broken at
$\{F_1=0\}$, we take $\gamma^{l,r}$ which is not broken at $\{F_1=0\}$,
but then $\gamma^{l,r}$ can be broken at $\{F_2=0\}$, in which case
we try to take $\gamma^{l,l}$ to avoid being broken at  $\{F_2=0\}$.
However, we already know that  $\gamma^{l,l}$ is broken at $\{F_1=0\}$.
See again Figure \ref{fig:4wrong}.

So, at least theoretically, it may happen that all of our four choices
go wrong: all of the lines
$ \gamma^{l,l}, \gamma^{l,r}, \gamma^{r,l}, \gamma^{r,r}$ are broken
somewhere. More specifically, up to a permutation of indices,
$ \gamma^{l,l}$ and $\gamma^{r,r}$ are broken at $\{F_1=0\}$, while
$ \gamma^{l,r}$ and $\gamma^{r,l}$ are broken at $\{F_2=0\}$. If this
happens, then our box is \emph{not} convex.

Such a situation (where all the possible four choices go wrong)
can actually happen. For example, we may have
$F_1(A) = F_2(A) = -\delta$, $F_1(B) = F_2(B) = \delta$ the four lines
$ \gamma^{l,l}, \gamma^{l,r}, \gamma^{r,l}, \gamma^{r,r}$
intersect the plane $\{F_1 = F_2 = 0\}$ (with
affine coordinates $(G_1, G_2)$) at four
respective points $ E^{l,l}, E^{l,r}, E^{r,l}, E^{r,r}$ as shown in
Figure \ref{fig:4wrong}, where the curves $H_1 = 0$ and $H_2 = 0$
(i.e., the curves of critical values for $G_1$ and $G_2$, respectively)
are also shown. One can see from this picture that all the four choices
in the situation on it are
wrong choices, so there is no straight line from $A$ to $B$.

By general results on \textit{integrable surgery} (the method
for constructing integrable system by glueing small pieces together, see
\cite[Proposition 4.10]{Zung-Integrable2003}: one first glues the base spaces, and then lift the glueing to
Lagrangian torus fibrations provided that the cohomological obstruction vanishes;
the cohomological obstruction lies in the so called Lagrange-Chern classes, which
automatically vanish when the base space is simple enough),
there is no obstruction to the creation of
an integrable Hamiltonian system with such a picture.

Before applying integrable surgery, we need to show the existence of 
semi-local models which correspond to sufficiently small neighborhoods of singular points of the base space. But in all our examples here and below, we may assume that very locally near each $\text{focus}^m$ point with $m \geq 2$ the base space is a local direct product (together with the affine structure), so that the existence of a corresponding integrable system over it is assured by simply taking a direct product of elementary pieces. For $\text{focus}^1$ points, which form codimension-2 submanifolds in $\cB$ which can be as curved as one likes, the existence of local integrable systems over them with a given local singular affine structure over the base space
is well known: for example, one can simply make a parametrized version of V\~u Ng\d{o}c's construction in \cite{VuNgoc2003} to construct them; see also 
\cite{Wacheux-Asymptotics2015}.
\end{proof}

\begin{remark}
{\rm
In our construction in the above theorem, the box is not convex, but if one takes a much smaller box
of the $\text{focus}^2$ point inside the original box then it is still convex. We suspect that this situation 
is true in general, though we do not have a formal proof for this fact at this moment: for any $\text{focus}^m$
point on a base space $\cB$ of a toric-focus integrable system there exists a small neighborhood of it
in $\cB$ which is convex, i.e., we still have (very) local convexity.}
\end{remark}

\section{Global convexity}
\label{sec_global_convexity}

\subsection{Local-global convexity principle}  \hfill

Let us recall the following two lemmas from \cite{Zung-Proper2006} for
regular affine manifolds (with boundary, but without focus singularities),
which are along the lines of the well-known local-global convexity
principle (see, e.g., \cite{CoDaMo-Moment1988,HiNePl-Convexity1994}).

\begin{lemma} \label{lem:convex_affine_map1}
Let $X$ be a connected compact locally convex regular affine manifold
(with boundary) and $\phi: X \rightarrow \bbR^m$ a locally injective
affine map, where $\bbR^m$ is endowed with the standard affine structure.
Then $\phi$ is injective and its image $\phi(X)$ is convex in $\bbR^m$.
\end{lemma}

\begin{lemma} \label{lem:convex_affine_map2}
Let $X$ be a connected locally convex regular affine manifold (with boundary) and
$\phi: X\rightarrow \bbR^m$ a proper locally injective affine map, where
$\bbR^m$ with is endowed with the standard affine structure. Then
$\phi$ is injective and its image $\phi(X)$ is convex in $\bbR^m$.
\end{lemma}

A simplified version of Lemma \ref{lem:convex_affine_map1} is the
following lemma, which is nothing but the compact version of the classical Tietze-Nakajima  theorem \cite{Tietze1928,Nakajima1928}:

\begin{lemma}[Compact version of the Tietze-Nakajima theorem]
\label{lemma_local_global_flat}
Let $C \subset \mathbb{R}^m$ be a compact connected locally convex set.
Then $C$ is convex.
\end{lemma}

\begin{proof} 

{\bf Proof 1:} There is a simple proof which goes as follows. For any two
points $x$ and $y$ in $C$ there is a piecewise linear path joining   
$x$ to $y$ in $C$ (by connectedness and local convexity of $C$) and
is, of course, of finite length with respect to a fixed Euclidean metric.
By the Arzel\`a-Ascoli theorem, the set of piecewise linear paths 
of finite lengths from  $x$ to $y$ in $C$ (parametrized by the length) 
admits a  sequence converging uniformly to a path $\gamma$ lying in $C$
(because $C$ is closed) and which realizes the infimum of lengths of paths
between $x$ and $y$. This path $\gamma$ must be a line segment. Indeed, if 
this were not the case  in a neighborhood of a  point, then we could shorten 
the length of $\gamma$ by joining  two nearby points on it by a straight line segment, which is possible by local convexity of $C$.

Note that this proof works because the affine structure of the Euclidean space 
$\mathbb{R}^m$ is compatible with its metric structure, i.e., straight line segments in the affine structure are paths of minimal length in the metric structure. However, later on (Proposition \ref{prop:FocusBox1}), we will work with affine structures on sets a priori do not have a compatible metric structure. This is why we provide a second proof that does not require a metric.

{\bf Proof 2:} If $x, y \in\mathbb{R}^m$ denote $[x,y]: = \{tx + (1-t)y \mid t \in [0,1]\}$
the straight line segment with extremities $x$ and $y$. For any
$x \in C \subset \mathbb{R}^m$ define the {\bfi star} $S(x):= \{y \in C\mid [x,y]\subset C\}$ of
$x$. It is clear that $C$ is convex if and only if $S(x) =C$ for all $x \in
C$. So, fix an arbitrary $x \in C$. We will show that $S(x)$ is both closed
and open in $C$; connectedness of $C$ guarantees then that $S(x) =C$.

To show that $S(x)$ is closed, let $y_n \rightarrow y$, where
$\{y_n\}_{n \in\mathbb{N}} \subset S(x)$. Since $C= \overline{C}$,
it follows that $y \in C$. This then implies that for any $t \in [0,1]$,
$tx + (1-t)y_n \rightarrow tx + (1-t)y$. However, $tx + (1-t)y_n \in
[x, y_n] \subset C$, by the definition of $S(x)$, and hence $[x,y] \subset
C$ since $C = \overline{C}$.

We now show that $S(x)$ is open in $C$, i.e., for any $y\in S(x)$ there is 
a neighborhood $U_y$ of $y$ in $\mathbb{R}^m$ such that 
$U_y \cap C \subset S(x)$. By local convexity of $C$, we can choose
a bounded neighborhood $U_y$ of $y$ in $\mathbb{R}^m$ such that $U_y \cap C$
is convex. Take any point $z \in U_y \cap C$, we will show that
$z \in S(x)$. In fact, we will show that there exists an affine map 
$\Phi:\Delta abc \to C$, where $\Delta abc$ is a triangle in $\RR^2$ 
with vertices $a,b,c \in \RR^2$, such that $\Phi(a)=x, \Phi(b)=y, \Phi(c)=z$.
Here, the affine property of $\Phi$ means the usual thing: the
pull-back of a local affine function by $\Phi$ is again a local affine function.
If such a $\Phi$ exists then of course $\Phi([a,c])$ is a straight line in
$C$ joining $x$ with $z$ and hence $z \in S(x)$.

To prove the existence of $\Phi$, we will use the
"engulfing technique", which can be conveniently expressed via the Zorn's lemma. 
Consider the set of "bombed triangles"
$$ \mathcal{S} = \{S \subset \Delta abc\; \text{such that}\;
[a,b]\cup [b,c]\subset S \; \text{and}\; \Delta abc \setminus S \; \text{is convex}\}$$
and the set of affine maps 
$$\cH = \{ \Psi: S  \to C \; | \; S\in \mathcal{S},
\Psi(a)=x, \Psi(b)=y, \Psi(c) = z,\; \text{and}\; \Psi 
\; \text{is affine}\}.$$
(The condition that $\Psi$ is affine means that the pull-back of a local affine function on $C$ by $\Psi$ 
is equal to the restriction of a local affine function on $\Delta abc$ to $S$.)

Then $\mathcal{S}$ and $\mathcal{H}$ are partially ordered sets 
with respect to the obvious inclusion order: for two elements  
$\Psi_1: S_1 \to C$ and $\Psi_2: S_2 \to C$ of $\mathcal{H}$, 
we write $\Psi_1 \preceq \Psi_2$ if $S_1 \subset S_2$ and 
the restriction of $\Psi_2$ to $S_1$ is equal to $\Psi_1$.
Notice that $\mathcal{H}$ is nonempty: 
it contains a (minimal) element $\Psi_{min}: [a,b] \cup [b,c] \to [x,y] \cup [y,z] \subset C$.
It is  also clear that $\mathcal{H}$ satisfies the inductivity
condition: if $\{\Psi_i: S_i \to C\}_{i \in I}$
is a totally ordered subset of $\mathcal{H}$, then it admits
an upper bound $\Psi: \cup_i S_i \to C$ defined by the formula $\Psi (x) = \Psi_i(x)$ if $x \in S_i$. (One verifies easily that $\cup_i S_i \in \mathcal{S}$ and $\Psi$ is well-defined and is affine.) Hence, by Zorn's lemma, $\mathcal{H}$ admits a maximal element $\Psi_{max}: S_{max} \to C$. 

\begin{figure}[!ht]
\vspace{-10pt}
\centering
 {\mbox{} \hspace*{0cm }\includegraphics[width=0.8 \textwidth]{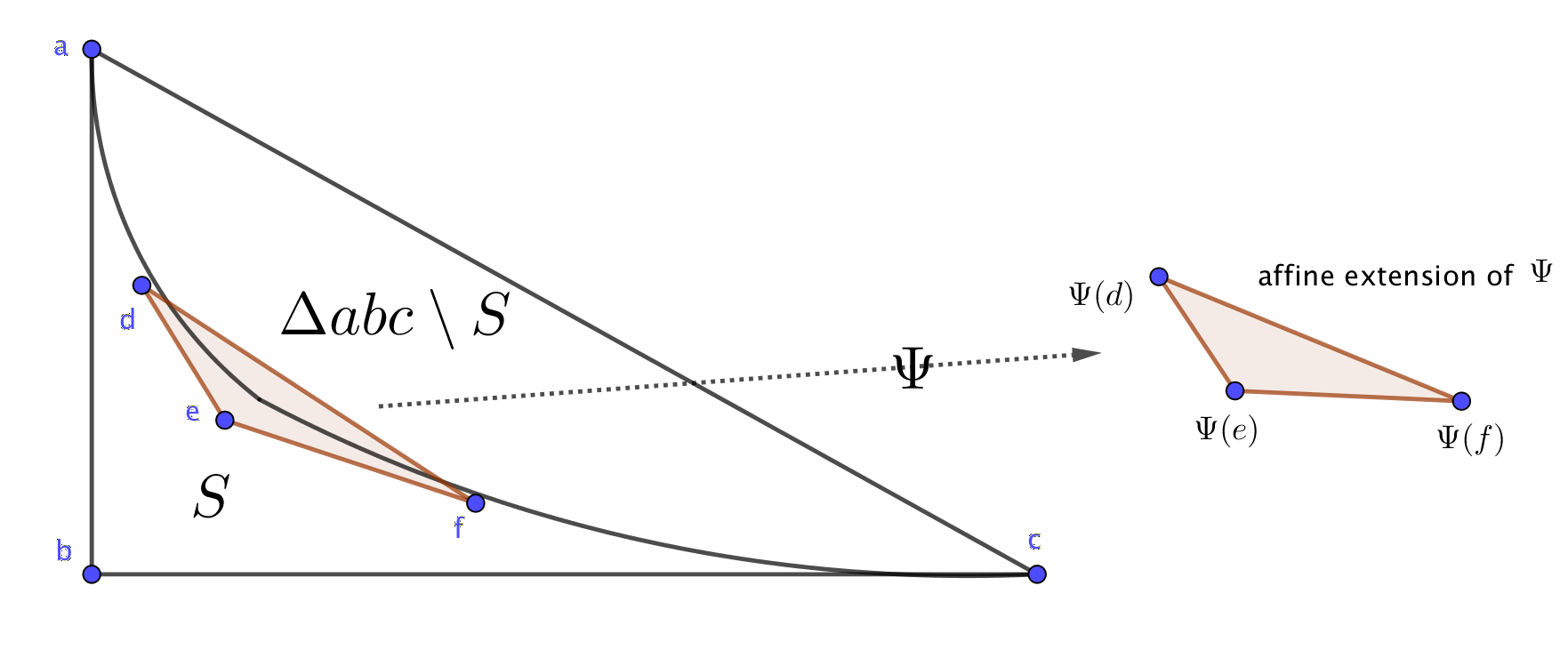} }

\vspace{-15pt}
  \caption{Making $\Psi$ bigger by affine extension.} \label{fig:AffineExtension}
\end{figure}

It remains to show that $S_{max} = \Delta abc$. Indeed, $S_{max}$ must be closed; otherwise by simply taking the continuous extension of $\Psi_{max}$ to the closure
of $S_{max}$ we would get a bigger element in 
$\mathcal{H}$. Now, if $\Delta abc \setminus S_{max}$ is nonempty convex, then there exists an arbitrarily small triangle $\Delta edf \subset 
\Delta abc$ such that $e,d,f \in S_{max}$ but 
$\Delta edf \cup S_{max} \in \mathcal{S}$ and is strictly bigger than $S$. Then, by local convexity of $C$, we can extend $\Psi_{max}$ affinely to 
$\Delta edf \cup S_{max}$ to get an element in 
$\mathcal{H}$ which is bigger than $\Psi_{max}$,
which is a contradiction. See Figure \ref{fig:AffineExtension}
for an illustration.
\end{proof}

For integral affine manifolds with focus points, the local-global convexity
principle no longer works, and indeed, as we shall see, there are many
examples of locally convex but globally non-convex integral affine manifolds
with focus points (which are base spaces of toric-focus integrable systems),
even in dimensions 2 and 3.

However, when there is only one focus point,
then the local-global convexity principle still holds,
as the following propositions show.
Recall from the previous section that a 2-dimensional focus 
box is an affine convex quadrilateral  figure (of any size, 
not necessarily small) with one focus point in it.

\begin{proposition} \label{prop:FocusBox1}
Let $B$ be a 2-dimensional focus box and
$C$ a closed connected subset of  $B$
which is strongly locally convex in $B$. Then $C$ is  convex.
If, moreover, $C$ contains the focus point then it is strongly
convex in $B$.
\end{proposition}

\begin{proof}
Since $B$ is a focus box, it is compact. Denote by $O$ the
focus point in $B$.

\textsf{Case 1: $O \notin C$:} It is not possible to invoke Lemma
\ref{lemma_local_global_flat}, because $C$ is not affinely diffeomorphic to a
subset of $\RR^2$, only locally. However, the proof of Lemma
\ref{lemma_local_global_flat} works. For any $x \in C$ define $S(x)= \{y \in C
\mid [x,y] \subset C\}$ to be the $x$-star in $C$. Since none of the segments in
the definition of $S(x)$ contains $O$, one can repeat the proof of Lemma
\ref{lemma_local_global_flat}, even though there may be more than one straight
line segment joining $x$ to $y$ in the affine structure of $B$.

\textsf{Case 2: $O \in C$:}  Cut $B$ into two parts $B_+ = \{F \geq 0\}$
and $B_- = \{F \leq 0\}$, where $F$ is the smooth single valued affine coordinate
of the box, $F(O) =0$. Then $B_+$ and $B_-$ are affinely isomorphic to
convex \emph{polygons} in $\mathbb{R}^2$ (if we forget about the singular point $O$), so
$C_+ = C \cap B_+$ and $C_- = C \cap B_-$ are locally convex, hence
they are disjoint unions of convex sets.

The intersection $D = C \cap \{F =0\} \ni O$ is also locally convex, and so is a
finite union of closed intervals. If $D$ is not connected, say $D$ has connected
components $D_1, D_2, \hdots$, then $D_1$ must be connected to some $D_i, i \geq 2$
by a path in $C_+$ or  $C_-$ because $C$ is connected (and hence path connected, because
it is also locally convex). But then $C_+$ or $C_-$ will contain a connected component (which contains $D_1$ and $D_i$) which is not convex (because there is a hole between $D_1$ and $D_i$,
contradicting Lemma~\ref{lemma_local_global_flat} applied to this component. Thus $D$ is an interval, which implies that $C_+$ and $C_-$ are
connected, hence they are each convex by Lemma~\ref{lemma_local_global_flat}.

Denote by $x$ and $y$ the two end points of $D$, $x$ is on the left and $y$ is
on the right of $O$ (or equal to $O$). Then, because of local convexity at $x$
and $y$ and convexity of $C_+$ and $C_-$, there are straight lines $\ell_l$ and
$\ell_r$, which pass through $x$ and $y$ respectively,  such that $C$   lies on
the right of $\ell_l$ and on the left of $\ell_r$. If, for example, $y=O$ then
$\ell_r$ is a straight line ``coming from the right''. It means that, in the
multi-valued affine coordinate system $(F,G)$ of the box, where $G$ has 2
branches $G_l$ and $G_r$, $\ell_l$ is given by an affine equation of the type
$G_l + a_lF = b_l$ while $\ell_l$ is given by an affine equation of the type
$G_r+ a_rF = b_r$.

If $p,q$ are two arbitrary different points in $C$, and they are both in $C_+$
or both in $C_-$, then there is only one straight line from $p$ to $q$ in $B$
and that straight line also lies in $C$.

Consider the opposite case, when $F(p) < 0$ and $F(q) > 0$ (or vice versa). Then
for both  potential straight lines $\gamma_l$ from $p$ to $q$ constructed in
Subsection \ref{subsection:Local1}, we have that the intersection of this
(straight or broken) line lies on the right of $x$ (because both points $p$ and
$p$ lie on the right of $\ell_l$), and if this point also lies on the left of
$O$, or coincides with $O$, then $\gamma_l$ is a true straight line from $p$ to
$q$. The same for  $\gamma_r$. At least one of the two lines $\gamma_l$ and
$\gamma_r$ must be straight, and whichever line is straight, is a straight line
in $C$. The strong convexity of $C$ is proved.
\end{proof}

\subsection{Angle variation of a curve on an affine surface} \hfill

This subsection is an adaptation of some old ideas and results
from Zung's thesis (see \cite{Zung-Integrable1993}), which were inspired
by Milnor's paper \cite{Milnor1958} on the non-existence of (locally
flat regular) affine structures on closed surfaces of genus greater
than or equal to 2.

In this subsection we define and show some results about the
\textbf{\textit{angle variation}} of a closed curve on an oriented
affine surface, which will be used to show topological restrictions on
affine surfaces with focus points which are locally convex at the boundary.

Since the word "boundary" in the literature may mean 
different things depending on the context, to avoid confusions,
let us spell out the notion of boundary used in the rest of our paper for subsets of 2-dimensional surfaces:

\begin{definition}
\label{def:boundary}
Let $\cB$ be a topological space such that the set $\cB_0$ 
of points $x \in \cB_0$ which admit a neighborhood homeomorphic to a 2-dimensional disk is open and dense in $\cB$. Then we say that
$\cB$ is a 2-dimensional surface (in a generalized sense), and the points in $\cB \setminus \cB_0$ are called boundary points of $\cB$. If $C$ is a closed subset of $\cB$ then by the boundary of $C$ we mean the set of points of $C$ which lie on the boundary of $\cB$ or on the topological frontier of $C$ in $\cB$.  
\end{definition}

Let $\cB$ be a smooth surface equipped with an affine structure with
or without focus points. We assume that $\cB$ is
\textit{oriented}; if $\cB$ is not orientable, then we
work on an oriented covering of $\cB$ instead of $\cB$ and, of course, if a
double covering of $\cB$ is convex then $\cB$ is also be convex.

Let $\gamma: [0,1] \to \cB$ be a continuously differentiable closed curve (i.e.,
$\gamma(0) = \gamma(1)$, $\dot{\gamma}(0) = \dot{\gamma}(1)$) on $\cB$ which
does not contain focus points of $\cB$ and whose velocity is non-zero
everywhere: $\dot{\gamma}(t) \neq 0,\,\forall\ t\in [0,1]$. We say that such a
curve is \textbf{\textit{non-critical}}. Put a smooth conformal structure on
$\cB$ and take a nonzero tangent vector $v \in T_{\gamma(0)}\cB$. We define the
\textit{angle variation} of $\gamma$ with respect to the affine structure of
$\cB$, the choice of $v$, and the conformal structure\footnote{Recall that a
conformal structure on a real smooth manifold $\mathcal{B}$ is an equivalence
class of Riemannian metrics, where two metrics are equivalent if one is a
multiple of the other via a smooth strictly positive function.} as follows.

Denote by $v(t) = v(t, \gamma)$ the parallel transport  of
$v$ by the affine structure of $\cB$ along $\gamma$ from $\gamma(0)$
to $\gamma(t)$ for each $t \in [0,1]$. Define
$A(t)$ to be the angle from $v(t)$ to $\dot{\gamma}(t)$
with respect to the chosen conformal structure.
Our angle $A(t)$ is algebraic, in the sense that
it can be negative and can be greater than $\pi$.
With every given choice of $A(0)$ (it is only unique up to a multiple of $2\pi$)
there is a unique choice of $A(t)$ for each $t\in [0,1]$ such that
$t \mapsto A(t)$ is continuous.

\begin{definition}
With the above assumptions and notations, the value
\begin{equation*}
AV(\gamma;v) := AV(\gamma; v, \cC) = A(1) - A(0)
\end{equation*}
is called the \textbf{\textit{angle variation}} of $\gamma$ on the affine
surface $\cB$ with respect to the conformal structure $\cC$ on $\cB$
and non-zero vector $v \in T_{\gamma(0)}(B)$.
\end{definition}

We collect some simple and useful facts about the angle variation.

\begin{lemma}
\label{lem:AV1}
With the above notations, we have:

{\rm(i)} If two $C^1$ non-critical
curves $\gamma$ and $\mu$ have the same initial point
($\gamma(0) = \mu(0)$) and are homotopic by a 1-dimensional
family of non-critical curves
with the same initial point, then $AV(\gamma; v, \cC) =
AV(\mu; v,\cC) $ with respect to any vector $v \in T_{\gamma(0)}\cB$.

{\rm(ii)} If $\cC_1$ and $\cC_2$ are two conformal structures on $\cB$
which coincide at $\gamma(0)$, then
$AV(\gamma; v, \cC_1) = AV(\gamma; v, \cC_2)$, i.e., we only have
to specify the conformal structure at the initial point $\gamma(0)$.

{\rm(iii)} If $AV(\gamma; v, \cC) = m \pi$ for some  integer $m$,
then $AV(\gamma; v, \cC') = m \pi$ for any other conformal
structure $\cC'$. If $m\pi < AV(\gamma; v, \cC) < (m+1) \pi$, then
$m\pi < AV(\gamma; v, \cC') < (m+1) \pi$ for any other $\cC'$.

{\rm(iv)} For any two different non-zero vectors $v$ and $v'$ we have
\begin{equation*}
|AV(\gamma; v', \cC)  - AV(\gamma; v, \cC) | < \pi.
\end{equation*}

{\rm(v)} If $\gamma$ bounds a regular disk
(without focus points) in $\cB$ then $AV(\gamma; v, \cC) = 2\pi$
for any $v$ and $\cC$,
provided that $\gamma$ is positively oriented with respect to the disk.

{\rm(vi)}  If $\cB$ (or a neighborhood of $\gamma$) can be affinely
immersed into $\mathbb{R}^2$,
then $AV(\gamma; v, \cC)$ is nothing but the
\textbf{\textit{winding number}} of $\gamma$ times $2\pi$.

{\rm(vii)} If $\gamma$ is an affine straight line,
then $AV(\gamma; \dot{\gamma}(0), \cC) = 0$.

{\rm(viii)}  If $\gamma$ is locally convex, in the sense that the set
of points near $\gamma$ and lying on $\gamma$ or ``on the left'' of
$\gamma$ is locally convex at $\gamma$, then
$AV(\gamma; \dot{\gamma}(0), \cC) \geq 0$. If $\mu$ is homotopic by a
path of non-critical closed curves to a locally convex $\gamma$, then
$AV(\mu; v) > -\pi$ for any $v$.
\end{lemma}

\begin{lemma}[Collars of handles have negative angle variation]
\label{lem:AVHandle}
Let $\gamma$ be a  $C^1$ simple closed curve on $\cB$
which bounds a compact orientable domain $T$ of genus $\geq 1$
which does not contain any focus point
and such that $T$ ``lies on the left''  of
$\gamma$ with respect to the orientation of $\cB$.
Then we have
\begin{equation*}
AV(\gamma; v) < 0
\end{equation*}
with respect to any vector $v$ (and any conformal structure $\cC$).
\end{lemma}

\begin{figure}[!ht]
\vspace{-15pt}
\centering
 {\mbox{} \hspace*{0cm }\includegraphics[width=1 \textwidth]{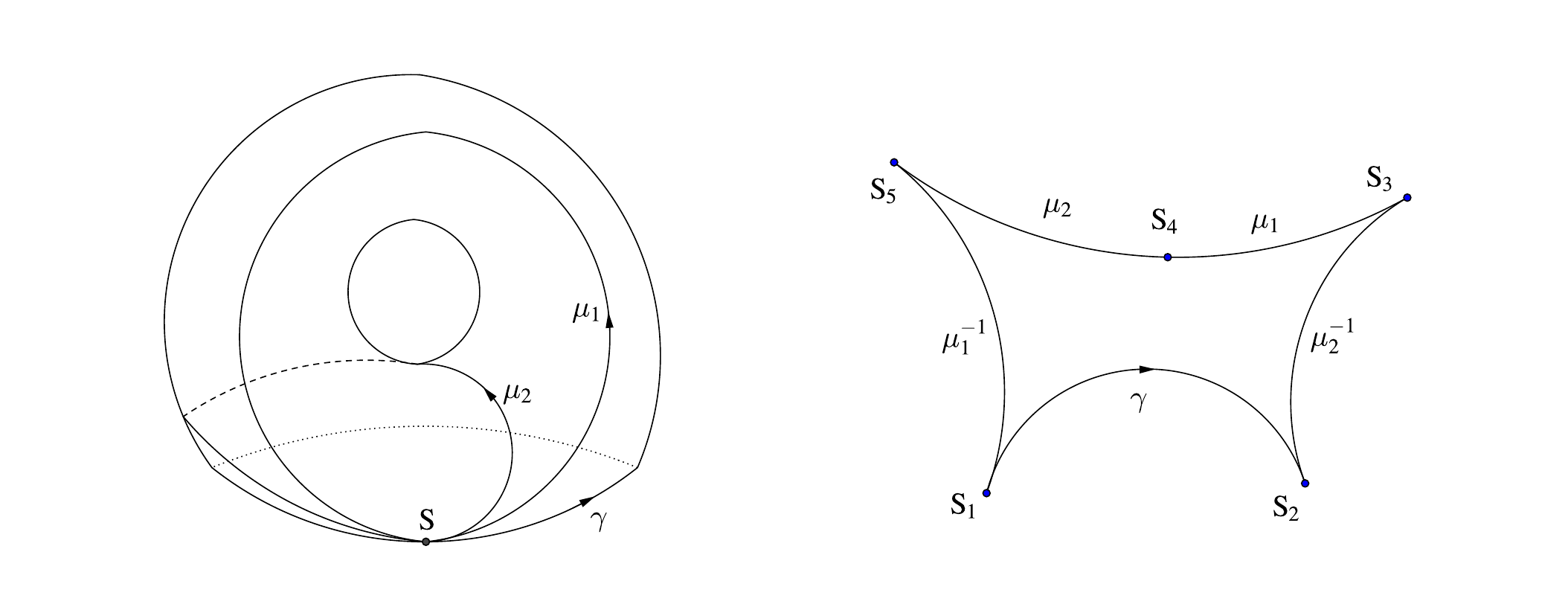} }

\vspace{-15pt}
  \caption{Collars of handles have negative angle variation} \label{fig:AV1}
\end{figure}

\begin{proof}
We will give the proof if the genus is 1. (The higher genus case is
proved in a completely similar manner.)
Draw two simple closed curves $\mu_1$ and $\mu_2$ on the handle $T$,
which are tangent to $\gamma$ at the point $S = \gamma(0) = \gamma(1)
= \mu_1(0) = \mu_1(1) = \mu_2(0) = \mu_2(1)$,
as shown in the left part of Figure \ref{fig:AV1}. Cutting $T$ by
$\mu_1$ and $\mu_2$, we get a disk $T'$ with the boundary consisting
of five consecutive paths $\gamma$, $\mu_2^{-1}$, $\mu_1$, $\mu_2$,
$\mu_1^{-1}$, as shown in the right part of Figure \ref{fig:AV1}.
Notice that, after this cutting,
$S$ becomes five points $S_1, S_2, S_3, S_4, S_5$, with U-turns at
$S_1$, $S_2$, $S_3$, $S_5$ and has $C^1$ continuation at $S^4$ for the
boundary of the  disk $T'$. The total angle variation of the boundary of
$T'$ would be $2\pi$, but since each U-turn accounts for $\pi$ in this
variation, if we count the sum of the angle variations of the five
pieces of the boundary of $T$, it is only  $2\pi - 4 \pi = -2 \pi$. In
other words, we have
\begin{equation}
\label{eq:-2pi}
 AV(\gamma; v) + AV(\mu_2^{-1}; v') +
 AV(\mu_1; v'') + AV(\mu_2; v''') + AV(\mu_1^{-1}; v'''')
 = -2\pi,
\end{equation}
where  $v'$ is the parallel transport of $v$ by $\gamma$
with respect to the affine structure, $v''$ is the parallel transport of
$v'$ by $\mu_2^{-1}$ with respect to the affine structure, and so on.

According to Lemma \ref{lem:AV1} we have
$-\pi < AV(\mu_2^{-1}; v') + AV(\mu_2;v''') $ and
$-\pi < AV(\mu_1; v'') +  AV(\mu_1^{-1};v'''')$. These
two inequalities together with equality \eqref{eq:-2pi} imply that
$AV(\gamma, v) < 0$.
\end{proof}

We say that a non-critical closed curve $\gamma$ is of
\textbf{\textit{completely negative angle variation}} if $AV(\gamma, v, \cC) <
0$ for any choice of $v$ and $\cC$. So Lemma \ref{lem:AVHandle} says that the
collar of a handle is of completely negative angle variation. It is easy to see
that the property of completely negative angle variation is invariant under
homotopy.

 \begin{lemma}
 If $\gamma_s$ ($s \in [0,1]$)
 is a continuous family of  non-critical closed curves,
 then $\gamma_0$ is of completely negative angle variation if and only if
 $\gamma_1$ is of  completely negative angle variation.
 \end{lemma}

 The following lemma says what happens when the homotopy passes over
 a focus point.

\begin{lemma} \label{lem:AVFocus1}
Assume that two closed non-critical curves $\gamma$ and $\mu$ on $\cB$ form the
boundary of an annulus which contains exactly one focus point $O$ inside it, and
are oriented in such a way that the annulus ``lies of the left'' of $\mu$ and
``lies on the right'' of $\gamma$.
If $\gamma$ is of completely negative angle
variation, then $\mu$ is also of completely negative angle variation.
\end{lemma}

\begin{proof}
The reason is simply that the affine monodromy around a focus point in the
positive direction pushes every tangent vector to the left (except for one
direction which remains unchanged under this monodromy). When $v(t)$
is pushed to the left then the angle between $v(t)$ and $\dot{\gamma}(t)$
becomes smaller, and so the angle variation becomes even more negative.
\end{proof}

\begin{lemma} \label{lem:AVFocus2}
Assume that two  closed non-critical curves $\gamma$ and $\mu$
on $\cB$ form the boundary of an annulus which contains
a finite number of focus points  inside it,
and are oriented in such a way that the annulus
``lies of the left''
of $\mu$ and ``lies on the right'' of $\gamma$.  If $\gamma$
is of negative angle variation,  then $\mu$ is also
of negative angle variation.
\end{lemma}

\begin{proof}
Just apply Lemma \ref{lem:AVFocus1} $m$ times, where $m$ is the
number of focus points in the annulus.
\end{proof}

 \begin{lemma} \label{lem:AV8}
 If $\gamma_1, \gamma_2, \hdots, \gamma_k$ ($k \geq 3$)
 are closed non-critical curves
 which bound a regular domain $\cD$ (without focus points) of genus 0
 and are positively oriented with respect to $\cD$, and such that
$\cD$ is locally convex at $\gamma_2, \hdots, \gamma_k$, then
$\gamma_1$ is of completely negative angle variation.
 \end{lemma}

\begin{proof}
The proof is similar to the proof of Lemma \ref{lem:AVHandle}.
For example, if $k=3$, we can draw curves
$\mu_2$ and $\mu_3$ which are homotopic by paths of
non-critical closed curves to $\gamma_2$ and $\gamma_3$ respectively,
such that $AV(\gamma_1; v_1, \cC) + AV(\mu_1; w_2, \cC) +
AV(\mu_3,w_3, \cC)  = -2\pi$, where $v_1$ is arbitrary and the choice
of $v', v''$ depends on $v$. Then use the inequalities
$AV(\mu_i; w_i, \cC) > -\pi$ to conclude the statement.
\end{proof}

\begin{proposition}\label{prop:TopoLocallyConvex2Dim}
Let $C$ be a compact subset of an orientable affine surface $\cB$
with (or without) focus points, such that $C$ has non-empty boundary
and is locally convex at its boundary. Then $C$ has no handle
(i.e., it can be embedded into $\mathbb{R}^2$) and has at most two
boundary components.
\end{proposition}

\begin{proof}
Just put together Lemma \ref{lem:AVHandle}, Lemma \ref{lem:AVFocus2},
and Lemma \ref{lem:AV8}.
\end{proof}

\begin{remark}
{\rm
In \cite{Zung-Integrable1993}, it was shown that the conclusion of
Proposition \ref{prop:TopoLocallyConvex2Dim} still holds for each regular domain of the base space even
in the case  with hyperbolic singularities.
Proposition \ref{prop:TopoLocallyConvex2Dim} was also mentioned in
\cite{Zung-Integrable2003}(without a proof), and proved by Leung and Symington
 in \cite{LeuSym-AlmostToric2010} .
}
\end{remark}
\subsection{Convexity of compact affine surfaces with non-empty boundary} \hfill

The Local-Global Principle (see Proposition \ref{prop:FocusBox1})
in a focus box allows us to prove  global convexity results in
dimension 2, for base spaces of toric-focus integrable systems on
symplectic 4-manifolds.

First we consider the compact case.

\begin{theorem}
\label{thm_gobal_convex_2D_compact}
Let $\mathcal{B}$ be a closed connected subset of the base space of a 
toric-focus integrable Hamiltonian system on a connected, compact, 4-dimensional symplectic manifold
Assume $\cB$ is locally convex with respect to the associated affine structure, and
that $\cB$ has a nonempty boundary. 
Then $\mathcal{B}$ is convex. Moreover, if $\cB$ is orientable, then it is
topologically a disk or an annulus. If $\cB$ is an annulus, there is a global
single-valued non-constant affine function $F$ on $\cB$ such that $F$ is
constant on each of the two boundary components of $\cB$ and, in particular, the
boundary components of $\cB$ are straight curves.
\end{theorem}

In order to prove the above theorem, we will use the \textbf{\textit{shrinking
method}} to define and study  \textbf{\textit{convex hulls}} of subsets on
affine  manifolds with singular points. In the Euclidean space, the convex hull
of a given subset is unique. In our case,  it may be non-unique, but exists and
can be defined as follows.

By compactness of $\mathcal{B}$, there is a finite set of closed focus boxes and
regular boxes $\Sigma:= \{B_i \subset \mathcal{B} \mid i \in I, \, I \,
\text{finite}\}$ whose interiors cover $\mathcal{B}$. We need to introduce the
following definition:

\begin{quote}
$C \subset \mathcal{B}$ \textit{is} $\Sigma$-{\bfi convex} \textit{if}
$C \cap B_i$ \textit{is strongly convex in $B_i$  (see Definition
\ref{def_strong_convexity}) for all} $B_i \in \Sigma$.
\end{quote}

Note that
if $C$ is $\Sigma$-convex, then $C$ is locally convex since each $B_i$ is
convex (see Theorem \ref{thm:FF2Local}). For each given non-empty subset
$S$ of $\cB$,  the family
\begin{equation*}
\mathcal{C}_S = \{C \subset \mathcal{B} \mid S \subset C,\, C \, \text{is
closed, connected, $\Sigma$-convex}\}
\end{equation*}
of subsets of $\mathcal{B}$ is not empty since $\mathcal{B} \in
\mathcal{C}_S$.

We claim that $\mathcal{C}_S$ \textit{has a minimal element.}
Indeed, $\cC_S$ is a partially ordered set with respect to inclusion.
Let $\{C_\alpha\}\subset \mathcal{C}_S$ be a totally ordered subset of $\cC$, i.e.,
for any two elements in
$\mathcal{C}$ one is included in the other.
Let $C_\infty = \cap_\alpha C_\alpha$. Clearly $S \subset C_\infty$ and $C_\infty$ is strongly
convex in each box $\mathcal{B}_i$ by Proposition \ref{prop:FocusBox1}.
The set $C_\infty$ is
also connected (see \cite[Theorem 6.1.18, page 355]{Engelking1989}) and,
therefore, $C_\infty \in \mathcal{C}_S$. By Zorn's Lemma, the set
$\mathcal{C}_S$ has a minimal element $C_S$, which is hence connected, closed, $\Sigma$-convex.

We will call a minimal element $C_S$ of  $\mathcal{C}_S$
a \textbf{\textit{$\Sigma$-convex hull of $S$ in $\cB$}}. It exists but is not necessarily unique. 

\begin{lemma} \label{lem:ConvexHull}
With the above notations, under the assumption that $\cB$ has non-empty boundary,
any $\Sigma$-convex hull $C_S$ of a non-empty set $S \subset B$
has a non-empty boundary (in the sense of Definition
\ref{def:boundary}). Moreover, each connected component of the boundary of $C_S$
contains at least one point of $S$, and is either a single point or 
locally a straight line outside the points of $S$.
\end{lemma}

\begin{proof}
Since the boundary of $\mathcal{B}$ is not empty, neither is
the boundary of $C_S$. Indeed, if the boundary of $C_S$ were
empty, then $C_S$ would be open in $\cB$ and different from $\cB$,
which contradicts the  closedness of $C_S$ and the 
connectedness of $\cB$.

If $C_S$ is one-dimensional, then the boundary of $C_S$ is $C_S$ itself,  and it contains $S$, so we are done.

If $C_S$ is two-dimensional, it follows that the interior of $C_S$ is not empty. Let $\mu$ be a connected component of the boundary of $C_S$. We know that $\mu$ must
be homeomorphic to a circle and that $C_S$  is locally convex at $\mu$.

\begin{figure}[!ht]
\vspace{-15pt}
\centering
 {\mbox{} \hspace*{0cm }\includegraphics[width=0.6 \textwidth]{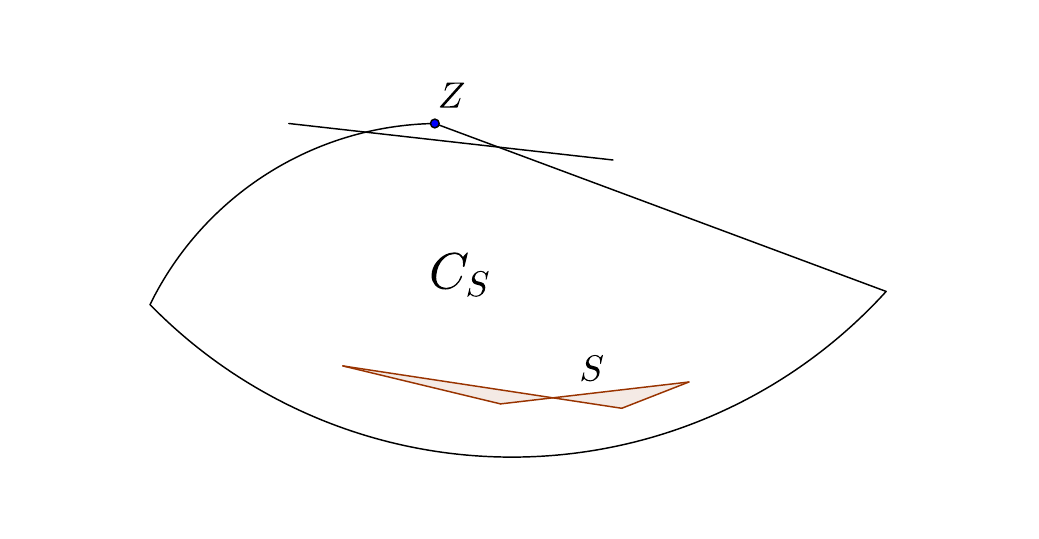}}

\vspace{-15pt}
  \caption{Making $C_S$ smaller by cutting out a corner at $Z$.}
  \label{fig:Cutting1}
\end{figure}

For every $z \in \mu$ such that  $z \notin S$ and $z$
is not a focus point, we have that $\mu $ is locally straight at $z$.
Indeed, if $\mu$ is not straight at
$z$, we can cut a small piece of $C_S$ with $z \in \mu$
being a vertex and the resulting set is still $\Sigma$-convex, closed,
connected, containing $S$ and it is strictly included in $C_S$ (see
Figure \ref{fig:Cutting1}). This contradicts minimality of $C_S$.

If $z$ is a focus point, then because $C_S$ is convex at $z$, we can
either cut out a small piece containing $z$ just like in the previous
case, or the boundary component $\mu$ must be locally a limit of straight
lines from the interior of $C_z$. Since $C_z$ is minimal, this is the
only possible case. By forgetting about the complement of $C_z$,
we may forget the fact that $z$ is a focus point and pretend that
it is a regular point and have the same situation as in the previous case.

Thus,  if  $\mu$ does not contain any point of $S$, then $\mu$ is a
straight circle, and we can push the boundary $\mu$ a bit into the
interior of $C$ to another straight curve $\mu '$, and then cut the
``collar'' bounded by $\mu$ and $\mu'$. (Here we use the fact that the
affine structure is integral.) The resulting set is clearly in
$\mathcal{C}_S$ if the collar that has been cut out is sufficiently thin.
This contradicts the minimality of $C_S$ in $\mathcal{C}_S$.
\end{proof}

\noindent {\it Proof of Theorem \ref{thm_gobal_convex_2D_compact} in the case 
when $\cB$ is not a disk}.

We may assume that $\cB$  is orientable (if not just take an orientable
covering of it). Then by Proposition \ref{prop:TopoLocallyConvex2Dim},
if $\cB$ is not a disk, it must be topologically an annulus.

Denote the two components of the boundary of  $\cB$ by $\gamma$ and $\mu$.
Notice that a $\Sigma$-convex hull $C_\gamma$ of $\gamma$ in $\cB$ must
be 1-dimensional, for otherwise it must have some other boundary
component $\mu'$ disjoint from $\gamma$, and we can still shrink
$C_\gamma$ near $\mu$, which is a contradiction. However, if $C_\gamma$
is 1-dimensional, then it is equal to $\gamma$, i.e., $\gamma$ must be
a $\Sigma$-convex set in $\cB$, and it follows that it is a straight line.

Apply the same argument to $\mu$ to conclude that both $\gamma$ and
$\mu$ are straight lines. Now we want to show the existence
of a global single-valued integral affine function $F$ such that $F$
is constant on $\gamma$ and on $\mu$. We do this by induction on the
number of focus points in $\cB$.

If  $\cB$ does not contain any focus point, then the universal covering
of $\cB$ is affinely isomorphic to a part of $\mathbb{R}^2$ by two
straight lines, and these lines must be parallel. From that we can
construct our affine  function $F$.

Assume now that the statement is true when there are less than $n$ focus
points ($n \in \mathbb{N}$). Let us show that the statement is also
true when there are exactly $n$ focus points.

 \begin{figure}[!ht]
\vspace{-15pt}
\centering
{\mbox{} \hspace*{0cm} \includegraphics[width=0.8 \textwidth]{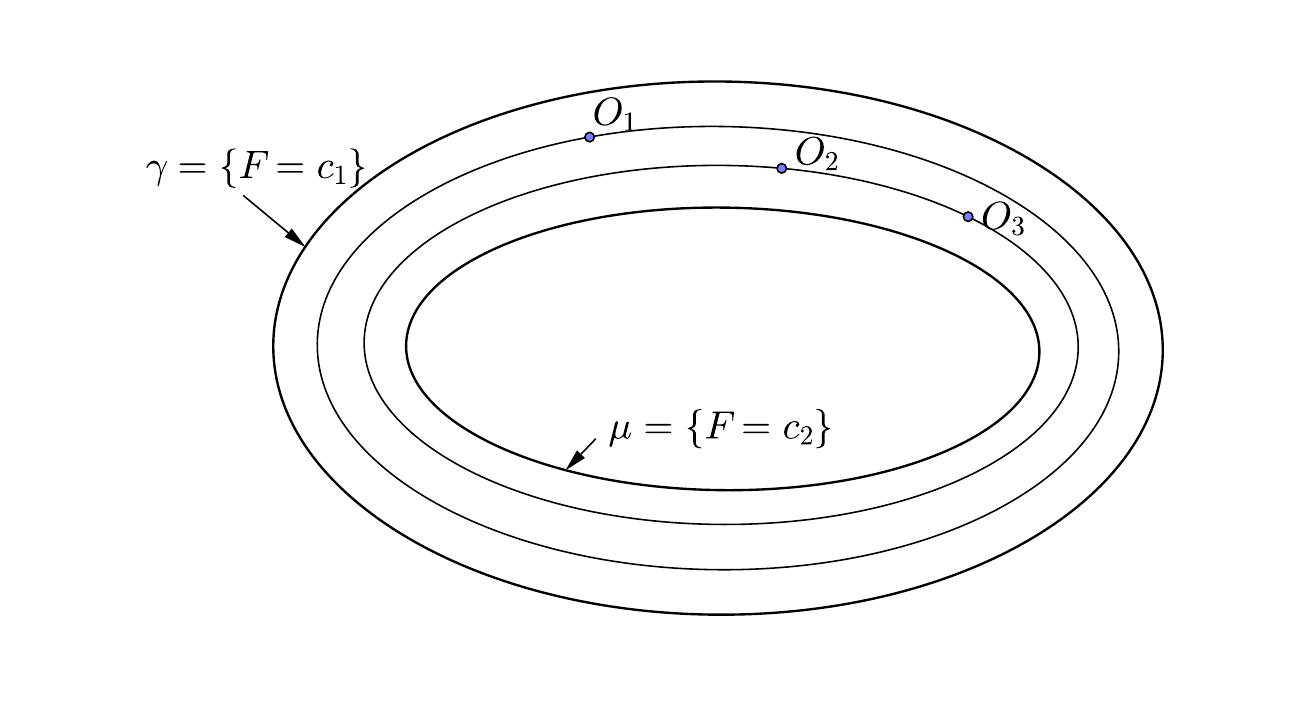}}

\vspace{-15pt}
  \caption{Convex annulus.}
  \label{fig:ConvexAnnulus}
\end{figure}

Denote the set of focus points by $\cO = \{O_1,\hdots, O_n\}$ and
consider a $\Sigma$-convex hull of the union $\gamma \cup \cO$ in $\cB$.
Its boundary must contain $\gamma$ but not only $\gamma$,
so it is an annulus, and has
another boundary component, say $\mu'$,
which contains a focus point, say $O_n$  (see Figure \ref{fig:ConvexAnnulus}).
 Then $\mu'$ must be  a
straight line and, by induction, there is a single-valued non-constant
affine function $F$ on the annulus between $\gamma$ and $\mu'$, which is
constant on $\gamma$ and on $\mu'$. It follows that the strip between
$\mu'$ and $\mu$ is also locally convex, and hence it is  an annulus on
which we have a non-constant affine function $F'$ such that $F'$ is
constant on $\mu'$ and on $\mu$. It is easy to see that we can ``glue''
$F$ with $F'$ after some affine transformation to get a non-constant
affine function on $\cB$ which is constant on $\gamma$ and on $\mu$.
The global convexity of the annulus is now easy to see, by the same method
of potential straight lines as in the proof of Theorem \ref{thm:FF1Local}.
\QED

\noindent {\it Proof of Theorem \ref{thm_gobal_convex_2D_compact} in
the case when $\cB$ is a disk}.

Take $x, y \in \mathcal{B}$. We need to show the existence of a
straight line segment $[x,y] \subset \mathcal{B}$.

Denote by $C_0 = C_{\{x,y\}}$ a $\Sigma$-convex hull of $x$ and $y$
in $\cB$. If $C_0$ is one-dimensional, it follows that it is locally
straight (in the affine structure of $\mathcal{B}$) so it is straight and
hence there is a straight line $[x,y] \subset C_0 \subset \cB$.

If $C_0$ is two-dimensional, it follows that the interior of $C_0$ is not
empty. Let $\mu$ be a boundary component of $C_0$. We know that $\mu$ must
be homeomorphic to a circle and that $C_0$  is locally convex at $\mu$.
So, according to Lemma \ref{lem:ConvexHull}, $\mu$ must contain
$x$ or $y$, or both of them, and $\mu$ is locally straight outside these
points.

If both $x$ and $y$ belong to $\mu$ then both paths
on $\mu$ from $x$ to $y$ are straight, and we are done.

Consider now the case when $y \in \mu$ but $x \notin \mu$ (or vice versa).
Then Lemma \ref{lem:DiskSections} below states, in particular, that
there is a straight line from $x$ to $y$ in $C_0$, and we are again done.
\QED

\begin{lemma}
\label{lem:DiskSections}
Let $D$ be a $\Sigma$-convex closed disk on $\cB$ with boundary $\mu$
and $x\in D$.

{\rm(i)} {\rm(}See Figure \ref{fig:DiskSectors}{\rm)}.
If $x$ is in the interior of $D$, then there exist a finite number of
consecutive points $A_1,\hdots, A_n \in \mu$ ($n \geq 1$) and
$n$ corresponding disjoint closed segments $I_1, \hdots, I_n$ on the
circle of  nonzero tangent vectors at $x$, such that:

-  Each straight ray emanating from $x$ in direction $d$ belonging to
the interior of $I_i$ hits a point $y(d)$ on the ``arc'' $A_iA_{i+1}$
on $\mu$ (with the convention that $A_{n+1} = A_1$;

- The above straight line $[x, y(d)]$  lies in $D$, is regular in $D$
(i.e., it does not  hit any focus point in $D$),
and the map $d \mapsto y(d)$  is a homeomorphism from the interior
of $I_i$ to the interior of the ``arc'' $A_iA_{i+1}$ on $\mu$;

- When $d$ tends to the end points of $I_i$, then $[x, y(d)]$ tends to
straight lines going from $x$ to $A_i$ and $A_{i+1}$ (these straight
lines may contain focus points).

\begin{figure}[!ht]
\vspace{-15pt}
\centering
 {\mbox{} \hspace*{0cm }\includegraphics[width=0.7 \textwidth]{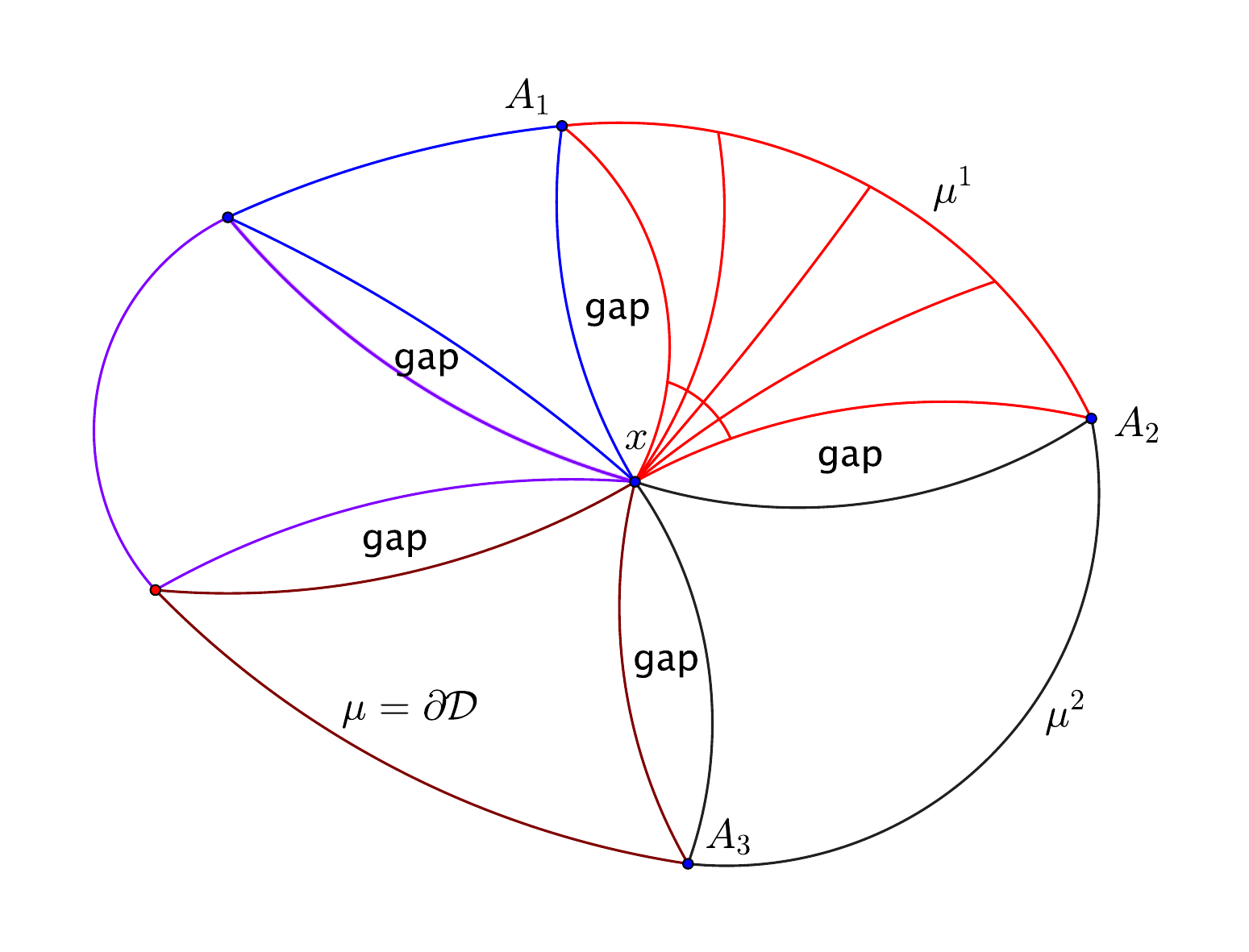}}

\vspace{-15pt}
  \caption{Dividing the disk into sectors and gaps.}
  \label{fig:DiskSectors}
\end{figure}

{\rm(ii)} {\rm(}See Figure \ref{fig:DiskSectors2}{\rm)}.
If $x \in \mu$, then either $\mu \setminus \{x\}$ is a straight line, or
there exist a finite number of consecutive points
$A_1,\hdots, A_n \in \mu$ ($n \geq 1$, it may happen than $A_1 = x$
or $A_n = x$, or both) and  $n-1$ corresponding disjoint closed
segments $I_1, \hdots, I_{n-1}$ on the circle of  nonzero tangent
vectors at $x$, such that, the ``arcs'' $xA_1$ and $xA_n$ on $\mu$
are straight lines, and the other ``arcs'' $A_iA_{i+1}$ on
$\mu$ {\rm(}$1 \leq i \leq n-1${\rm)} satisfy the same properties as in the
previous case:

-  Each straight ray emanating from $x$ in direction $d$ belonging to
the interior of $I_i$ hits a point $y(d)$ on the ``arc'' $A_iA_{i+1}$
on $\mu$;

- The above straight line $[x, y(d)]$  lies in $D$, is regular in $D$,
and the map $d \mapsto y(d)$ is a homeomorphism from the interior of
$I_i$ to the interior of the ``arc'' $A_iA_{i+1}$ on $\mu$;

- When $d$ tends to the end points of $I_i$, then $[x, y(d)]$ tend to
straight lines going from $x$ to $A_i$ and $A_{i+1}$.

\begin{figure}[!ht]
\vspace{-15pt}
\centering
 {\mbox{} \hspace*{0cm }\includegraphics[width=0.85 \textwidth]{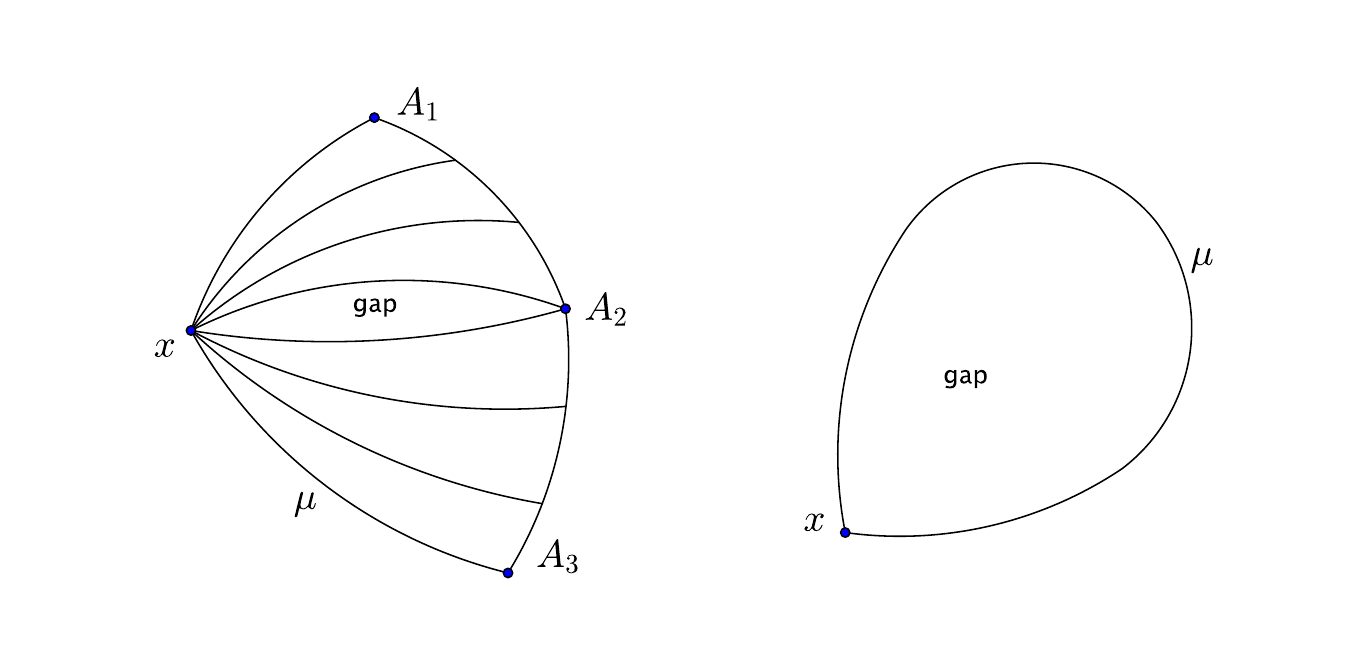}}

\vspace{-15pt}
  \caption{Special cases when $x$ is on the boundary of the disk.}
  \label{fig:DiskSectors2}
\end{figure}
\end{lemma}

In other words, the above lemma says that $D$ can be divided into
\textbf{\textit{sectors}} and \textbf{\textit{gaps}}: each sector is a
regular ``convex cone'' with $x$ as the vertex.
The interiors of the sectors do not overlap and the whole boundary
$\mu$ is covered by theses sectors. To go from $x$ to the points on
$\mu$, we can move in the sectors and forget about the  gaps.

\begin{proof}
We prove the lemma by induction on the number of focus points in the
interior of $D$. (Focus points on the boundary of $D$ can be ignored).

If $D$ does not contain any focus point, then we can apply the regular
local-global convexity principle to $D$ to conclude that $D$ is affinely
isomorphic to a compact convex set in $\mathbb{R}^2$, in which case the lemma
becomes trivial (there is only one sector).

Assume now that there are exactly $n$ focus points $O_1, \hdots, O_n$
in the interior of $D$. Denote by $C = C_{\{x,O_1,\hdots, O_n\}}$
a $\Sigma$-convex hull of $x, O_1, \hdots, O_n$, and let $\gamma$ be
the boundary of $C$. (If $C$ has a hole, then denote by $C'$ the union
of the $C$ with the disk inside the hole and by $\gamma$  the  boundary
of $C'$).
The following cases may happen:

\underline{Case 1.} $x$ lies in the interior of $C$ (or $C'$, if $C$
has a hole) and at least one of the focus points, say $O_1$, lies on
$\gamma$. Then we can apply the induction hypothesis to $C$
(if $C$ is a disk) or the annulus case of Theorem
\ref{thm_gobal_convex_2D_compact}
(if $C$ has a hole) to divide $C$ into sectors and holes with the above
properties.

Let $d$ be any direction at $x$ in one of the sectors of $C$ (or $C'$)
and consider the corresponding straight line $\ell = \ell(d)$ passing
through $x$ and $y(d)$.  After hitting the point $y(d)$ on
$\gamma = \partial C$ (or $\partial C'$), this straight line $\ell$
gets out of $C$ and goes in the regular region between $\gamma$ and
$\mu$. It must then hit $\mu$ at a point, say $z(d)$ (it may happen
that $z(d) = y(d)$, i.e., it already hits $\mu$ at $y(d)$ without
having to extend), or, otherwise, one of the following 3 bad
 possibilities happens (see Figure \ref{fig:3Impossible}):

\begin{figure}[!ht]
\vspace{-15pt}
\centering
 {\mbox{} \hspace*{-1cm }\includegraphics[width=1.2 \textwidth]{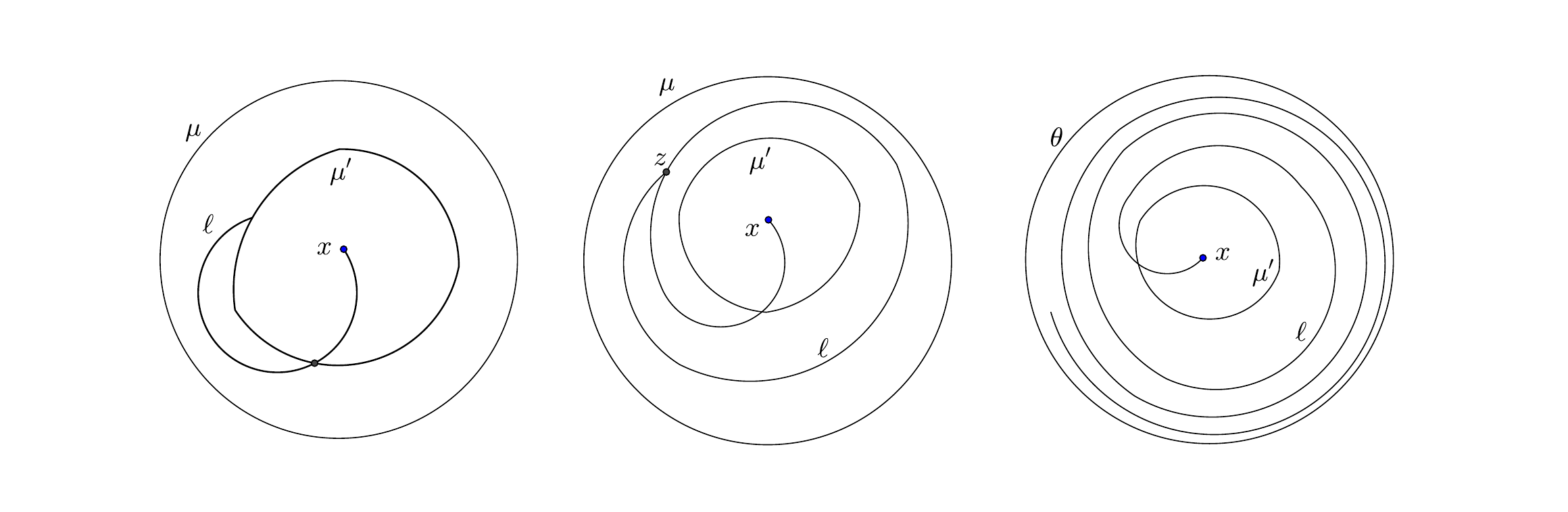}}

\vspace{-15pt}
  \caption{The three impossible situations.}
  \label{fig:3Impossible}
\end{figure}

a) The straight line $\ell$ falls back to $C$ at a point $y'(d)$.
This situation is impossible, because the region between the segment
from $y(d)$ to $y'(d)$ on the straight line and the ``arc'' from $y(d)$
to $y'(d)$ on $\gamma$ can be affinely immersed into $\mathbb{R}^2$.
However, $\gamma$ viewed from $\ell$ is concave while $\ell$ is straight,
a situation that cannot happen in $\mathbb{R}^2$.

b) The straight line $\ell$ cuts itself (after going around). Then we
have locally an annulus between a loop created by the straight line
and $\mu$, whose boundary is not straight (at the self-intersection
point of the straight line), which contradicts the annulus case of
Theorem \ref{thm_gobal_convex_2D_compact}. So this situation
is also impossible.

c) The straight line $\ell$ winds around an infinite number of times
without cutting itself or hitting $\mu$.  Denote by $\mu'$ the limit
set of $\ell$. Then $\mu'$ is a closed straight line union $\ell$. Such
a situation cannot happen in integral affine geometry, because in a
neighborhood of $\mu'$ there would exists a non-constant
affine function $F$ such that $\mu' = \{F = 0\}$, and $F$ is not constant
on $\ell$, implying that $F =0$ on some point of $\ell$, i.e., $\ell$
would have to cut $\mu'$.

So none of the above three ``bad'' situations can happen, which means
that $\ell$ must necessarily cut $\mu$ at a point, say $z(d)$.
The map $d \mapsto z(d)$ maps the segments $I_i$ of directions $d$
at $x$ (which correspond to the  sectors of $C$) to ``arcs''
in $\mu$. Then one can easily see that these arcs overlap and cover the
whole $\mu$. (The focus point on $\gamma$ is also responsible for some
overlapping). In order to avoid overlapping, one simply shrinks the
segments as much as necessary (the choice of shrinking is not unique).
After this shrinking, one gets the required sectors and gaps for $D$.

\underline{Case 2.} $x$ lies in the interior of $D$ and also on the
boundary of $C$. The treatment of this case is similar to Case 1,
with one segment of directions at $x$ added: the one which consists
of the directions which point towards the exterior of $C$.

\underline{Case 3.} $x$ lies in on the boundary $\mu$ of $D$.
If $\mu \setminus \{x\}$ is a straight line, then the conclusion of
the lemma is empty, we have nothing to prove.
If  $\mu \setminus \{x\}$ is not a straight line then $C$ is strictly
included in $D$ and $C$ contains a focus point on its boundary. We can
apply the induction hypothesis to $C$ and treat this case similarly
to Case 2.
\end{proof}

\begin{remark}
\label{global_affine_remark}
{\rm
In the proof of Theorem \ref{thm_gobal_convex_2D_compact} we used
the fact that the affine structure is integral, especially for the
annulus case. However, if there is a \textit{global} affine function
on $\mathcal{B}$, then the fact that the affine structure is integral
is not needed and the conclusions of the theorem still hold. In particular,
if there is a global affine function on $\mathcal{B}$, the indices of the
focus points  are allowed to be non-integers, and $\mathcal{B}$ would still be convex.
}
\end{remark}

\subsection{Convexity in the non-compact proper case} \hfill

Let $\cB$ be the 2-dimensional base space of a toric-focus integrable
Hamiltonian system on a non-compact symplectic 4-manifold
which has both elliptic and focus-focus singularities.
Assume that the number of focus-focus singularities is finite.
The boundary of $\cB$ corresponds to the elliptic singularities of the
system. In this subsection, we assume that the interior of $\mathcal{B}$
is homeomorphic to an open disk.

Fix a point $p$ on the boundary.
Take continuous paths from $p$ to each
focus point of $\cB$ and modify, if necessary, the paths such that
they do not intersect in $\cB$. Now cut along these paths, thereby
obtaining a new set $\widehat{\cB}$, whose boundary consists of
the old boundary of $\cB$ union with twice each path linking the
elliptic point to the focus singularity. This new set does not
contain any point with monodromy, because of the cuts that place the
former focus singularities on the boundary of $\widehat{\cB}$. Thus
we get a global affine map $\varphi: \widehat{\cB} \rightarrow
\mathbb{R}^2$.

\begin{definition}
\label{def_proper}
If $\mathcal{B}$ is such a base space, then $\cB$ is said to be
{\bfi proper} if the affine map $\varphi: \widehat{\cB} \rightarrow
\mathbb{R}^2$ is proper (i.e., the inverse image by $\varphi$
of any compact set in $\mathbb{R}^2$ is compact in $\widehat{\cB}$).
\end{definition}

Notice that the above definition does not depend on the choice of the paths
linking the elliptic singularities on the boundary of $\cB$ to the
focus points inside $\mathcal{B}$. This fact follows easily from the
following lemma:

\begin{lemma}
Let $X$ be a Hausdorff topological space, 
$Y=\overline{Y} \subset X$ and $Z =\overline{Z} \subset X$ are two closed subsets which together cover $X$:
$Y \cup Z = X$. Then a continuous map $f:X \rightarrow T$ to
some other topological space $T$ is proper if and only if both restrictions
$f|_Y:Y \rightarrow T$ an $f|_Z:Z \rightarrow T$ are proper.
\end{lemma}

\begin{proof}
If $K \subset T$, then $f^{-1}(K) =
f^{-1}(K)\cap(Y \cup Z) = \left(f^{-1}(K) \cap Y\right) \cup
\left(f^{-1}(K) \cap Z \right)$.

If $K \subset T$ is compact and $f: X \rightarrow T$ is proper, it
follows that $f^{-1}(K)$ is compact in $X$. Therefore, the inverse
images of $f^{-1}(K)$ by the inclusions $Y \hookrightarrow X$,
$Z \hookrightarrow X$ are compact sets in $Y$ and $Z$,
respectively, because $Y$ and $Z$ are closed. 
These inverse images coincide with $(f|_Y)^{-1}(K)
\subset Y$ and $(f|_Z)^{-1}(K) \subset Z$, which shows that both
$f|_Y:Y \rightarrow T$ an $f|_Z:Z \rightarrow T$ are proper.

Conversely, if $f|_Y:Y \rightarrow T$ an $f|_Z:Z \rightarrow T$ are
proper, then $(f|_Y)^{-1} (K) = f^{-1}(K) \cap Y$ is a compact set in
$Y$ and $(f|_Z)^{-1} (K) = f^{-1}(K) \cap Z$ is a compact set in $Z$. Since
$Y= \overline{Y}$, $Z = \overline{Z}$, it follows that $f^{-1}(K) \cap Y$
and $f^{-1}(K) \cap Z$ are compact in $X$ which then implies that their
union, which equals $f^{-1}(K)$ is compact in $X$.
\end{proof}

Now we can formulate a global convexity
theorem in the proper non-compact 2-dimensional case.

\begin{theorem}
\label{thm:GlobalConvex2Dfirst}
Let $\mathcal{B}$ be the 2-dimensional base space of a toric-focus
integrable Hamiltonian system on a connected, non-compact,
symplectic, 4-manifold without boundary.
Assume:
\begin{itemize}
\item[{\rm (i)}]  the system has elliptic singularities (so
the boundary of $\cB$ is not empty);
\item[{\rm (ii)}] the number of focus points in $\mathcal{B}$ is
finite, and $\mathcal{B}$ minus its boundary is homeomorphic to an open disk;
\item[{\rm (iii)}] $\mathcal{B}$ is proper (see Definition \ref{def_proper}).
\end{itemize}
Then $\mathcal{B}$ is convex (in its own underlying affine structure).
\end{theorem}

\begin{proof} The main idea is to reduce Theorem
\ref{thm:GlobalConvex2Dfirst} to Theorem \ref{thm_gobal_convex_2D_compact},
in a way similar to the reduction of Lemma \ref{lem:convex_affine_map2} to
Lemma \ref{lem:convex_affine_map1}.

Fix the proper integral affine map $\varphi: \widehat{\cB} \rightarrow
\mathbb{R}^2$ given in the definition of the properness of $\cB$.
Then, for $N \in \mathbb{N}$ large enough, the square
$D_N = \{(s,t) \in \mathbb{R}^2 \; | \; |s|, |t| \leq N\}$
contains the images of all the paths used to cut $\cB$, because
these paths are compact and there are only finitely many of them. It
follows that $\varphi^{-1}(D_N)$ corresponds to a locally convex compact
subset $\cB_N$ in $\cB$ whose boundary is not empty.
By theorem \ref{thm_gobal_convex_2D_compact}, we obtain that $\cB_N$
is convex. For any $x,y \in \cB$ there exists $N$ large enough such that
$x, y \in \cB_N$, hence there is a straight line from $x$ to $y$. Thus
$\cB$ is convex.
\end{proof}

\subsection{Non-convex examples in the non-proper case} \hfill
\label{subsec_nonconvex_ex_non_proper_J}

In this subsection we briefly recall some (non)convexity results due to Vu Ngoc, Pelayo and Ratiu 
\cite{VuNgoc2006,PeRaVN2015} in the case of toric-focus integrable Hamiltonian systems on symplectic 4-manifolds
with a global Hamiltonian $\mathbb{T}^1$-action. Let
$\textbf{F}=(J, H): (M^4,\omega) \rightarrow \mathbb{R}^2$ be a proper
momentum map of a toric-focus integrable system with the following additional properties:
\begin{itemize}
\item  $J: (M^4,\omega) \rightarrow \mathbb{R}$ is the momentum map of a Hamiltonian $\mathbb{T}^1$-action;
\item the fibers of $J$ are connected;
\item the  bifurcation set of $J$ is discrete;
\item for any critical value $x$ of $J$,
there exists a neighborhood $V\ni x$ such that the number of connected components of the critical set of $J$ in $J^{-1}(V)$ is finite.
\end{itemize}
Such a system is called a \textit{proper semitoric system}. If, moreover, the function $J$ itself is proper,
then the system is called a \textit{semitoric system}.
(Unfortunately, the above notions are a bit confusing,
because a "semitoric system" is "more proper" than a
"proper semitoric system".) Vu Ngoc \cite{VuNgoc2006}
showed that in the semitoric case one can associate
to the system a family of convex "momentum polygons"
similar to Delzant polytopes, so in a sense we still have
convexity in this case. (This convexity is not the same as our notion of convexity used in this paper, but it is clearly very closely related.)

The case when $\textbf{F}$ is proper but $J$ is not 
proper is fundamentally different, as was shown in \cite{PeRaVN2017}. In the
study of semitoric systems, properness of $J$ plays a crucial role since
it permits the use of Morse-Bott theory and of techniques related to
the Duistermaat-Heckman theorem, ultimately leading to the
proof of the connectedness of the fibers of $\textbf{F}$ and of the convexity
result of the "momentum polygons". These methods
are not available if $J$ is not proper. The consequences of the
loss of properness of $J$ are remarkable. Not only are the proofs totally
different, but even the properties of $\textbf{F}(M)$ change radically.
In \cite{PeRaVN2017}, an invariant generalizing that for semitoric and
toric systems, the \textit{cartographic projection}, was introduced: A cartographic projection is the natural planar representation of the singular affine structure of the proper semitoric
system, and is a union of subsets $\mathcal{R}
\subset \mathbb{R}^2$ of the following four very specific types:
\begin{itemize}
\item $\mathcal{R} \subset \mathbb{R}^2$ is of \emph{type} \textup{I}
if there is an interval $I \subseteq \mathbb{R}$ and
$f, g\colon I \to \mathbb{R}$ such that $f$ is a piecewise linear
continuous convex function, $g$ is a piecewise linear continuous concave
function, and
$$
\mathcal{R}=\Big\{(x,y) \in \mathbb{R}^2\,\,|\, \, x \in I\,\,\,\, {\rm
  and}\,\,\,\, f(x) \leq y \leq g(x)\Big\}.
$$
\item $\mathcal{R} \subset \mathbb{R}^2$ is of \emph{type} \textup{II} if
there is an interval $I \subseteq \mathbb{R}$ and $f\colon I\to\mathbb{R}$,
$g \colon I \to \overline{\mathbb{R}}$ such that $f$ is a piecewise
linear continuous convex function, $g$ is lower semicontinuous, and
$$
\mathcal{R}=\Big\{(x,y) \in \mathbb{R}^2\,\,|\, \, x \in I\,\,\,\,
{\rm and}\,\,\,\, f(x) \leq y < g(x)\Big\}.
$$
\item $\mathcal{R} \subset \mathbb{R}^2$ is of \emph{type} \textup{III}
if there is an interval $I \subseteq \mathbb{R}$ and
$f \colon I \to \overline{\mathbb{R}}$, $g \colon I \to \mathbb{R}$
such that $f$ is upper semicontinuous, $g$ is a piecewise linear
continuous concave function, and
$$
\mathcal{R}=\Big\{(x,y) \in \mathbb{R}^2\,\,|\, \, x \in I \,\,\,\,
{\rm and}\,\,\,\, f(x) < y \leq g(x)\Big\}.
$$
\item $\mathcal{R} \subset \mathbb{R}^2$ is of \emph{type} \textup{IV}
if there is an interval $I \subseteq \mathbb{R}$ and
$f,g \colon I \to \overline{\mathbb{R}}$ such that $f$ is upper
semicontinuous, $g$ is lower semicontinuous, and
$$
\mathcal{R}=\Big\{(x,y) \in \mathbb{R}^2\,\,|\, \, x \in I \,\,\,\,
{\rm and}\,\,\,\, f(x) < y< g(x)\Big\}.
$$
\end{itemize}

Here, $\overline{\mathbb{R}}:=
\mathbb{R}\cup\{-\infty,+\infty\}$ is endowed with the standard topology.
Notice that $\mathcal{R}$ is a type I set if and only if there exists
a convex polygon $\mathcal{P} \subset \mathbb{R}^2$ and an interval
$I \subseteq \mathbb{R}$ such that $\mathcal{R}=\mathcal{P} \cap
\left\{(x,y) \in \mathbb{R}^2\mid x \in I\right\}$.

\begin{figure}[!ht]
  \centering
  \includegraphics[height=5.5cm]{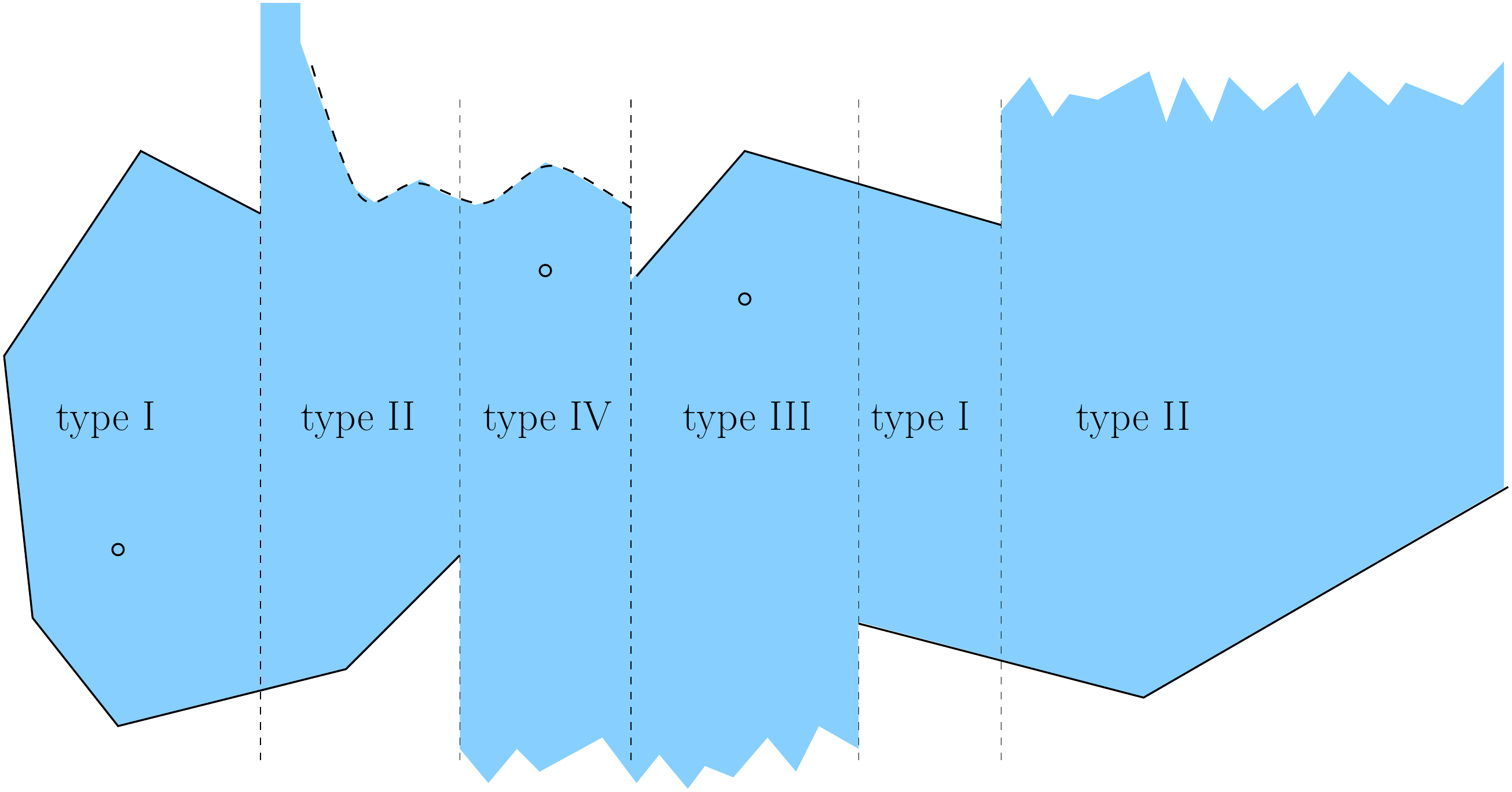}
  \caption{A cartographic projection of $F$. It is a symplectic
    invariant of $F$, see \cite[Theorem C]{PeRaVN2017}.}
  \label{fig:cartographic}
\end{figure}

An example of
a possible cartographic projection is given in Figure
\ref{fig:cartographic} (taken from \cite{PeRaVN2017} with permission). The cartographic
projection contains the information given by the singular affine
structure induced by the singular Lagrangian fibration $\textbf{F} \colon M
\to \mathbb{R}^2$ on the base $\textbf{F}(M)$.
For its construction and the study of its properties we refer to
\cite[Theorems B, C, and Corollary 4.3]{PeRaVN2017}. In
\cite[Theorem D]{PeRaVN2017}, examples of proper semitoric systems
are produced for which the cartographic projection is neither polygonal,
nor convex. More precisely, it is shown, by construction, that there
are uncountably many proper semitoric systems having the range of
$F$ unbounded and the cartographic projections not convex, neither
open nor closed in $\mathbb{R}^2$, and containing every type of the
four possible sets listed above in the union
forming the cartographic projection. In addition, one can build
two uncountable subfamilies such that the cartographic projections
for the first family are bounded and those for the second family are
unbounded. One can even construct semitoric systems such that they
are isomorphic if and only if their parameter indices coincide.  
(See \cite[Section 7]{PeRaVN2017}.)

\subsection{An affine black hole and non-convex $S^2$}  \hfill

During a long time, we thought that any 
locally convex singular affine structure on $S^2$ 
must be automatically globally convex. 
So the following non-convexity result came to us as a big surprise:

\begin{theorem} \label{thm:NonConvexS2}
There exists a non-degenerate singular Lagrangian torus fibration on a
symplectic manifold diffeomorphic to K3, with only focus-focus
singularities, and whose base space is a sphere $S^2$ with a non-convex
singular integral affine structure.
\end{theorem}

The theorem is proved by constructing an explicit example of an integral
affine structure on $S^2$ with focus points, which is, of course, locally
convex, but which is globally non-convex. This example is an instance
of the phenomenon \textit{monodromy can kill global convexity}, where
the  \textit{local-global convexity principle} no longer holds.

Look at the ``8-vertex shuriken'' drawn in Figure \ref{fig:Shuriken},
together with a standard integral lattice in $\mathbb{R}^2$. We glue
the edges of this shuriken together, by the arrows shown in Figure
\ref{fig:Shuriken}. For example:

$O_1Q_8$ is glued to $O_1P_1$ by the linear transformation
which admits $O_1$ as the origin and is given by the matrix
$\begin{pmatrix} 1 & 2 \\  0  & 1\end{pmatrix}$.
Indeed, $\displaystyle \overrightarrow{O_1Q_8} =
\begin{pmatrix} -6 \\ 3\end{pmatrix}, \overrightarrow{O_1P_1} =
\begin{pmatrix} 0\\ 3\end{pmatrix}$, and
$ \begin{pmatrix} 1 & 2 \\ 0 & 1\end{pmatrix} \begin{pmatrix} -6 \\
3\end{pmatrix} = \begin{pmatrix} 0 \\ 3\end{pmatrix}$.

\begin{figure}[!ht]
\centering
\includegraphics[width=0.8 \textwidth]{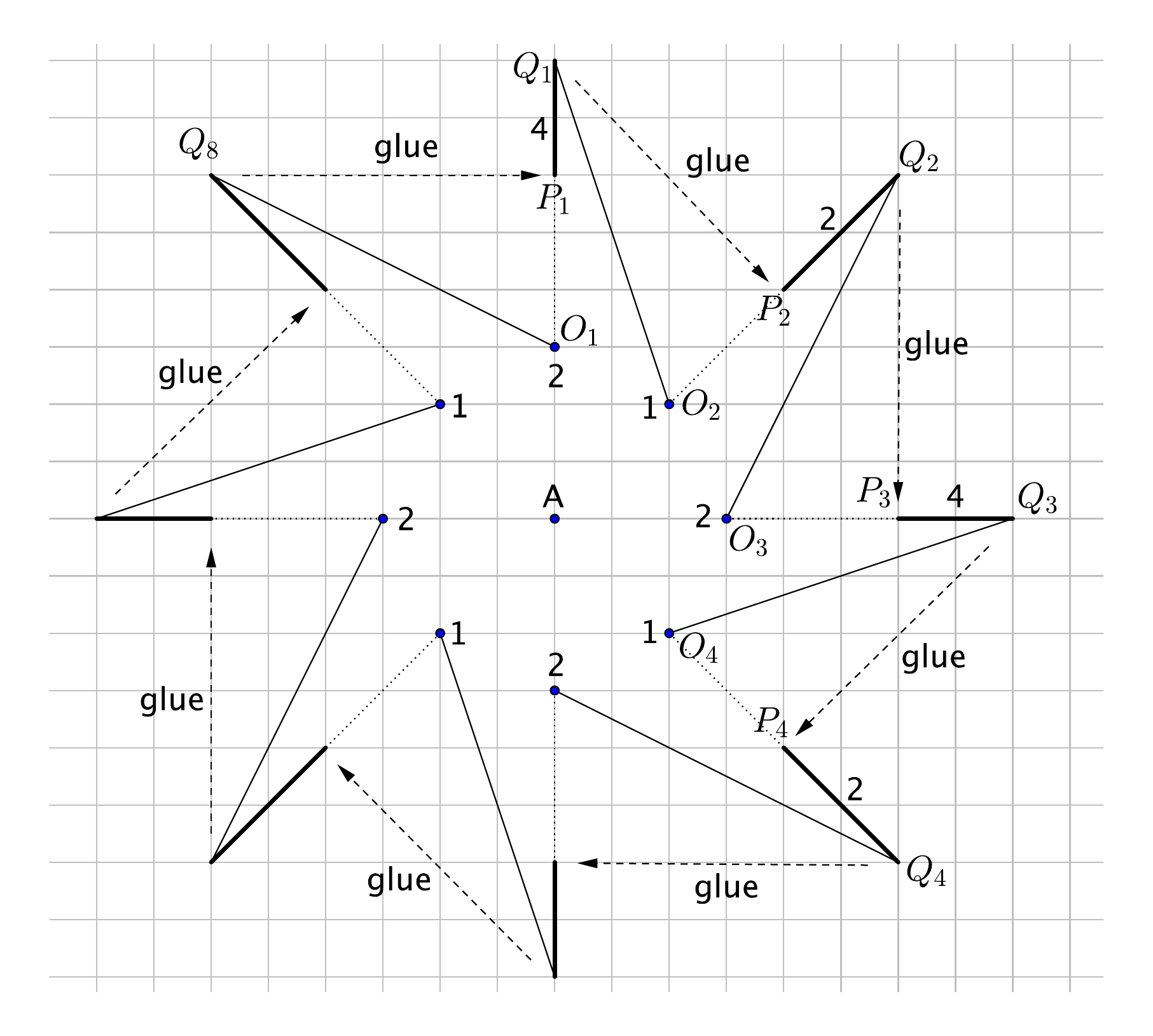}
\caption{The ``shuriken''.} \label{fig:Shuriken}
\end{figure}

$O_2Q_1$ is glued to $O_2P_2$ by the linear transformation which
admits $O_2$ as the origin and is given by the matrix
$$ \begin{pmatrix} 2 & 1 \\ - 1 & 0\end{pmatrix} =
\begin{pmatrix} 1 & 0 \\ - 1 & 0\end{pmatrix}
\begin{pmatrix} 1 & 1 \\ 0 & 1\end{pmatrix}
\begin{pmatrix} 1 & 0 \\ - 1 & 1\end{pmatrix}^{-1}$$
Indeed, $\displaystyle \overrightarrow{O_2Q_1} = \begin{pmatrix} -2 \\
6\end{pmatrix}, \overrightarrow{O_2P_2} = \begin{pmatrix} 2 \\
2\end{pmatrix}$, and
$ \begin{pmatrix} 2 & 1 \\ - 1 & 0\end{pmatrix} \begin{pmatrix} -2 \\
6\end{pmatrix} = \begin{pmatrix} 2 \\ 2\end{pmatrix}$.

After this glueing process, we get an ``8-petal flower'', as shown
in Figure \ref{fig:Flower}.
Its boundary consists of the straight lines
$P_1P_2, P_2P_3,\hdots, P_8P_1$, and it is concave at the vertices $P_i$.
(Note that $Q_i$ is identified with $P_{i+1}$, and $P_9 = P_1$ by our
convention). The flower has 8 focus points $O_1, O_2,\hdots, O_8$
inside ($O_1,O_3,O_5,O_7$ have index 2, and $O_2,O_4,O_6,O_8$ have
index 1). Each curve $AO_iP_iP_{i+1}$ (consisting of a dashed part
and a boundary part) is in fact a straight line ``coming from the left''
with respect to the singular affine structure on the flower, and each petal
$AO_iP_iP_{i+1}O_{i+1}$ is affinely isomorphic to a triangle with vertices
$A_i, P_{i+1}, O_{i+1}$ ($P_9 = P_1$ and $O_9 = O_1$ by convention).
(The fact that each petal is a triangle is more clearly shown on the
shuriken picture).

\begin{figure}[!ht]
\centering
\includegraphics[width=0.6 \textwidth]{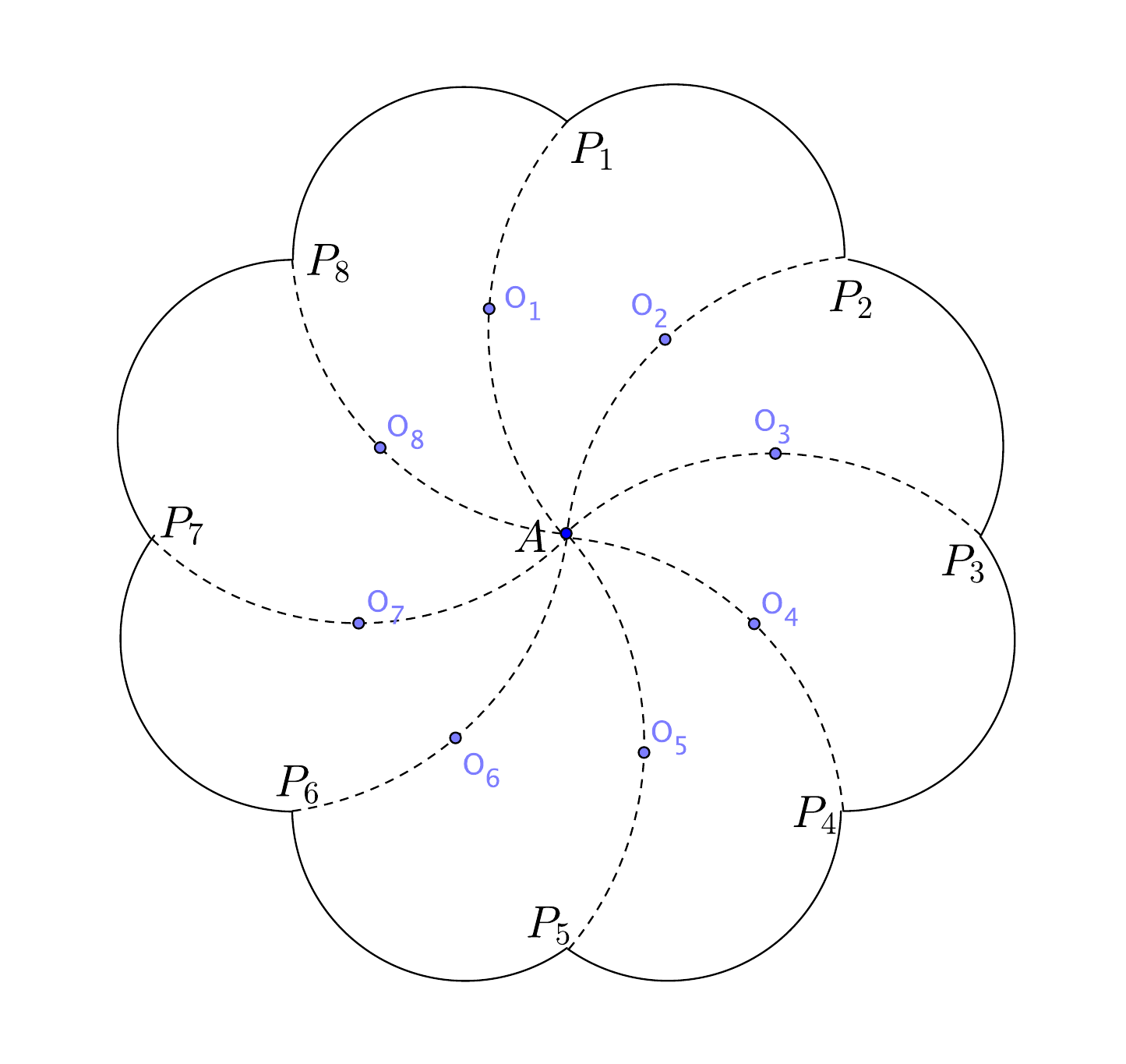}
\caption{The ``shuriken'' becomes an 8-petal flower after glueing. (The dashed lines
are straight lines with respect to the affine structure.)}
\label{fig:Flower}

\end{figure}

The boundary of our 8-petal flower
$P_1P_2\hdots P_8P_1$ is a piecewise
straight simple closed curve  with the following integral affine
invariants:

- Every edge $P_iP_{i+1}$ has {\bfi integral direction}, which means that
there is a local integral affine function $F_i$ which vanishes on
$P_iP_{i+1}$.

- Each vertex $P_i$ is a {\bfi simple vertex}, which means that the
covectors $dF_{i-i}(P_i)$, $dF_{i}(P_i)$
form a basis of the lattice of integral covectors at $P_i$.

- Each edge has {\bfi integral affine length} equal to 2.

- Each edge $P_iP_{i+1}$ of the boundary has its
{\bfi characteristic number} defined as follows. We
can embed a small neighborhood of the 3-piece broken line
$P_{i-1}P_iP_{i+1}P_{i+2}$ by an integral affine map $\Phi_i$ into the
standard integral affine plane $\mathbb{R}^2$, such that under this map
$\Phi(P_i) = (0,0)$, $\Phi(P_{i+1}) = (\ell,0)$ (where $\ell =2$
is the integral affine length of $P_iP_{i+1}$), $\Phi(P_{i-1})$ lies on the
axis $(0,y)$, and $\Phi_{i+2}$ lies on the line $x = cy + \ell$. Then $c$
is called the {\bfi characteristic number} of the  edge $P_iP_{i+1}$
(with respect to the two adjacent edges). The edges $P_1P_2, P_3P_4,
P_5P_6, P_7P_8$ have index 4, while the edges $P_2P_3, P_4P_5, P_6P_7, P_8P_1$ have index 2, as shown on Figure \ref{fig:Shuriken}.

Since our 8-petal flower has concave boundary,
we can glue an appropriate
octagon with convex boundary to it to obtain a sphere $S^2$ with a
singular  integral affine structure with focus points. Such
an explicit octagon is given on Figure \ref{fig:Octagon},
with exactly the same
integral lattice as in Figure \ref{fig:Shuriken}.

We take a flat octagon, then create 8 focus points inside it
(4 focus points of index 1 and 4 focus points of
index 2) by cutting out 8 small triangles and glueing the edges of
the created angles together, as shown on Figure \ref{fig:Octagon}.

\begin{figure}[!ht]
\centering
\includegraphics[width=0.6 \textwidth]{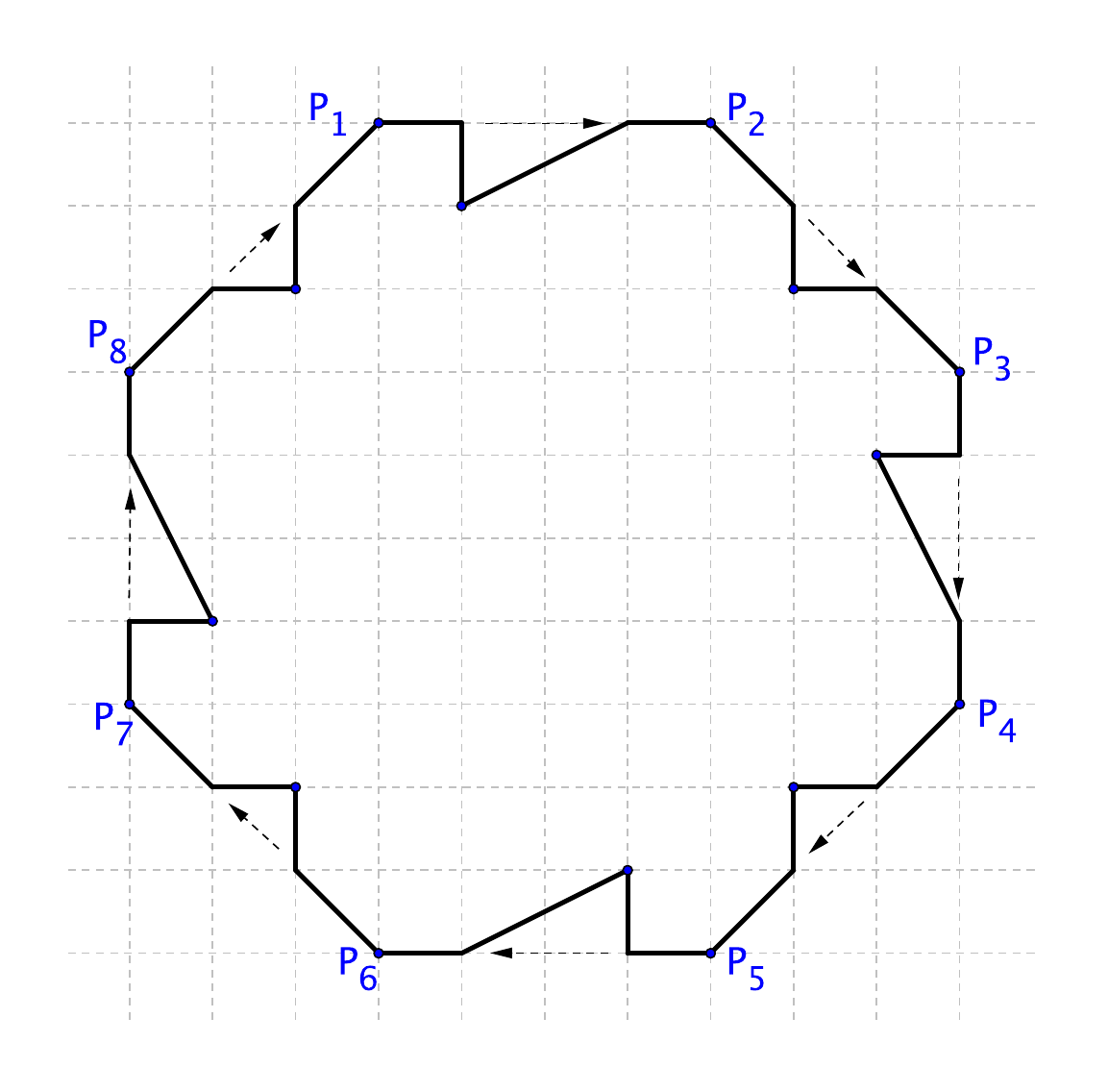}
\caption{Glue by the arrows to get the complementary octagon.}
\label{fig:Octagon}
\end{figure}

It is easy to check that the boundary of our octagon has exactly the same
integral affine invariants as the boundary of our flower. So the octagon can be
glued to the flower edge by edge in a integral affine way to obtain a sphere
$S^2$ with an integral affine structure on it. By integrable surgery (see
\cite{Zung-Integrable2003}), $S^2$  is the base space of a singular Lagrangian
torus fibration  with 8 simple and 8 double focus-focus fibers (for a total of
24 singular focus-focus points in the symplectic manifold, and it is well-known
that such a manifold is diffeomorphic to a complex K3 surface; see
\cite{LeuSym-AlmostToric2010}). We now show that this singular affine $S^2$ is
\textit{not convex}.

Indeed, take any straight line going from the center $A$ of the flower.
One can see that this straight line is trapped inside the flower, can
never get out of it, and so for any point $B$ lying in the interior
of the complementary octagon, there is no straight line going from
$A$ to $B$. One may imagine the flower as a kind of ``black hole''
in which the ``light rays'' are bent so much by the ``affine
gravity'' (i.e., monodromy) of the focus points that no light
ray from $A$ can escape it.

For example, let's say that we take a ray (i.e., straight line) starting
from $A$ and lying between $AO_1$ and $AO_2$. Then it must get out of the
petal $AO_1P_1P_2O_2$ by the edge $O_2P_2$. When it gets out of the first
petal, it enters the second petal $AO_2P_2P_3O_3$ by the edge $AO_2P_2P_3$.
Hence it must get out of this petal by one of the edges
$AO_3$ and $O_3P_3$ to enter the third petal. But $AO_3$ and $O_3P_3$
lie on the same edge $AO_3P_3P_4$ of the petal $AO_3P_3P_4O_4$, and so it
must get out of the third petal by $AO_4$ or $O_4P_4$, to go into the next
petal, and so on. By induction, the ray from $A$ will never get out of the
flower.

Theorem \ref{thm:NonConvexS2} is proved.

\begin{remark}{\rm
It is also easy to see that, points near $A$ in the proof of the above theorem have a similar fate as $A$: any affine "light ray" passing near $A$ will forever remain in a
small neighborhood of the "black hole flower". Indeed, let $\gamma$
be a straight ray starting from a point $A'$ near $A$. If $\gamma$
hits one of the dashed lines in Figure \ref{fig:Flower}, say $AO_1P_1$, "from the left", then by induction one sees that $\gamma$ will also hit all the other consecutive dashed lines $AO_2P_2, AO_3P_3$ and so on from the left and will forever remain in the flower. The only way for $\gamma$ to escape the flower is to hit the dashed lines only from the right (whenever it hits them) and to hit a boundary point, say $R$ on the boundary piece $P_1P_2$. But then $\gamma$ must be "nearly parallel" to the straight line segment
$AO_1P_1P_2$, which implies that $\gamma$ will hit the boundary piece
$P_2P_3$ from outside soon after $R$, and after that, 
by the same inductive arguments as in the previous case, $\gamma$ will forever stay in the flower. In fact, any affine ray which enters the flower from outside will stay in the flower and will never get out again. 
}
\end{remark}

\begin{remark}{\rm
Since any function on $S^2$ has critical points,
there is no integrable Hamiltonian system (with a global momentum map)
with only focus-focus singularities and whose base space is $S^2$.
However, if we take the direct product of our non-convex integral
affine $S^2$ with a closed integral affine interval $D^1$ (with the product affine structure), and view 
this direct product as the base space of the direct product of some toric-focus singular Lagrangian fibrations over $S^2$ and $D^1$ respectively, then by composing the
projection map from the product symplectic manifold to $S^2 \times D^1$ with any embedding from $S^2 \times D^1$ into
$\mathbb{R}^3$, we get the momentum map of a
3-degrees-of-freedom integrable Hamiltonian system
whose base space is our product $S^2 \times D^1$, 
which is non-convex.
}
\end{remark}

\subsection{A globally convex $S^2$ example}  \hfill

This subsection is
devoted to the proof of the following theorem.

\begin{theorem} \label{thm:ConvexS2}
There exists a non-degenerate singular Lagrangian torus fibration on a
symplectic manifold diffeomorphic to K3, with only focus-focus
singularities, and whose base space is a sphere $S^2$ with a globally
convex singular integral affine structure.
\end{theorem}

The original example constructed by Zung in 1993
(see \cite{Zung-Integrable1993,Zung-Integrable2003})
of a singular Lagrangian torus fibration with only focus-focus
singularities on a symplectic 4-manifold diffeomorphic to K3,
whose base space is the  sphere $S^2$ equipped with a singular affine
structure, is as follows:

\begin{figure}[!ht]
\centering
\includegraphics[width=\textwidth]{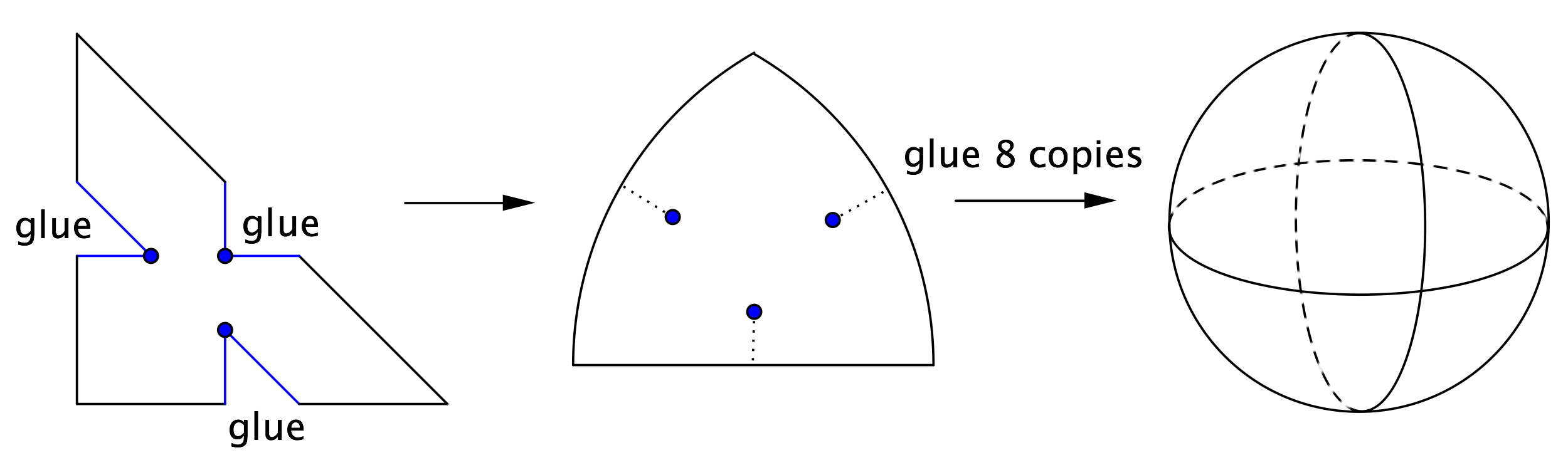}
\vspace{-15pt}
\caption{Construction of $S^2$ with 24 focus-focus points.} \label{fig:ConvexS2a}
\end{figure}

Take a standard triangle $\{ (x,y)  \in \mathbb{R}^2\; | \; x \geq 0,
y \geq 0, x+y \leq c\}$ in the Euclidean space $\mathbb{R}^2$ with the
standard integral affine structure. Cut out from it three small
triangles homothetic to it, one on each edge, as shown on the left
in Figure \ref{fig:ConvexS2a}, and then glue the edges of the created
angles together to obtain a ``rounded'' singular affine triangle with
3 focus points of index 1 inside, as shown in the middle of
Figure \ref{fig:ConvexS2a}. We can now take 8 copies of this
``rounded out'' singular affine triangle, glue them together to get
an affine sphere $S^2$ with 24 singular points, which is the base space of
a singular Lagrangian torus fibration with 24 simple focus-focus
singular fibers on a 4-dimensional symplectic manifold, which therefore
is diffeomorphic to a complex K3 surface; see \cite{LeuSym-AlmostToric2010}
for details.

We can simplify the construction by pushing the focus points
to the edges of the ``rounded''
affine triangles. By doing so, the 24 focus points of
index 1 on $S^2$ become 12 focus points
of index 2, each lying on one quarter of one
of the three ``great circles'' on $S^2$, as shown in
Figure \ref{fig:ConvexS2b}. Of course, this new example of
an affine $S^2$ with 12 double
focus points is also the base space of a singular Lagrangian
torus fibration with 12 double focus-focus fibers.  Each 1/8 of the
sphere (rounded triangle) in this example is affinely
isomorphic to the standard triangle $\{ (x,y)  \in \mathbb{R}^2\; |
\; x \geq 0, y \geq 0, x+y \leq c\}$ in $\mathbb{R}^2$ (because the
singular points have been pushed to the boundary).
Let us show that this simplified example is convex.

Notice that in each triangle which is 1/8 of our sphere
we have 3 trapezoids ``X'', ``Y'', ``Z'', and any
two of them already cover the whole triangle.
If we take all the trapezoids ``X'' in all the eight triangles,
then they form a strip, which is an annulus whose boundary
consists of two affine circles (there are regular affine circles
in the interior of the annulus which tend to these boundary
circles).  The same with trapezoids ``Y'' and trapezoids ``Z''.
So we have three locally convex annuli,
and  any two of these annuli already cover the whole $S^2$.

\begin{figure}[!ht]
\centering
\vspace{-25pt}
\includegraphics[width=\textwidth]{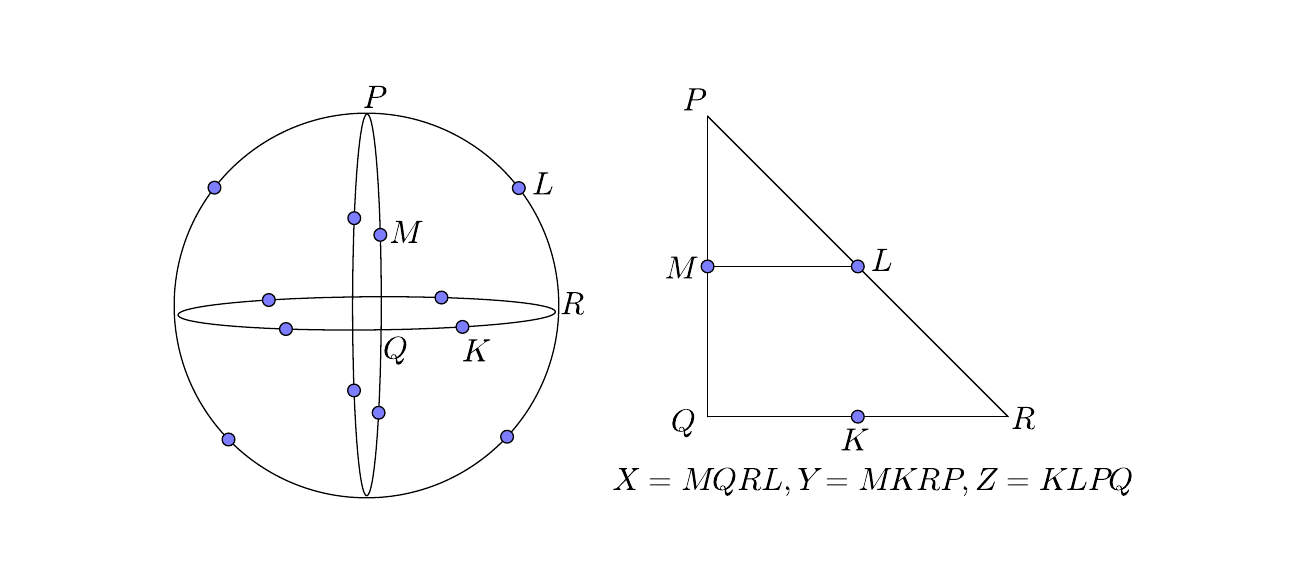}
\vspace{-25pt}
\caption{A convex affine structure on $S^2$} \label{fig:ConvexS2b}
\end{figure}

If $A$ and $B$ are two arbitrary points on our $S^2$,
then at least one of these three annuli contains both $A$ and $B$.
From Theorem \ref{thm_gobal_convex_2D_compact} we know that each annulus
is convex. So we can connect $A$ to $B$ by a straight line lying in
one of the annuli on our $S^2$. Thus our $S^2$ is convex. \hfill $\square$

\subsection{Convexity of toric-focus base spaces in higher dimensions} \hfill

We have seen in Subsection \ref{subsection:NonConvexBox} that there exist
non-convex compact toric-focus base space $\cB$ of any dimension $\geq 4$
with a $\text{focus}^2$ point, due to the ``all choices go wrong''
phenomenon near such points. In fact, already in dimension 3, this
phenomenon can already happen, with just two curves of focus points,
and so we get the following result.

\begin{theorem} \label{thm:NonConvexDim3}
There exists a toric-focus integrable system on a compact 6-dimensional
symplectic manifold, with elliptic and focus-focus singularities, whose
base space is not convex.
\end{theorem}

 \begin{proof} We construct a 3-dimensional box
 \begin{equation*}
 \cB = \{x\; |\; -1 \leq F(x) \leq 1;
 -1 \leq (G_1)_l (x), (G_1)_r(x) \leq 1; -1 \leq (G_1)_l (x), (G_1)_r(x) \leq 1\},
 \end{equation*}
 where:

\begin{itemize}
\item There is a smooth coordinate system $(F,H_1,H_2)$ on $\cB$, and
two curves $\cS_1 = \{x \in \cB \; |\; F(x) = H_1(x) =0\}$ $\cS_1 =
\{x \in \cB \; |\; F(x) = \delta, H_2(x) =0\}$ lying on the disks
$\{F = 0\}$ and $\{F = \delta\}$, respectively.

\item  $G_1$ is a function of $\cB$ with two branches $(G_1)_l$,
$(G_1)_r$ satisfying $(G_1)_r = (G_1)_l$ when $F \geq 0$ and
$(G_1)_r= (G_1)_l + F$ when $F \leq 0$.

\item  $G_2$ is a function of $\cB$ with two branches $(G_2)_l$,
$(G_2)_r$ satisfying $(G_1)_r = (G_1)_l$ when $F \leq \delta$
and $(G_1)_r= (G_1)_l + (F-\delta)$ when $F \geq 0$.

\item $(G_1)_l$ is smooth outside of the set $\{F=0, H_1 \geq 0\}$;
$(G_1)_r$ is smooth outside of the set $\{F=0, H_1 \leq 0\}$.

\item $(G_2)_l$ is smooth outside of the set $\{F=\delta, H_1 \geq 0\}$;
$(G_2)_r$ is smooth outside of the set $\{F=\delta, H_1 \leq 0\}$.

\item $(F,G_1,G_2)$ is a multi-valued integral affine coordinate
system for the singular integral affine structure on $\cB$ with
two curves of focus points $\cS_1$ and $\cS_2$.

\item $G_1$ is a smooth parametrization for $\cS_2$ and
$G_2$ is a smooth parametrization for $\cS_1$.

\item The restriction of $G_i$ to $\cS_i$ is a smooth function
for each $i = 0,1$.
  \end{itemize}

Using integrable surgery, one can construct an integrable system
with such a box $\cB$ as the base space in which $(F,G_1,G_2)$ is a
multi-valued system of action coordinates and $\cS_1$, $\cS_2$
are the two curves of focus points. For each $i=1,2$, the restriction
of $G_i$ to $S_i$ can be chosen to be an arbitrary
smooth function with absolute value smaller than or equal to 1.

\begin{figure}[!ht]
\centering
\vspace{-15pt}
\includegraphics[width=\textwidth]{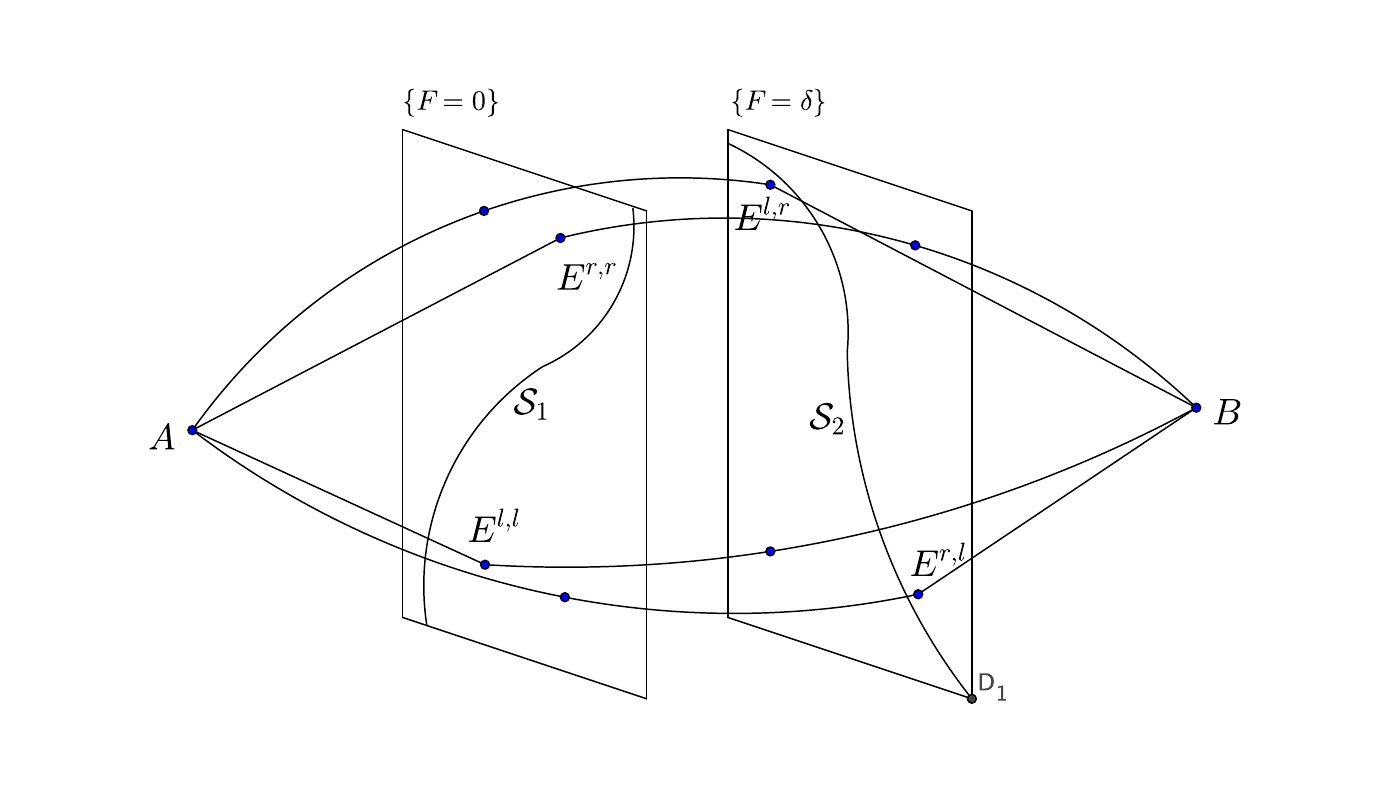}

\vspace{-15pt}
\caption{All four potential straight lines are broken at
respective points
$E^{l,l}, E^{r,l}, E^{l,r}, E^{r,r}$.} \label{fig:4wrongDim3}
\end{figure}

To go from a point $A$ with $F(A) = -1$ to a point $B$ with
$F(B) = 1$ in $\cB$, we have to go through the planes $\{F= 0\}$
and $\{F= \delta\}$ containing focus curves. For each intersecting
plane we have two potential choices:
either go ``on the left'' (i.e., in the domain with $H_i \leq 0$)
or ``on the right'' (i.e., in the domain with $H_i \geq 0$) of the
focus curve. So, in total we have 4 potential choices.  However,
similarly to the case considered in Theorem \ref{thm:NonConvexFocus2},
we can choose our data in such a way that all the 4 potential choices
are wrong, so there is no straight line going from $A$ to $B$,
as illustrated on Figure \ref{fig:4wrongDim3}.
\end{proof}

Theorem \ref{thm:NonConvexDim3} can, of course, be extended to higher
dimensions,  simply by taking symplectic direct products.
Notice that in all non-convex Theorems
\ref{thm:NonConvexDim3}, \ref{thm:NonConvexS2},
and \ref{thm:NonConvexFocus2}, the monodromy group
is big (its image in $GL(n,\mathbb{Z})$ has at least 2 generators).
So one may say that, in all these cases, it is big monodromy which makes non-convexity possible. 
The following theorems essentially states that,
in any dimension $n$, 
if there are $n-1$ independent global affine functions on $\mathcal{B}$
(so the monodromy group cannot be too complicated), then
we always have convexity (under the compactness or properness condition).

\begin{theorem} \label{thm:ConvexSmallMonodromy}
Let $\cB$ be the base space of a toric-focus integrable Hamiltonian
system with $n$ degrees of freedom on a connected compact symplectic
manifold $M$. Assume that the system admits a
global Hamiltonian $\mathbb{T}^{n-1}$-action. Then $\cB$ is convex.
\end{theorem}

\begin{proof}
Let $\mathbf{L}= (L_1 ,\ldots , L_{n-1}):M \rightarrow \mathbb{R}^{n-1}$
be the momentum map of the $\mathbb{T}^{n-1}$-action.
By the Atiyah-Guillemin-Sternberg theorem \cite{At1982, GuSt-Convexity1982},
$\mathbf{L}(M)$ is
a convex polytope in $\mathbb{R}^{n-1}$ and the preimage in $M$ by
$\mathbf{L}$
of each point is connected. Because the torus action preserves the
integrable system, $\mathbf{L}$ descends to a map $\mathcal{B} \rightarrow
\mathbb{R}^{n-1}$, which we still call $\mathbf{L}$, and the preimage
of each point by $\mathbf{L}$ in $\mathcal{B}$ is still connected.

Let $x, y \in \mathcal{B}$ be arbitrary
distinct points; we need to show that there is a straight line connecting
$x$ to $y$. If $\mathbf{L}(x) =\mathbf{L}(y) = c \in \mathbb{R}^{n-1}$,
then $x, y \in\mathbf{L}^{-1}(c) \subset \mathcal{B}$ which is connected
and is a straight line, so we are done.

Consider the case when $\mathbf{L}(x) = c_1$, $\mathbf{L}(y) = c_2$,
with $c_1 \neq c_2$. Let $\ell$ be the intersection of the straight
line passing through $c_1$ and $c_2$ with the polytope $\mathbf{L}(M)$.
Then $\ell$ is a closed interval. Let $P = \mathbf{L}^{-1}(\ell) \subset
\mathcal{B}$. Then $P$ contains $x $ and $y$, it inherits a singular
affine structure from $\mathcal{B}$, it is compact and locally convex
(since $\mathcal{B}$ is compact and locally convex). Note that
$P$ is connected since the interval $\ell$ is connected, the preimage
of every point in $\ell$ by $\mathbf{L}$ is connected, and $P$ is compact.

The singular points
in $\mathcal{B}$ are still of focus type, except for the fact their
indices may not be integers and the affine structure on $\mathcal{B}$
may not be integral. However, on $\mathcal{B}$ we have a global
affine function (one of the $L_i$), so by Remark \ref{global_affine_remark}, the conclusion
of Theorem \ref{thm_gobal_convex_2D_compact} still holds, which means that $P$ is convex and hence
there is a straight line from $x$ to $y$ in $P \subset \mathcal{B}$.
\end{proof}

\begin{definition}
\label{def_general_proper}
The singular affine manifold $\mathcal{B}$ is called {\bfi proper} if for any
closed, connected, simply connected subset $S \subset \mathcal{B}$, the local
injective map $\varphi: S \rightarrow \mathbb{R}^n$ given by an $n$-tuple of
independent affine functions is a proper map.
\end{definition}

\begin{theorem}
\label{thm:GlobalConvexnDnoncompact}
Let $\mathcal{B}$ be the $n$-dimensional base space of a toric-focus
integrable Hamiltonian system on a connected, non-compact,
symplectic, $2n$-manifold without boundary.
Assume that:
\begin{itemize}
\item[{\rm (i)}] The system admits a
global Hamiltonian $\mathbb{T}^{n-1}$-action with momentum map
$\mathbf{L}: M \rightarrow \mathbb{R}^{n-1} $;
\item[{\rm (ii)}] the set of focus points in $\mathcal{B}$ is
compact;
\item[{\rm (iii)}] the interior of  $\mathcal{B}$ is homeomorphic to an open
ball in $\mathbb{R}^n$;
\item[{\rm (iv)}] $\mathcal{B}$ is proper (see Definition \ref{def_general_proper}).
\end{itemize}
Then $\mathcal{B}$ is convex (in its own underlying affine structure).
\end{theorem}

\begin{proof}
This is done by reducing the statement to Theorem
\ref{thm:ConvexSmallMonodromy} exactly in the same way
Theorem \ref{thm:GlobalConvex2Dfirst} was proved as a consequence
of a cutting construction and
Theorem \ref{thm_gobal_convex_2D_compact}.
\end{proof}

\section*{Acknowledgement}

We would like to thank many colleagues, in particular Rui Loja Fernandes,
Eugene Lerman, Alvaro Pelayo, and San V\~u Ng\d{o}c, for useful
discussions on the topics of this paper.

This paper was written during  Nguyen Tien Zung's stay at the School of
Mathematical Sciences, Shanghai Jiao Tong University, as a visiting
professor. He would like to thank Shanghai Jiao Tong University, and
especially Tudor Ratiu, Jianshu Li, and Jie Hu for the invitation,
hospitality and excellent working conditions.

\end{document}